\documentclass{article}

\usepackage{microtype}
\usepackage{graphicx}
\usepackage{booktabs} 

\usepackage{hyperref}

\graphicspath{
  {images_reduced/},{prebuiltimages/}
}

\usepackage{amsmath,amsthm,amssymb,dsfont} 

\theoremstyle{plain}
\newtheorem{theorem}{Theorem}

\theoremstyle{plain}
\newtheorem{proposition}{Proposition}

\theoremstyle{plain}
\newtheorem{corollary}{Corollary}

\theoremstyle{plain}
\newtheorem{lemma}{Lemma}

\theoremstyle{definition}
\newtheorem{definition}{Definition}

\theoremstyle{definition}
\newtheorem{assumption}{Assumption}

\theoremstyle{definition}
\newtheorem{remark}{Remark}

\usepackage[accepted]{icml2019}

\icmltitlerunning{Optimal Mini-Batch and Step Sizes for SAGA}

\usepackage{shortcuts}

\usepackage[textwidth=2.8cm, textsize=scriptsize,disable]{todonotes} 

\usepackage{enumitem} 
\usepackage{graphics}

\usepackage{subcaption}

\usepackage{hyperref}

\usepackage{algorithm}
\usepackage{algorithmic}
\usepackage[titlenumbered,ruled,noend,algo2e]{algorithm2e}

\SetCommentSty{mycommfont}
\SetEndCharOfAlgoLine{}

\usepackage[capitalise]{cleveref}

\usepackage{thmtools}
\usepackage{thm-restate}

\begin{document}

\twocolumn[

\icmltitle{Optimal Mini-Batch and Step Sizes for SAGA}



\icmlsetsymbol{equal}{*}

\begin{icmlauthorlist}
\icmlauthor{Nidham Gazagnadou}{tpt}
\icmlauthor{Robert M.~Gower}{tpt}
\icmlauthor{Joseph Salmon}{um}
\end{icmlauthorlist}

\icmlaffiliation{tpt}{LTCI, T\'el\'ecom ParisTech, Universit\'e Paris-Saclay, Paris, France}
\icmlaffiliation{um}{IMAG, Univ Montpellier, CNRS, Montpellier, France}

\icmlcorrespondingauthor{Nidham Gazagnadou}{nidham.gazagnadou@telecom-paristech.fr}
\icmlcorrespondingauthor{Robert M.~Gower}{robert.gower@telecom-paristech.fr}

\icmlkeywords{Machine Learning, ICML}

\vskip 0.3in
]



\printAffiliationsAndNotice{}  

\begin{abstract}
%
%
%
%
Recently it has been shown that the step sizes of a family of variance reduced gradient methods called the JacSketch methods depend on the expected smoothness constant.
In particular, if this expected smoothness constant could be calculated a priori, then one could safely set much larger step sizes which would result in a much faster convergence rate.
We fill in this gap, and provide simple closed form expressions for the expected smoothness constant and careful numerical experiments verifying these bounds.
Using these bounds, and since the SAGA algorithm is part of this JacSketch family, we suggest a new standard practice for setting the step and mini-batch sizes for SAGA that are competitive with a numerical grid search.
Furthermore, we can now show that the total complexity of the SAGA algorithm decreases linearly in the mini-batch size up to a pre-defined value: the optimal mini-batch size.
This is a rare result in the stochastic variance reduced literature, only previously shown for the Katyusha algorithm.
Finally we conjecture that this is the case for many other stochastic variance reduced methods and that our bounds and analysis of the expected smoothness constant is key to extending these results.
\end{abstract}


\section{Introduction}
\label{sec:introduction}



Consider the empirical risk minimization (ERM) problem:
\begin{equation} \label{eq:ERM_pb}
    w^* \in \argmin_{w \in \R^d} \bigg(f(w) := \frac{1}{n} \sum_{i=1}^n f_i (w)\bigg)\enspace,
\end{equation}
where each $f_i$ is $L_i$-smooth and $f$ is $\mu$-strongly convex. Each $f_i$ represents a regularized loss over a sampled data point.
Solving the ERM problem is often time consuming for large number of samples $n$, so much so that algorithms scanning through all the data points at each iteration are not competitive.
Gradient descent (GD) falls into this category, and in practice its stochastic version is preferred.

Stochastic gradient descent (SGD), on the other hand, allows to solve the ERM incrementally by computing at each iteration an unbiased estimate of the full gradient,
$\nabla f_i (w^k)$ for $i$ randomly sampled in $[n] \eqdef \{1,2,\ldots,n\}$~\cite{RobbinsMonro:1951}.
On the downside, for SGD to converge one needs to tune a sequence of asymptotically vanishing step sizes, a cumbersome and time-consuming task for the user.
Recent works have taken advantage of the sum structure in~\cref{eq:ERM_pb} to design stochastic variance reduced gradient algorithms~\cite{Johnson2013, SDCA, SAGA_Nips,  SAG}.
In the strongly convex setting, these methods lead to fast linear convergence instead of the slow $\cO (1/t)$ rate of SGD. Moreover, they only require a constant step size, informed by theory, instead of sequence of decreasing step sizes.

In practice, most variance reduced methods rely on a mini-batching strategy for better performance.
Yet most convergence analysis (with the Katyusha algorithm of~\citet{katyusha} being an exception) indicates that a mini-batch size of $b=1$ gives the best overall complexity, disagreeing with practical findings, where larger mini-batch often gives better results.
Here, we show both theoretically and numerically that $b=1$ is not the optimal mini-batch size for the SAGA algorithm~\cite{SAGA_Nips}.

%
Our analysis leverages recent results in~\cite{gower2018stochastic}, where the authors prove that the iteration complexity and the step size of SAGA, and a larger family of methods called the JacSketch methods, depend on an expected smoothness constant.
This constant governs the trade-off between the increased cost of an iteration as the mini-batch size is increased, and the decreased total complexity.
Thus if this expected smoothness constant could be calculated a priori, then we could set the optimal mini-batch size and step size.
We provide simple formulas for computing the expected smoothness constant when sampling mini-batches without replacement, and use them to calculate optimal mini-batches and significantly larger step sizes for SAGA.

%

In particular, we provide two bounds on the expected smoothness constant, each resulting in a particular step size formula.
We first derive the \emph{simple bound} and then develop a matrix concentration inquality to obtain the refined \emph{Bernstein bound}.
We also provide substantial theoretical motivation and numerical evidence for practical estimate of the expected smoothness constant.
For illustration, we plot in Figure~\ref{fig:intro_step_size_example2} the evolution of each resulting step size as the mini-batch size grows on a classification problem (\Cref{sec:numerical_study} has more details on our experimental settings).

\begin{figure}[t]
    \vskip 0.2in
    \begin{center}
        \begin{subfigure}[b]{\columnwidth}
            \includegraphics[width=0.9\textwidth]{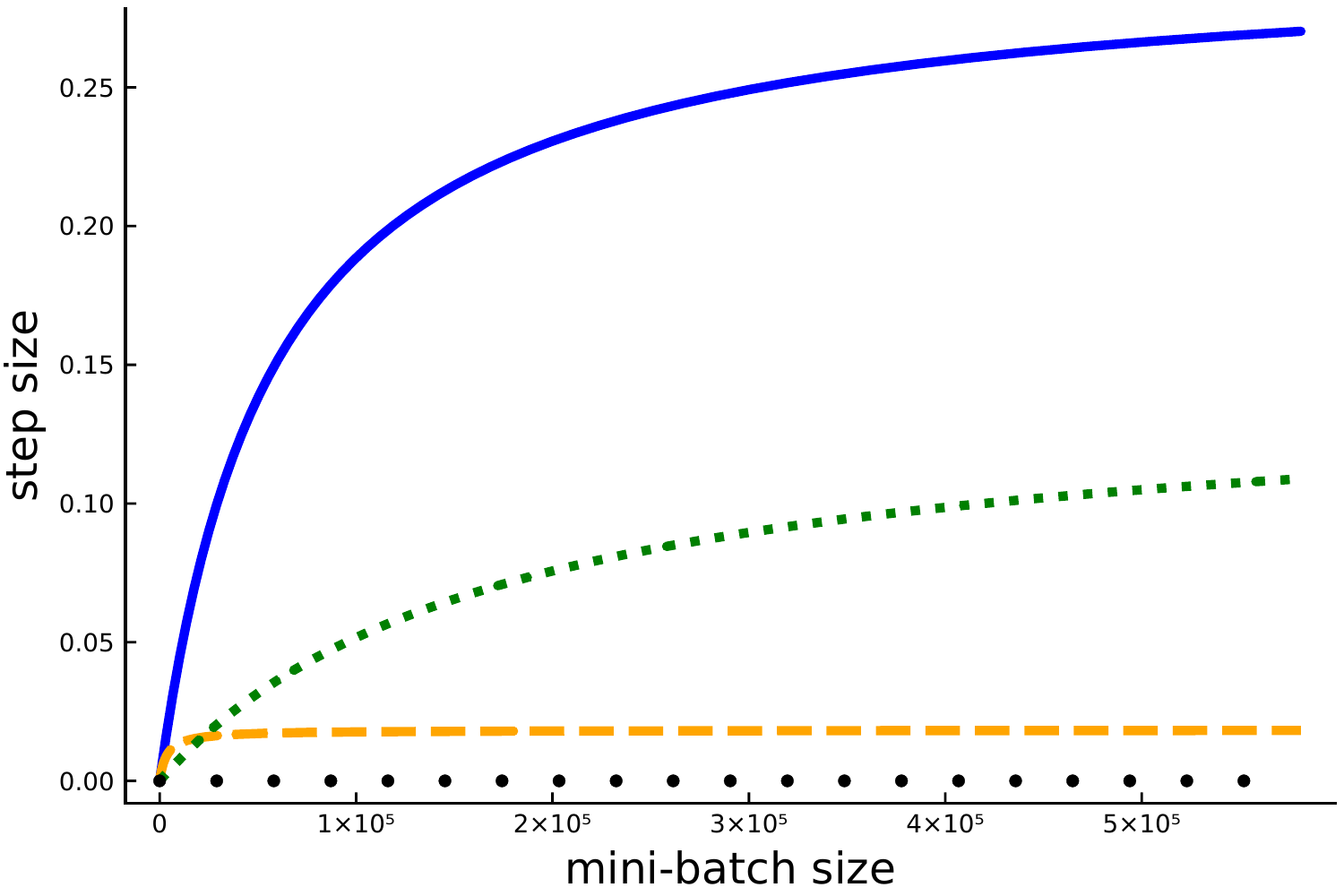}
        \end{subfigure}\\
        \begin{subfigure}[b]{\columnwidth}
            \centerline{\includegraphics[trim={0 16mm 0 10mm},clip,width=\textwidth]{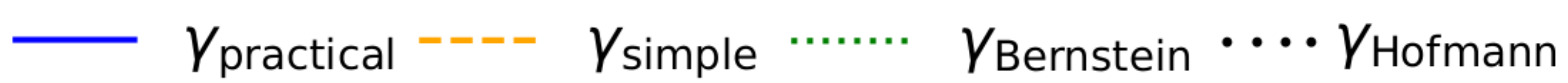}}
        \end{subfigure}
    \caption{Step size as a function of the mini-batch size for a regularized ($\lambda = 10^{-3}$) logistic regression problem applied to the feature-scaled \textit{covtype.binary} dataset from LIBSVM.}
    \label{fig:intro_step_size_example2}
    \end{center}
    \vskip -0.2in
\end{figure}

Furthermore, our bounds provide new insight into the \emph{total complexity}, denoted $K_{\mathrm{total}}$ hereafter, of SAGA.
For example, when using our \emph{simple bound} we show for regularized generalized linear models (GLM), with $\lambda >0$ as in \cref{eq:ERM_expended}, that $K_{\mathrm{total}}$ is piecewise linear in the mini-batch size $b$:
\begin{align}
    K_{\mathrm{total}}(b) = \max
      \bigg\{ \bigg. &n \frac{b -1}{n-1} \frac{4\Lbarconst}{\mu} + \frac{n -b}{n-1} \frac{4L_{\max}}{\mu}+\frac{4b \lambda}{\mu },\nonumber\\
      &n + \frac{n-b}{n-1}\frac{4(L_{\max}+\lambda)}{\mu}
      \bigg. \bigg\} \log \left( \frac{1}{\epsilon} \right)\enspace, \nonumber
\end{align}
with $\Lmax \eqdef \max_{i\in[n]} L_i$, $\Lbar \eqdef \frac{1}{n} \sum_{i=1}^n L_i$ and $\epsilon >0$ is the desired precision.
\begin{figure}[ht]
    \vskip 0.2in
    \begin{center}
    \centerline{\includegraphics[width=\columnwidth]{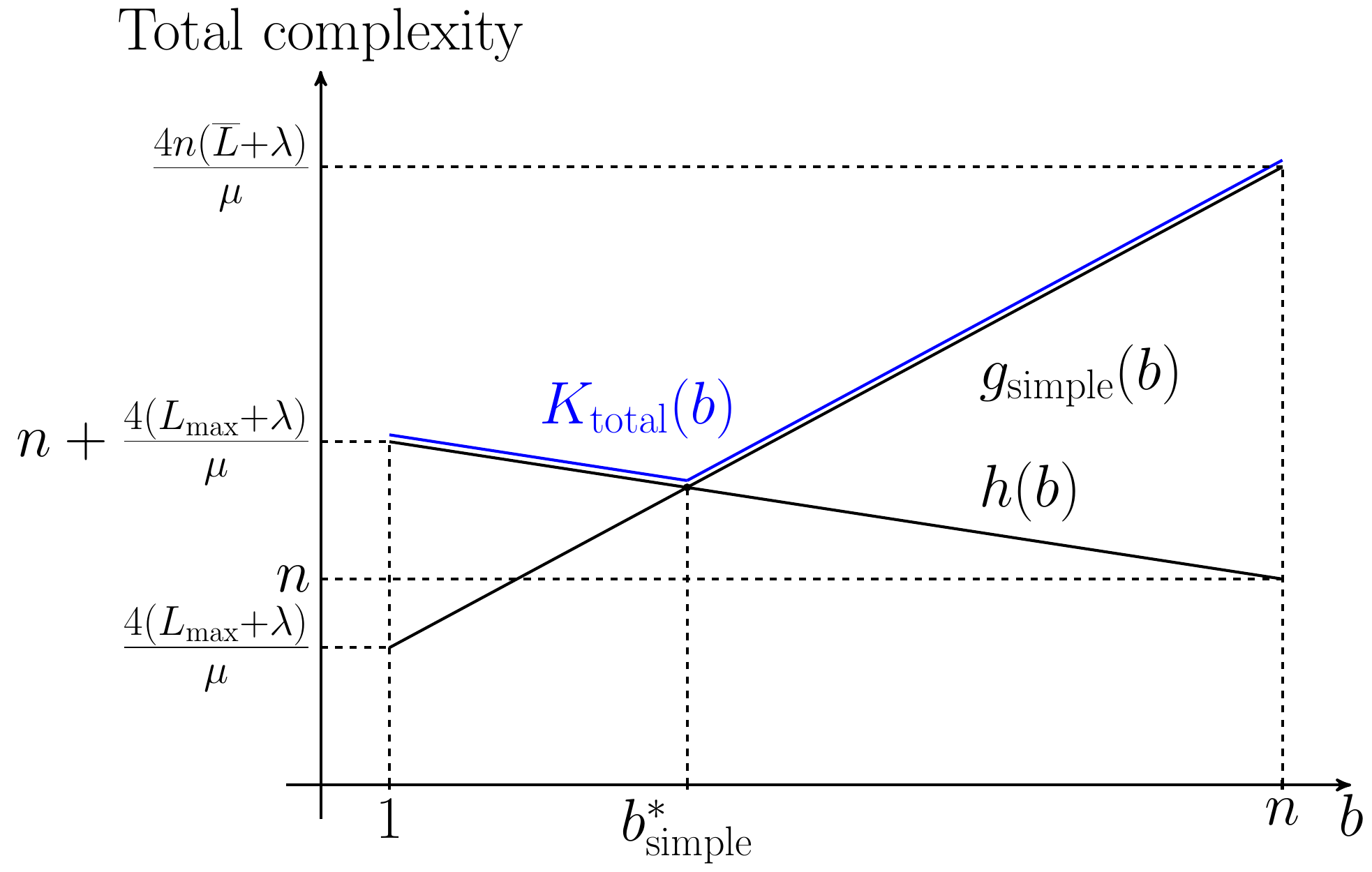}}
    \caption{Optimal mini-batch size $b_{\text{simple}}^*$ for the \emph{simple bound}, where $K_{\mathrm{total}}(b) =\max\{g_{\simple}(b), h(b)\} $.}
    \label{fig:optimal_mini_batch_diagram}
    \end{center}
    \vskip -0.2in
\end{figure}
This complexity bound, and others presented in~\cref{sub:upper_bounds_of_expsmoothness} show that SAGA enjoys a linear speedup as we increase the mini-batch size until an optimal one (as illustrated in \Cref{fig:optimal_mini_batch_diagram}).
After this point, the total complexity increases.
We use this observation to develop optimal and practical mini-batch sizes and step sizes.

The rest of the paper is structured as follows.
In \cref{sec:background_and_related_work} we first
introduce variance reduction techniques after presenting our main assumption, the expected smoothnes assumption.
We highlight how this assumption is necessary to capture the improvement in iteration complexity, and conclude the section by showing that to calculate the expected smoothness constant we need evaluate an intractable expectation. Which brings us to~\cref{sec:bounding_expsmoothness} where we directly address this issue and provide several tractable upper-bounds of the expected smoothness constant.
We then calculate optimal mini-batch sizes and step sizes by using our new bounds.
Finally, we give numerical experiments in \Cref{sec:numerical_study} that verify our theory on artificial and real datasets.
We also show how these new settings for the mini-batch size and step size lead to practical performance gains.

\section{Background}
\label{sec:background_and_related_work}

\subsection{Controlled stochastic reformulation and JacSketch}
\label{sub:controlled_stochastic_reformulation_and_jacsketch}
We can introduce variance reduced versions of SGD in a principled manner by using a \emph{sampling vector}.
\begin{definition}\label{def:v} We say that a random vector $v \in \R^n$ with distribution $\cD$ is a sampling vector if
\[\EE{\cD}{v_i} =1\enspace,\quad \mbox{for all }i \in [n]\enspace. \]
\end{definition}
With a sampling vector we can re-write~\eqref{eq:ERM_pb} through the following stochastic reformulation
\begin{equation} \label{eq:subsampled_function}
w^* = \arg\min_{w \in \R^d} \EE{\cD}{f_v(w) \eqdef \frac{1}{n} \sum_{i=1}^n f_i(w) \cdot v_i},
\end{equation}
where $f_v(w)$ is called a \emph{subsampled function}.
The stochastic Problem~\eqref{eq:subsampled_function} and our original Problem~\eqref{eq:ERM_pb} are equivalent :
\[\EE{\cD}{f_v(w)} = \frac{1}{n} \sum_{i=1}^n f_i(w) \cdot \EE{\cD}{v_i} \overset{\text{\scriptsize Definition~\ref{def:v}} }{=} \frac{1}{n} \sum_{i=1}^n f_i(w). \]
Consequently the gradient $\nabla f_v(w)$ is an unbiased estimate of $\nabla f (w)$ and we could use the SGD method to solve~\eqref{eq:subsampled_function}.
To tackle the variance of these stochastic gradients
we can further modify~\eqref{eq:subsampled_function} by introducing \emph{control variates} which leads to the following
 \emph{controlled stochastic reformulation}:
\begin{equation} \label{eq:controlled_reformulation}
    w^* \in \argmin_{w \in \R^d}  \mathbb{E}_{\cD} \big[ f_v (w) - z_v (w) + \EE{\cD}{z_v(w)} \big] \enspace,
\end{equation}
where $z_v (w) \in \R$ are the control variates.
Clearly~\eqref{eq:controlled_reformulation} is also equivalent to~\eqref{eq:ERM_pb} since $- z_v (w) + \EE{\cD}{z_v(w)} $ has zero expectation.
Thus, we can solve~\eqref{eq:controlled_reformulation} using an SGD algorithm where the stochastic gradients are given by
\begin{equation} \label{eq:gradient_estimator_reformulation}
    g_v (w)  \eqdef \nabla f_v (w) - \nabla z_v (w) + \EE{\cD}{\nabla z_v(w)} \enspace.
\end{equation}
That is, starting from a vector $w^0$,  given a positive step size $\gamma$, we can iterate the steps
\begin{equation} \label{eq:gradient_step}
    w^{k+1} = w^{k} - \gamma g_{v^k} (w^k) \enspace,
\end{equation}
where $v^k \sim \cD$ are \iid samples at each iteration.

The JacSketch algorithm introduced by \citet{gower2018stochastic} fits this format~\eqref{eq:gradient_step} and uses a linear control $z_v (w)=\frac{1}{n}\dotprod{J^\top w, v}$, where $J$ is a $d \times n$ matrix of parameters.
This matrix is updated at each iteration so as to 
decrease the variance of the resulting stochastic gradients.
Carefully updating the covariates through $J$ results in a method that has stochastic gradients with decreasing variance, \ie $\lim_{w^k \rightarrow w^*}\mathbb{E} \big[\normin{g_{v^k} (w^k) - \nabla f (w^k)}_2^2 \big] = 0$, which is why JacSketch is a stochastic variance reduced algorithm.
This is also why the user can set a single constant step size a priori instead of tuning a sequence of decreasing ones.
The SAGA algorithm, and all of its mini-batching variants, are instances of the JacSketch method.

\subsection{The expected smoothness constant}
\label{sub:the_expected_smoothness_constant}

In order to analyze stochastic variance reduced methods,
some form of smoothness assumption needs to be made. The most common assumption is
\begin{equation}\label{eq:smoothness}
\norm{\nabla f_i(w) -\nabla f_i(y)} \leq  L_{\max} \norm{w-y},
\end{equation}
for each $i \in [n]$.
That is each $f_i$ is uniformly smooth with smoothness constant $L_{\max}$, as is assumed in
\citep{SAGA_Nips,hofmann2015variance,stichlimited2018} for variants of SAGA\footnote{The same assumption is made in proofs of SVRG \citep{Johnson2013}, S2GD \citep{S2GD} and the SARAH algorithm \citep{nguyen17b}.}. In the analyses of these papers it was shown that the iteration complexity of SAGA is proportional to $L_{\max},$ and the step size is inversely proportional to $L_{\max}.$

%
%
But as was shown in~\citep{gower2018stochastic}, we can set a much larger step size by making use of the  smoothness of the subsampled functions $f_v.$
For this~\citet{gower2018stochastic} introduced the notion of expected smoothness, which we extend here to all sampling vectors and control variates.
\begin{definition}[Expected smoothness constant] \label{def:expected_smoothness}
Consider a sampling vector $v$ with distribution $\cD.$ We say that the expected smoothness assumption holds with constant $\cL$ if for every $w \in \R^d$ we have that
\begin{equation} \label{eq:expected_smoothness}
    \ED{\norm{ \nabla f_v (w) - \nabla f_v (w^*) }_2^2} \leq 2 \cL (f(w) - f(w^*)) \enspace.
\end{equation}
\end{definition}
\begin{remark} \label{rem:expected_smoothness}
    Note that we refer to any positive constant $\cL$ that satisfies~\eqref{eq:expected_smoothness} as an expected smoothness constant. Indeed $\cL \rightarrow \infty$ is a valid constant in the extended reals, but as we will see, the smaller $\cL$, the better for our complexity results.
\end{remark}

\citet{gower2018stochastic} show that
the expected smoothness constant plays the same role that $L_{\max}$ does in the previously existing analysis of SAGA, namely that the step size is inversely proportional to $\cL$ and the iteration complexity is proportional to $\cL$ (see details in Theorem~\ref{theo:conv}).
Furthermore, by assuming that $f$ is $L$--smooth,
 the expected smoothness constant is bounded
\begin{align}\label{eq:cLbound}
 L \quad \leq \quad \cL \quad \leq \quad L_{\max} \enspace,
\end{align}
as was proven in Theorem 4.17 in~\cite{gower2018stochastic}. Also, the bounds $L_{\max}$ and $L$ are attained when using a uniform single element sampling and a full batch, respectively. And as we will show, the constants $L_{\max}$ and $L$ can be orders of magnitude apart on large dimensional problems. Thus we could set much larger step sizes for larger mini-batch sizes if we could calculate $\cL$.
Though calculating $\cL$ is not easy, as we see in the next lemma.
\begin{lemma} \label{lemma:master_lemma}
    Let $v$ be an unbiased sampling vector. Suppose that $f_v (w) = \frac{1}{n} \sum_{i=1}^n f_i(w)v_i$  is $L_v$-smooth and each $f_i$ is convex for $i=1,\ldots, n$.
    It follows that the expected smoothness constant holds with $\cL = \max_{i \in [n]}\E{ L_v v_i}$.
\end{lemma}
\begin{proof}
The proof is given in \Cref{proof:master_lemma}.
\end{proof}
If the sampling has a very large combinatorial number of possible realizations --- for instance sampling mini-batches without replacement --- then this expectation becomes intractable to calculate.
This observation motivates the development of functional upper-bounds of the expected smoothness constant that can be efficiently evaluated.

\subsection{Mini-batch without replacement: $b$--nice sampling}
\label{sub:samplings}


Now we will choose a distribution of the sampling vector $v$ based on a mini-batch sampling without replacement.
We denote a mini-batch as $B \subseteq [n]$ and its size as $b= |B|$.

\begin{definition}[\emph{$b$-nice sampling}] \label{def:bnice}
 $S$ is a $b$-nice sampling if $S$ is a set valued map with a probability distribution given by
    \begin{equation*}
        \Prb{S = B} = \frac{1}{\binom{n}{b}}, \quad \forall B \subseteq [n] \ \ \text{\st} \ \ |B| = b\enspace.
    \end{equation*}
\end{definition}
We can construct a sampling vector based on a $b$-nice sampling by setting $v = \tfrac{n}{b} \sum_{i\in S} e_i$, where $e_1, \ldots, e_n$ is the canonical basis of $\R^n$.
Indeed, $v$ is a sampling vector according to \Cref{def:v} since for every $i \in [n]$ we have
 \begin{equation}\label{eq:vi}
 v_i = \bigg( \frac{n}{b} \sum_{j \in S} e_j \bigg)_i
 = \frac{n}{b}  \mathds{1}_{S}(i)\enspace,
 \end{equation}
 where $ \mathds{1}_{S}$ denotes the indicator function of the random set $S$.
 Now taking expectation in~\eqref{eq:vi} gives
\begin{align}
\EE{}{v_i} &= \frac{n}{b}\frac{1}{\binom{n}{b}}\sum_{B \subseteq [n] \,:\, |B| = b}   \mathds{1}_{B}(i) \nonumber
= \frac{n}{b\binom{n}{b}} \binom{n-1}{b-1}
\,=\,  1,
\end{align}
using $|\{{{B \subseteq [n]: |B| = b \wedge i \in B}}\}|=\binom{n-1}{b-1}$.

Here we are interested in the mini-batch SAGA algorithm with $b$-nice sampling, which we refer to as the $b$-nice SAGA. In particular, $b$-nice SAGA is the result of using $b$-nice sampling, together with a linear model for the control variate $z_v(w)$.
Different choices of the control variate $z_v(w)$ also recover popular algorithms such as gradient descent, SGD or the standard SAGA method (see \cref{tab:recovered_algorithms_parameters} for some examples).

A naive implementation of $b$-nice SAGA based on the JacSketch algorithm is given in \cref{alg:JacSketch_SAGA}\footnote{We also provide a more efficient implementation that we used for our experiments in the appendix in \cref{alg:practical_SAGA_implementation}.}.


{\fontsize{4}{4}\selectfont
\begin{algorithm}[tb]
\SetKwInOut{Input}{Input}
\SetKwInOut{Init}{Initialize}
\caption{\textsc{JacSketch version of $b$-nice SAGA}}
\label{alg:JacSketch_SAGA}
\Input{mini-batch size $b$, step size $\gamma > 0$}
\Init{$w^0 \in \R^d$, $J^0 \in \R^{d \times n}$}
\For{$k=0, 1, 2, \ldots$}{

    Sample a fresh batch $B \subseteq [n]$ s.t. $|B| = b$

    $g^k = \frac{1}{n} J^k e + \frac{1}{b}\sum_{i \in B} (\nabla f_i (w^k) - J_{:i}^k)$

    \tcp*[r]{update the gradient estimate}

    $J^{k+1}_{:i} = \begin{cases} J^k_{:i}, & \quad \text{if } i\notin B\\
                                  \nabla f_i(w^k), & \quad \text{if } i\in B.\end{cases}$ 


    \tcp*[r]{update the Jacobian estimate}

    $w^{k+1} = w^{k} - \gamma g^k$ \tcp*[r]{take a step}
}
\Return{$w^k$}
\end{algorithm}
}

%


\section{Upper Bounds on the Expected Smoothness}
\label{sec:bounding_expsmoothness}
To determine an optimal mini-batch size $b^*$ for $b$-nice SAGA, we first  state our assumptions and provide bounds of the smoothness of the subsampled function.
We then define $b^*$ as the mini-batch size that minimizes the total complexity of the considered algorithm, \ie the total number of stochastic gradients computed.
Finally we provide upper-bounds on the expected smoothness constant $\cL$, through which we  can deduce optimal mini-batch sizes.
Many proofs are deferred to the supplementary material.

\begin{table}[t]
    \caption{Algorithms covered by JacSketch and corresponding sampling vector $v$ and control variates $z_v$.}
    \label{tab:recovered_algorithms_parameters}
    \vskip 0.15in
    \begin{center}
    \begin{scriptsize}
    \begin{sc}
    \begin{tabular*}{\columnwidth}{l@{\extracolsep{28pt}}c@{\extracolsep{40pt}}c}
        \toprule
        Parameters & $v$ & $\nabla z_v(w)$\\
        \midrule
        GD & $e=e_1+\dots+e_n$ & $\nabla f_i (w)$ \\
        SGD & $ne_i, \;i \sim \frac{1}{n}$ & 0\\
        SAGA & $ne_i, \;i \sim \frac{1}{n}$ & $J_{:i}$\\
        $b$-nice SAGA & $(n/b) \sum_{i\in S} e_i$ & $(1/b)\sum_{i \in S}J_{:i}$\\
        \bottomrule
    \end{tabular*}
    \end{sc}
    \end{scriptsize}
    \end{center}
    \vskip -0.1in
\end{table}


\subsection{Assumptions and notation}
\label{sub:assumptions_and_smoothness_constants}

We consider that the objective function is a GLM with quadratic regularization controlled by a parameter $\lambda>0$:
\begin{equation} \label{eq:ERM_expended}
    w^* \in \argmin_{w \in \R^d} f(w) = \frac{1}{n} \sum_{i=1}^n \phi_i(a_i^\top w) + \frac{\lambda}{2} \norm{w}_2^2\enspace,
\end{equation}
with $\norm{\cdot}_2$ is the Euclidean norm, $\phi_1, \ldots, \phi_n$ are convex functions and $a_1, \ldots, a_n$ a sequence of observations in $\R^d$.
This framework covers regularized logistic regression by setting $\phi_i (z) = \log (1+\exp (-y_iz))$ for some binary labels $y_i, \ldots, y_n$ in $\{ \pm 1\}$,  ridge regression if $\phi_i (z) = (z-y_i)^2/2$ for real observations $y_i, \ldots, y_n$, and conditional random fields for when the $y_i$'s are structured outputs.

We assume that the second derivative of each $\phi_i$ is uniformly bounded, which holds for our aforementioned examples.
\begin{assumption}[Bounded second derivatives] \label{as:bounded_second_derivatives}
There exists $U  \geq 0$ such that $\phi_i''(x) \leq U, \forall x \in \R, \forall i \in [n]$.
\end{assumption}

For a batch $B \subseteq [n]$, we rewrite the subsampled function as
\begin{equation*}
    f_B (w) \eqdef \frac{1}{|B|} \sum_{i \in B} \phi_i (a_i^\top w) +\frac{\lambda}{2} \norm{w}_2^2 \enspace,
\end{equation*}
and its second derivative is thus given by
\begin{equation} \label{eq:bound_AB}
    \nabla^2 f_B(w) = \frac{1}{|B|} \sum_{i \in B} \phi_i'' (a_i^\top w) a_i a_i^\top + \lambda I_d \enspace,
\end{equation}
where $I_d$ denotes the identity matrix of size $d$.

 For a symmetric matrix $M$, we write $\lambda_{\max}(M)$ (resp. $\lambda_{\min}(M)$) for its largest (resp. smallest) eigenvalue.
Assumption~\ref{as:bounded_second_derivatives} directly implies the following.
\begin{lemma}[Subsample smoothness constant] \label{lemma:LB_def_1}
    Let $B\subset[n]$, and let
$A_B = [a_i]_{i\in B}$ denote the column concatenation of the vectors $a_i$ with $i\in B.$
     The smoothness constant of the subsampled loss function $\frac{1}{|B|} \sum_{i \in B} \phi_i ( a_i^\top w)$ is given by
    \begin{equation} \label{eq:LBdef}
        L_B \hspace{-1pt} \eqdef \hspace{-1pt} \frac{U}{|B|}\lambda_{\max} \hspace{-1pt}\left(\sum_{i \in B} a_i a_i^\top \hspace{-2pt}\right) \hspace{-2pt} = \hspace{-1pt} \frac{U}{|B|}\lambda_{\max}\left( A_B A_B^\top\right).
    \end{equation}
\end{lemma}
\noindent{\emph{Proof}.}
The proof follows from Assumption~\ref{as:bounded_second_derivatives} as
\begin{align*}
    \pushQED{\qed}
    \frac{1}{|B|} \sum_{i \in B} \nabla^2 \phi_i ( a_i^\top w) &=
    \frac{1}{|B|} \sum_{i \in B}\phi_i'' ( a_i^\top w) a_i a_i^\top \\
    &\preceq \frac{U}{|B|} A_B A_B^\top. \qedhere
    \popQED
\end{align*}
Combined with \eqref{eq:bound_AB}, we get that $f_B$ is $(L_B+\lambda)$-smooth.

Another key quantity in our analysis is the strong convexity parameter.
\begin{definition}
The strong convexity parameter is given by
\[\mu \eqdef \min_{w \in \R^d} \lambda_{\min} \left( \nabla^2 f(w) \right).\]
\end{definition}
Since we have an explicit regularization term with $\lambda >0$, $f$ is strongly convex and $\mu \geq \lambda > 0.$
%

We additionally define $L_i$, resp. $L$, as the smoothness constant of the individual function $\phi_i ( a_i^\top w)$, resp. the whole function $\frac{1}{n} \sum_{i=1}^n \phi_i ( a_i^\top w)$.
We also recall the definitions of the maximum of the individual smoothness constants by $\Lmax \eqdef \max_{i\in[n]} L_i$ and their average by $\Lbar \eqdef \frac{1}{n} \sum_{i=1}^n L_i$.
The three constants satisfies
\begin{equation}\label{eq:LoverlineLLmax}
L \quad \leq \quad \overline{L} \quad \leq \quad L_{\max}.
\end{equation}
The proof of~\eqref{eq:LoverlineLLmax} is given in \Cref{lem:lemma_smoothness_cst_ineq} in the appendix.



\subsection{Path to the optimal mini-batch size}
\label{sub:path_to_the_optimal_mini_batch_size}

Our starting point is the following theorem taken from combining Theorem 3.6 and Eq.~(103) in~\citep{gower2018stochastic}\footnote{Note that $\lambda$ has been added to every smoothness constant since the analysis in~\citet{gower2018stochastic} depends on the $(L+\lambda)$-smoothness of $f$ and the $(L_B+\lambda)$-smoothness of the subsampled functions $f_B$.}.

\begin{theorem}\label{theo:conv}
 Consider the iterates $w^k$ of Algorithm~\ref{alg:JacSketch_SAGA}.
Let the step size be given by
\begin{equation}\label{eq:gammamaster}
\gamma  = \frac{1}{4} \frac{1}{\max \left\{ \cL+\lambda, \displaystyle\frac{1}{b} \frac{n-b}{n-1} \left( L_{\max}+\lambda\right) + \frac{\mu}{4}\frac{n}{b} \right\}}.
\end{equation}
Given an $\epsilon >0$, if $k \geq K_{\mathrm{iter}}(b)$ where
\begin{equation} \label{eq:itercomplex}
K_{\mathrm{iter}}(b)\eqdef \Bigg\{ \Bigg. \hspace{-2pt} \frac{4 (\cL \hspace{-2pt} + \hspace{-2pt} \lambda)}{\mu}, \frac{n}{b} + \frac{n-b}{n-1}\frac{4(L_{\max} \hspace{-2pt} + \hspace{-2pt} \lambda)}{b \mu} \hspace{-2pt} \Bigg. \Bigg\} \hspace{-2pt} \log\left( \frac{1}{\epsilon} \right),
\end{equation}
then
$
\E{\norm{w^k-w^*}^2} \leq \epsilon \, C,
$
where $C>0$ is a constant
\footnote{Specifically, let $J^0 \in \R^{d \times n}$ be the initiated Jacobian of the $b$-nice SAGA Algorithm~\ref{alg:JacSketch_SAGA}. Then this constant is
\[C \eqdef \norm{w^0 -w^*}^2 +  \tfrac{\gamma}{2 L_{\max}}\sum_{i \in [n]}\norm{J^0_{:i} -\nabla f(w^*)}^2.\]

}.
\end{theorem}

Through Theorem~\ref{theo:conv} we can now explicitly see how the expected smoothness constant $\cL$ controls both the step size and the resulting iteration complexity. This is why we need bounds on $\cL$ so that we can set the step size. In particular, we will show that the expected smoothness constant is a function of the mini-batch size $b$. Consequently so is the step size, the iteration complexity and the \emph{total complexity}. We denote $K_{\text{total}}$ the total complexity defined as the number of stochastic gradients computed, hence with~\eqref{eq:itercomplex},
\begin{equation} \label{eq:total_cplx_b_nice}
    \begin{aligned}[b]
        & K_{\mathrm{total}}(b) = b K_{\mathrm{iter}}(b)\\
        &\hspace{-5pt}= \max \hspace{-2pt} \Bigg\{ \Bigg. \hspace{-2pt} \frac{4b (\cL \hspace{-2pt} + \hspace{-2pt} \lambda)}{\mu}, n + \frac{n-b}{n-1}\frac{4(L_{\max} \hspace{-2pt} + \hspace{-2pt} \lambda)}{\mu} \hspace{-2pt} \Bigg. \Bigg\} \hspace{-2pt} \log\left( \frac{1}{\epsilon} \right).
    \end{aligned}
\end{equation}
Once we have determined $\cL$ as a function of $b$, we will calculate the mini-batch size $b^*$ that optimizes the total complexity
$b^* \in \argmin_{b \in [n]} K_{\mathrm{total}}(b)$.

As we have shown in Lemma~\ref{lemma:master_lemma}, computing a precise bound on $\cL$ can be computationally intractable.
This is why we focus on finding upper bounds on $\cL$ that can be computed, but also tight enough to be useful.
To verify that our bounds are sufficiently tight, we will always have in mind the bounds $L \leq \cL \leq L_{\max}$ given in~\eqref{eq:cLbound}.
In particular, after expressing our bounds of $\cL=\cL(b)$ as a function of $b$,we would like the bounds~\eqref{eq:cLbound} to be attained for $\cL(1) = L_{\max}$ and $\cL(n) = L.$

\subsection{Expected smoothness} 
\label{sub:upper_bounds_of_expsmoothness}

All bounds we develop on $\cL$ are based on the following lemma, which is a specialization of~\eqref{lemma:master_lemma} for $b$-nice sampling.
\begin{proposition}[Expected smoothness constant] \label{prop:expsmooth}
    For the $b$-nice sampling, with $b \in [n]$, the expected smoothness constant is given by
    \begin{equation} \label{eq:expsmooth}
        \cL = \frac{1}{\binom{n-1}{b-1}} \max_{i=1,\ldots, n} \Bigg\{\sum_{\substack{B \subseteq [n] \,: |B| = b \wedge i \in B}}  L_B \Bigg\}\enspace.
    \end{equation}
\end{proposition}
\begin{proof}
    Let $S$ the $b$-nice sampling as defined in \Cref{def:bnice} and let $v =  \frac{n}{b} \sum_{j \in S} e_j $ be its corresponding sampling vector.
    Note that
    \[f_v(w) = \frac{1}{n}\sum_{i \in [n]} f_i(w) v_i =
     \frac{1}{b}\sum_{i \in S} f_i(w) = f_S(w).
    \]
    Finally from Lemma~\ref{lemma:master_lemma}, we have that:
    \begin{align*}
        \cL & =  \E{ L_v v_i} \, \overset{\eqref{eq:vi}}{=}\, \mathrm{E} \bigg[L_S \frac{n}{b}  \mathds{1}_{\{i \in S\}}\bigg]\\
              &= \frac{1}{\binom{n}{b}}  \frac{n}{b} \sum_{B \subseteq [n] \,:\, |B| = b} L_B  \mathds{1}_{\{i \in B\}} \\
              &= \frac{1}{\binom{n}{b}} \frac{n}{b} \sum_{\substack{B \subseteq [n] \,:\\ |B| = b \wedge i \in B}} L_B = \frac{1}{\binom{n-1}{b-1}} \sum_{\substack{B \subseteq [n] \,:\\ |B| = b \wedge i \in B}} L_B.
    \end{align*}
    Taking the maximum over all $i\in [n]$ gives the result.
\end{proof}
The first bound we present is  technically the simplest to derive, which is why we refer to it as the \emph{simple bound}.
\begin{theorem}[Simple bound] \label{thm:upper_bound_exp_smooth_simple_combination}

    For a $b$-nice sampling $S$, for $b\in [n]$, we have that
    \begin{equation} \label{eq:cLbndLmacbarLnice}
        \cL \leq \cL_{\simple} (b) \eqdef \frac{n }{b}\frac{b -1}{n-1} \Lbarconst + \frac{1}{b} \frac{n-b}{n-1} L_{\max} \enspace,
    \end{equation}
\end{theorem}
\begin{proof}
The proof, given in \Cref{sub:proof_of_the_simple_bound}, starts by using the that $ L_B \leq \frac{1}{b}\sum_{j\in B} L_j$ for all subsets $B$, which follows from repeatedly applying~\Cref{lem:lemma1} in the appendix.  The remainder of the proof follows by straightforward counting arguments.
\end{proof}

The previous bound interpolates, respectively for $b =1$ and $b = n$, between $L_{\max}$ and $\Lbarconst.$ On the one hand, we have that $\cL_{\simple}(b)$ is a good bound for when $b$ is small, since $\cL_{\simple}(1) = L_{\max}$. Though $\cL_{\simple}(b)$ may not be a good bound for large $b$, since $\cL_{\simple}(n) = \overline{L} \quad {\geq} \quad L$, thanks to \eqref{eq:LoverlineLLmax}.
Thus $\cL_{\simple}(b)$ does not achieve the left-hand side of~\eqref{eq:cLbound}.
Indeed $\overline{L}$ can be far from $L$.
For instance\footnote{We numerically explore such extreme settings in~\Cref{sec:numerical_study}.}, if $f(w) = \frac{1}{n} \sum_{i \in [n]} \frac{1}{2}(a_i^\top w -b_i)^2$ is a quadratic function, then we have that $\overline{L} = \frac{1}{n} \Tr{AA^\top}$ and $L = \frac{1}{n} \lambda_{\max}(AA^\top)$.
Thus if the eigenvalues of $AA^\top$ are all equal then $\overline{L} = d L$. Alternatively, if one eigenvalue is significantly larger than the rest then $\overline{L} \approx L$.

%

Due to this shortcoming of $\cL_{\simple},$ we now derive the \emph{Bernstein bound}. This bound explicitly depends on $L$ instead of $\Lbarconst$, and is developed through a specialized variant of a matrix Bernstein inequality~\citep{Tropp2012, tropp2015introduction} for sampling \emph{without} replacement in~\cref{appendix:matrix_bernstein_inequality}.
\begin{theorem}[Bernstein bound] \label{thm:cL_bound_direct_expectation}
    The expected smoothness constant is upper bounded by
    \begin{equation}
     \label{eq:cL_bound_direct_expectation}
        \cL \leq \cL_{\text{Bernstein}} (b)
        \eqdef 2 \tfrac{b -1}{b}\tfrac{n}{n-1} L \hspace{-1pt} + \hspace{-1pt} \tfrac{1}{b} \hspace{-1pt} \left( \tfrac{n-b}{n-1} \hspace{-2pt} + \hspace{-2pt} \tfrac{4}{3} \log d \hspace{-1pt} \right) \hspace{-2pt} L_{\max}.
\end{equation}
\end{theorem}

Checking again the bounds of $\cL_{\text{Bernstein}} (b) $, we have on the one hand that $\cL_{\text{Bernstein}} (1) = \left( 1 + \tfrac{4}{3} \log d \right) L_{\max} \geq L_{\max},$
 thus there is a little bit of slack for $b$ small.
On the other hand, using $\tfrac{1}{n}L_{\max} \leq L$ (see \Cref{lem:lemma_smoothness_cst_ineq} in appendix),
 we have that
 \[\cL_{\text{Bernstein}} (n) = 2L + \frac{1}{n}\frac{4}{3} \log d \, L_{\max} \leq \left(2 +\frac{4}{3} \log d\right)L, \]
which depends only logarithmically on $d$. Thus we expect the \emph{Bernstein bound} to be more useful in the large $d$ domains, as compared to the \emph{simple bound}.
We confirm this numerically in~\Cref{sub:experiment_1_upper_bounds_of_the_expected_smoothness_constant}.


\begin{remark}
 The simple bound is relatively tight for $b$ small, while the Bernstein bound is better for large $b$ and large $d$. Fortunately,  we can obtain a more refined bound by taking the minimum of the \emph{simple} and the \emph{Bernstein bounds}. This is highlighted numerically in \Cref{sec:numerical_study}.
\end{remark}

Next we propose a \emph{practical estimate} of $\cL$ that is tight for both small and large mini-batch sizes.
\begin{definition}[Practical estimate] \label{def:practical_bound}
    \begin{equation} \label{eq:practical_bound}
        \cL_{\text{practical}} (b) \eqdef \frac{n }{b}\frac{b -1}{n-1} L + \frac{1}{b} \frac{n-b}{n-1} L_{\max} \enspace.
    \end{equation}
\end{definition}
Indeed $\cL_{\text{practical}} (1) = L_{\max}$ and $\cL_{\text{practical}} (n) =L,$ achieving both limits of~\eqref{eq:cLbound}. The downside to $\cL_{\text{practical}} (b)$ is that it is not an upper bound of $\cL$. Rather, we are able to show that $\cL_{\text{practical}} (b)$ is very close to a valid smoothness constant, but it can be slightly smaller. Our theoretical justification for using $\cL_{\text{practical}} (b)$ comes from a mid step in the proof of the Bernstein bound which is captured in the next lemma.
\begin{restatable}{lemma}{practicalmotivate} \label{lem:practicalmotivate}
    Let  $a_j \in \R^d$ for $j \in [n]$ and let $S^i$ be a $(b-1)$-nice sampling over $[n]\setminus\{i\},$ for every $i \in [n]$. It follows that
    \begin{equation}\label{eq:withpracticalmotivate}
      \cL \leq \cL_{\text{practical}} (b)
      +U \max_{i\in[n]}\E{\lambda_{\max} \left( N_i \right)}\enspace,
    \end{equation}
    with $N_i \eqdef \frac{1}{b} \sum_{ j \in S^i} a_j a_j^\top - \frac{1}{b} \frac{b -1}{n-1} \sum_{j\in [n]\setminus\{i\}} a_j a_j^\top$.
\end{restatable}
\begin{proof}
The proof is given in~\Cref{proof:proof_practical}.
\end{proof}
Lemma~\ref{lem:practicalmotivate}  shows that
the expected smoothness constant is upper-bounded by
$\cL_{\text{practical}}(b)$ and an additional term.
In this additional term we have the largest eigenvalue of a random matrix.
This matrix is zero in expectation, and we also find that its eigenvalues oscillate around zero.
Indeed, we provide extensive experiments in~\Cref{sec:numerical_study} confirming that $\cL_{\text{practical}}(b)$ is very close to $\cL$ given in~\eqref{eq:expsmooth}.

%


\label{sec:larger_step_and_mini_batch_sizes}


\section{Optimal Mini-Batch Sizes}
\label{sub:optimal_mini_batch_size}


Now that we have established the \emph{simple} and the \emph{Bernstein bounds}, we can minimize the total complexity~\eqref{eq:total_cplx_b_nice} in the mini-batch size.

For instance for the \emph{simple bound}, given $\epsilon >0$ and plugging in~\eqref{eq:cLbndLmacbarLnice} into~\eqref{eq:total_cplx_b_nice} gives
\begin{align*}
  K_{\mathrm{total}}(b)
  & \leq  \max \left\{ g_{\text{simple}}(b), h(b) \right\} \log \left(\frac{1}{\epsilon}\right)\enspace,
\end{align*}
where
$g_{\text{simple}}(b) \eqdef \tfrac{4\left( n \Lbarconst - L_{\max} +(n-1)\lambda\right)}{\mu (n-1)}b + \tfrac{4 n \left( L_{\max} - \Lbarconst \right)}{\mu (n-1)},$
and
$h(b) \eqdef -\tfrac{4(L_{\max}+\lambda)}{\mu (n-1)}b  + n \left( 1 + \tfrac{1}{n-1}\tfrac{4(L_{\max}+\lambda)}{\mu} \right).$

\begin{remark}
    The right-hand side term $h(b)$ is common to all our bounds since it does not depend on $\cL$. It linearly decreases from $h(1) = n + \frac{4(L_{\max}+\lambda)}{\mu}$ to $h(n) = n$.
\end{remark}

\begin{figure}[h!]
    \vskip 0.2in
    \begin{center}
        \begin{subfigure}[b]{\columnwidth}
            \includegraphics[width=0.85\textwidth]{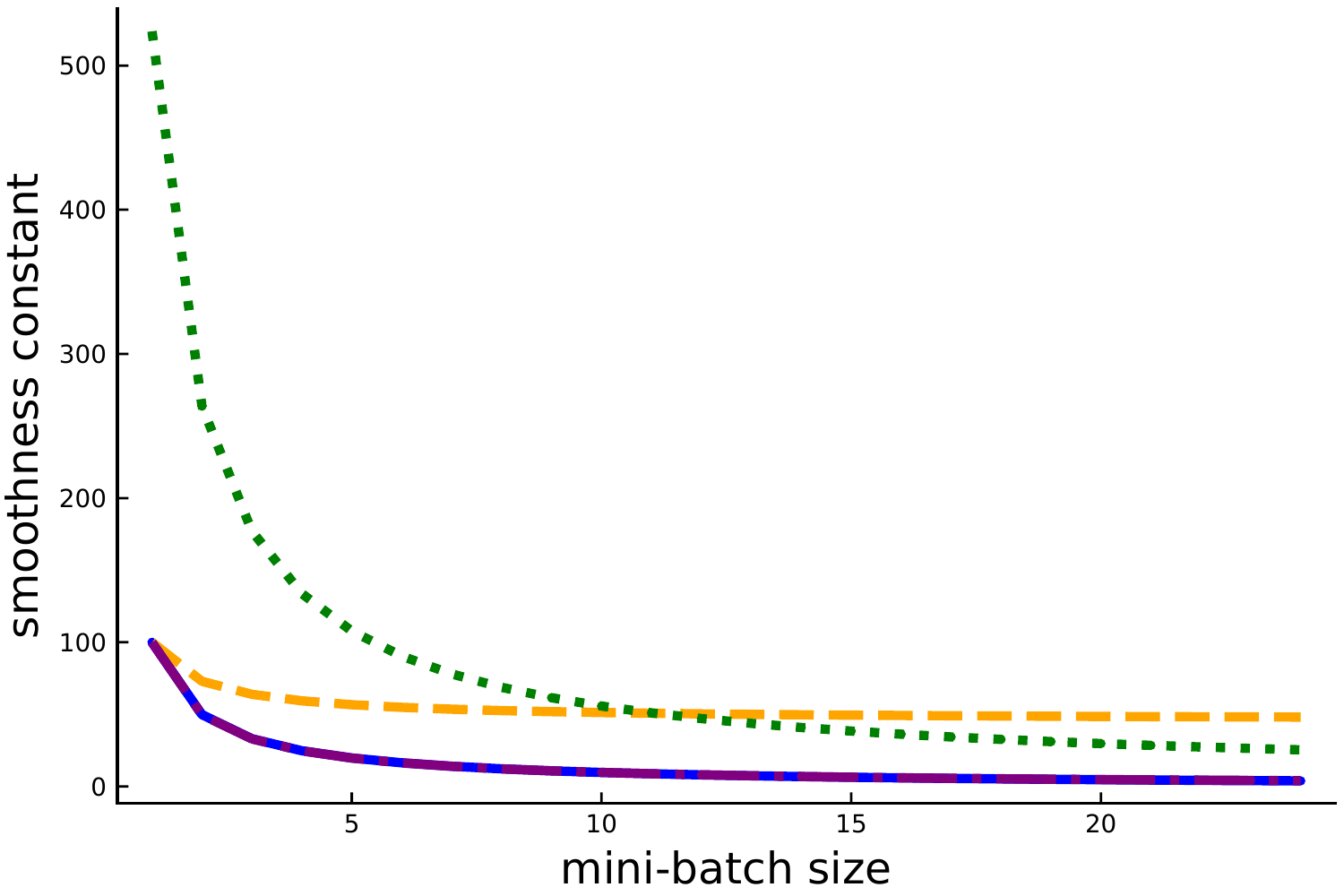}
        \end{subfigure}\\[-0.1cm]
        \begin{subfigure}[b]{\columnwidth}
            \centerline{\includegraphics[width=0.85\textwidth]{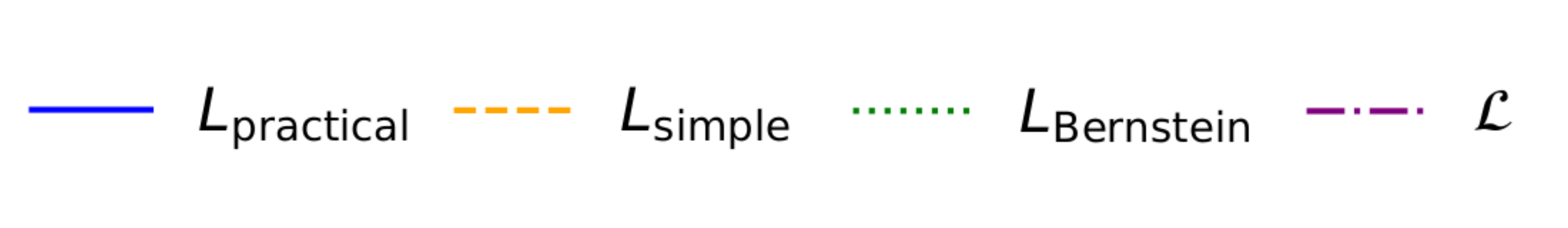}}
        \end{subfigure}
        \caption{Expected smoothness constant $\cL$ and its upper-bounds as a function of the mini-batch size $b$ for ridge regression on the unscaled \emph{staircase eigval} dataset ($\lambda = 10^{-3}$).}
        \label{fig:intro_step_size_example}
    \end{center}
    \vskip -0.2in
\end{figure}
We note that $g_{\text{simple}}(b)$ is a linearly increasing function of $b$, because $L_{\max} \leq n\Lbarconst$ (as proven in~\Cref{lem:lemma_smoothness_cst_ineq}).
One can easily verify that $g_{\text{simple}}(b)$ and $h(b)$ cross, as presented in \Cref{fig:optimal_mini_batch_diagram}, by looking at initial and final values:
\begin{itemize}[leftmargin=*]
    \item At $b\! =\! 1$, $g_{\text{simple}}(1) = \tfrac{4}{\mu}(L_{\max}+\lambda) = h(1) - n$. So, $g_{\text{simple}}(1) \leq h(1)$.
    \item At $b\! =\! n$, $g_{\text{simple}}(n) = \tfrac{4(\Lbarconst+\lambda)}{\mu}n = \tfrac{4(\Lbarconst+\lambda)}{\mu} h(n)$. Since $\Lbarconst \geq \mu$, we get $g_{\text{simple}}(n) \geq h(n)$.
\end{itemize}


Consequently, solving $g_{\text{simple}}(b)=h(b)$ in $b$ gives the optimal mini-batch size
\begin{equation} \label{eq:opt_tau_1st_bound}
    b_{\text{simple}}^* = \left\lfloor 1 + \frac{\mu (n-1)}{4 (\Lbarconst+\lambda)} \right\rfloor\enspace.
\end{equation}
For the \emph{Bernstein bound}, plugging \eqref{eq:cL_bound_direct_expectation} into \eqref{eq:total_cplx_b_nice} leads to
\begin{equation}
K_{\text{total}}(b) \leq  \max \left\{ g_{\text{Bernstein}}(b), h(b) \right\} \log \left(\frac{1}{\epsilon}\right) \enspace,
\end{equation}
where
\begin{align*}
    g_{\text{Bernstein}}(b) &\eqdef \tfrac{4}{\mu (n-1)} \left( 2nL - L_{\max} +(n-1)\lambda\right) b \\
    &\quad + \tfrac{4n}{\mu (n-1)} \left( L_{\max} -2L \right) + \tfrac{16}{3\mu} \log (d) L_{\max} \enspace.
\end{align*}
The function $g_{\text{Bernstein}}$ is also linearly increasing in $b$ and its initial and final values are
\begin{itemize}[leftmargin=*]
    \item At $b\! =\! 1$, $g_{\text{Bernstein}}(1) = (1+\frac{4}{3} \log d) \tfrac{4L_{\max}}{\mu} + \tfrac{4\lambda}{\mu}.$
    \item At $b\! =\! n$, $g_{\text{Bernstein}}(n) \! = \! n \tfrac{4(2L+\lambda)}{\mu} \! +\!  \tfrac{16}{3\mu} {(L_{\max} \! + \! \lambda)}\log (d)$. Since $L \geq \mu$, we get $g_{\text{Bernstein}} (n) \geq h(n)$.
\end{itemize}
Yet, it is unclear whether $g_{\text{Bernstein}} (1)$ is dominated by $h(1)$. This is why we need to distinguish two cases to  minimize the total complexity, which leads to the following solution
\begin{align*}
    &b_{\text{Bernstein}}^* \nonumber\\
    &= \hspace{-4pt}
    \begin{cases}
        \hspace{-2pt} \left\lfloor  \hspace{-2pt} 1 \hspace{-2pt} + \hspace{-2pt} \frac{\mu (n-1)}{4(2 L+\lambda)} \hspace{-2pt} - \hspace{-2pt} \frac{4}{3} \log d \frac{n-1}{n} \frac{L_{\max}}{2L+\lambda} \hspace{-2pt} \right\rfloor\!, & \!\!\!\!\text{if } \frac{4}{3} \frac{4L_{\max}}{\mu} \log d\hspace{-2pt} \leq \hspace{-2pt} n,\\
        \hspace{-2pt}1, & \text{otherwise} \enspace.
    \end{cases}
\end{align*}

In the first case, the problem is well-conditioned and $g_{\text{Bernstein}}$ and $h$ do cross at a mini-batch size between $1$ and $n$.
In the second case, the total complexity $K_{\mathrm{total}}$ is governed by $g_{\text{Bernstein}}$ because $g_{\text{Bernstein}} (b) \geq h(b)$ for all $b\in[n]$, and the resulting optimal mini-batch size is $b =1$.


\section{Numerical Study}
\label{sec:numerical_study}

\begin{figure}[t!]
    \vskip 0.2in
    \begin{center}
        \begin{subfigure}[b]{\columnwidth}
            \includegraphics[width=0.85\textwidth]{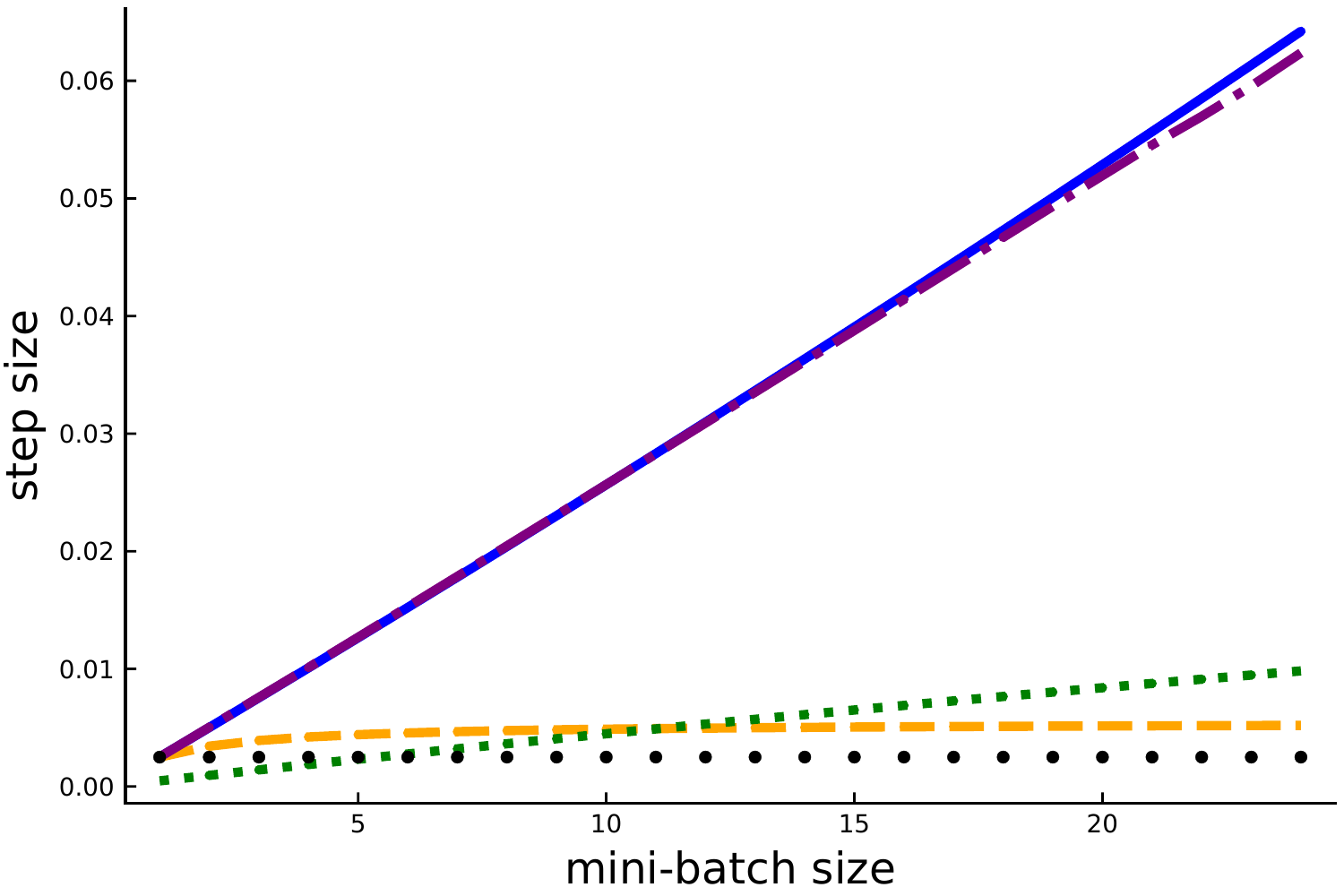}
        \end{subfigure}\\[-0.05cm]
        \begin{subfigure}[b]{\columnwidth}
            \centerline{\includegraphics[width=0.95\textwidth]{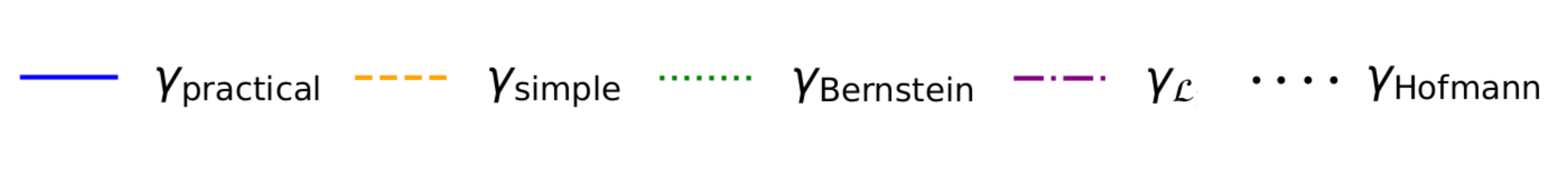}}
        \end{subfigure}
    \caption{Step size estimates as a function the mini-batch size $b$ for ridge regression on the unscaled \emph{staircase eigval} dataset ($\lambda = 10^{-3}$).}
    \label{fig:intro_step_size_example}
    \end{center}
    \vskip -0.2in
\end{figure}

All the experiments were run in Julia and the code is freely available on~\url{https://github.com/gowerrobert/StochOpt.jl}.
\subsection{Upper-bounds of the expected smoothness constant}
\label{sub:experiment_1_upper_bounds_of_the_expected_smoothness_constant}
First we experimentally verify that our upper-bounds hold and how much slack there is between them and $\cL$ given in~\Cref{eq:expsmooth}.
For ridge regression applied to artificially generated small datasets, we compute~\Cref{eq:expsmooth} and compare it to our \emph{simple} and \emph{Bernstein bounds}, and our \emph{practical} estimate.
Our data are matrices $A \in \R^{d \times n}$ defined as follows
\begin{align*}
&\textit{$\bullet$ uniform } (n=24, d=50):  [A]_{ij} \sim \cU([0,1)) \enspace,\\
&\textit{$\bullet$ alone eigval } (n=d=24):
A = \text{diag} \left( 1, \ldots , 1, 100 \right) \enspace,\\
&\textit{$\bullet$ staircase eigval } (n=d=24):\\
&\quad \quad A = \text{diag} \left( 1, 10\sqrt{1/n}, \ldots , 10\sqrt{(n-2)/n}, 10 \right) \enspace.
\end{align*}

\begin{figure}[t]
\vskip 0.2in
\begin{center}
\centerline{\includegraphics[width=0.88\columnwidth]{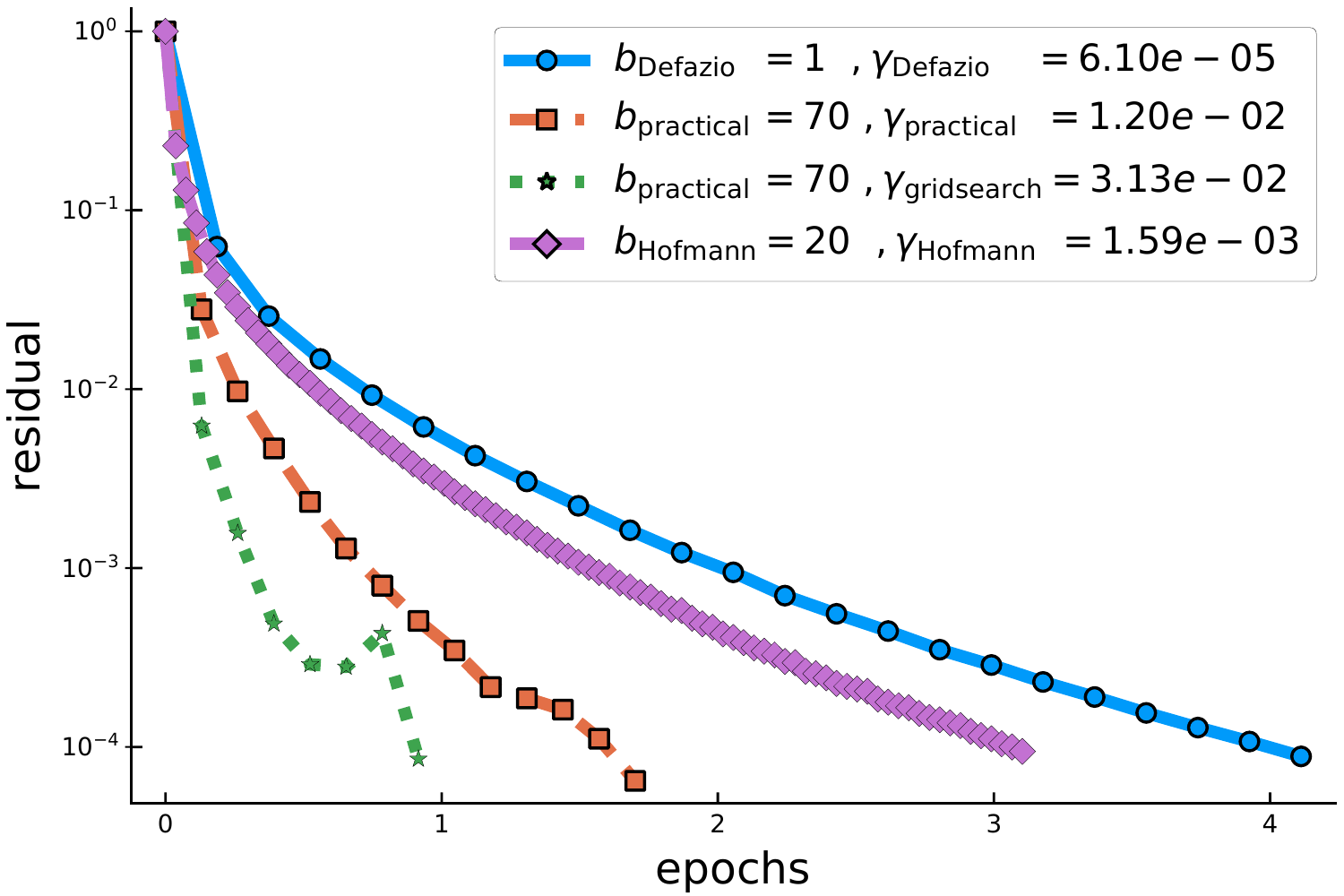}}
\caption{Comparison of SAGA settings for ridge regression on the unscaled \emph{slice} dataset ($\lambda = 10^{-1}$).}
\label{exp_3_main}
\end{center}
\vskip -0.2in
\end{figure}

In Figure~\ref{fig:intro_step_size_example} we see that
$\cL_{\text{practical}}$ is arbitrarily close to $\cL$, making it hard to distinguish the two line plots. This was the case in many other experiments, which we defer to \Cref{appendix:exp1}. For this reason, we use $\gamma_{\text{practical}}$ in our experiments with the SAGA method.

Furthermore, in accordance with our discussion in~\Cref{sub:upper_bounds_of_expsmoothness}, we have that $\cL_{\text{simple}}$ and  $\cL_{\text{Bernstein}}$ are close to $\cL$ when $b$ is small and large, respectively.
In~\Cref{appendix:exp_1_real_datasets} we show, by applying ridge and regularized logistic regression to publicly available datasets from LIBSVM\footnote{\tiny{\url{https://www.csie.ntu.edu.tw/~cjlin/libsvmtools/datasets/}}} and the UCI repository\footnote{\tiny\url{https://archive.ics.uci.edu/ml/datasets/}}, that the \emph{simple bound} performs better than the \emph{Bernstein bound} when $n \gg d$, and conversely for $d$ slightly smaller than $n$, larger than $n$ or when scaling the data.

\subsection{Related step size estimation}
\label{sub:experiment_2_related_step_size_estimation}

Different bounds on $\cL$ also give different step sizes~\eqref{eq:gammamaster}.
Plugging in our estimates $\cL_\text{simple}$, $\cL_\text{Bernstein}$ and $\cL_\text{practical}$ into~\eqref{eq:gammamaster} gives the step sizes
$\gamma_\text{simple}$, $\gamma_\text{Bernstein}$ and $\gamma_\text{practical}$, respectively.
We compare our resulting step sizes to $\gamma_{\cL}$ where $\cL$ is given by~\cref{eq:expsmooth} and to
the step size  given by~\citet{hofmann2015variance}, which is
$\gamma_\text{Hofmann} (b) = \frac{K}{2L_{\max}(1+K+\sqrt{1+K^2})}$,
where $K\eqdef \frac{4b L_{\max}}{n \mu}.$
We can see in Figure~\ref{fig:intro_step_size_example}, that for $b=1$, all the step sizes are approximately the same, with the exceptions of the Bernstein step size.
For $b>5$, all of our step sizes are larger than $\gamma_\text{Hofmann} (b)$, in particular $\gamma_\text{practical} (b)$ is significantly larger. These observations are verified in other artificial and real data examples in \Cref{appendix:exp_2_artificial_datasets,appendix:exp_2_real_datasets}.

\subsection{Comparison with previous SAGA settings}
\label{sub:experiment_3_comparison_with_previous_saga_settings}

Here we compare the performance of SAGA when using the mini-batch size and step size
$b =1, \gamma_{\text{Defazio}} \eqdef 1/{3(n\mu + L_{\max})}$ given in \citep{SAGA_Nips},
 $b =20$ and $ \gamma_{\text{Hofmann}} = {20}/{n\mu}$ given in \citet{hofmann2015variance}, to our new practical mini-batch size $b_{\text{practical}} = \left\lfloor 1 + \frac{\mu (n-1)}{4 (L+\lambda)} \right\rfloor$ and step size $\gamma_{\text{practical}}$. Our goal is to verify how much our parameter setting can improve practical performance. We also compare
with a step size $\gamma_{\text{gridsearch}}$ obtained by grid search over odd powers of $2$. These methods are run until they reach a relative error of $10^{-4}$.

We find in Figure~\ref{exp_3_main} that our parameter settings $(\gamma_{\text{practical}}, b_{\text{practical}})$ significantly outperforms the previously suggested parameters, and is even comparable to grid search. Finally, In we show in \Cref{appendix:exp_3} that the settings $(\gamma_{\text{Hofmann}}, b=20)$ can lead to very poor performance compared to our settings.


%

\subsection{Optimality of our mini-batch size}
\label{sub:experiment_4_optimality_of_our_mini_batch_size}

In the last experiment, detailed in \Cref{appendix:exp_4}, we show that our estimation of the optimal mini-batch size $b_{\text{practical}}$ leads to a faster implementation of SAGA.
We build a grid of mini-batch sizes\footnote{Our grid is $\{2^i, i=0,\dots, 14\}$, with $2^{16}, 2^{18}$ and $n$ being added when needed.}
and compute the empirical total complexity required to achieve a relative error of $10^{-4}$, as in \cref{sub:experiment_3_comparison_with_previous_saga_settings}.
\begin{figure}[t]
    \vskip 0.2in
    \begin{center}
    \centerline{\includegraphics[width=0.9\columnwidth]{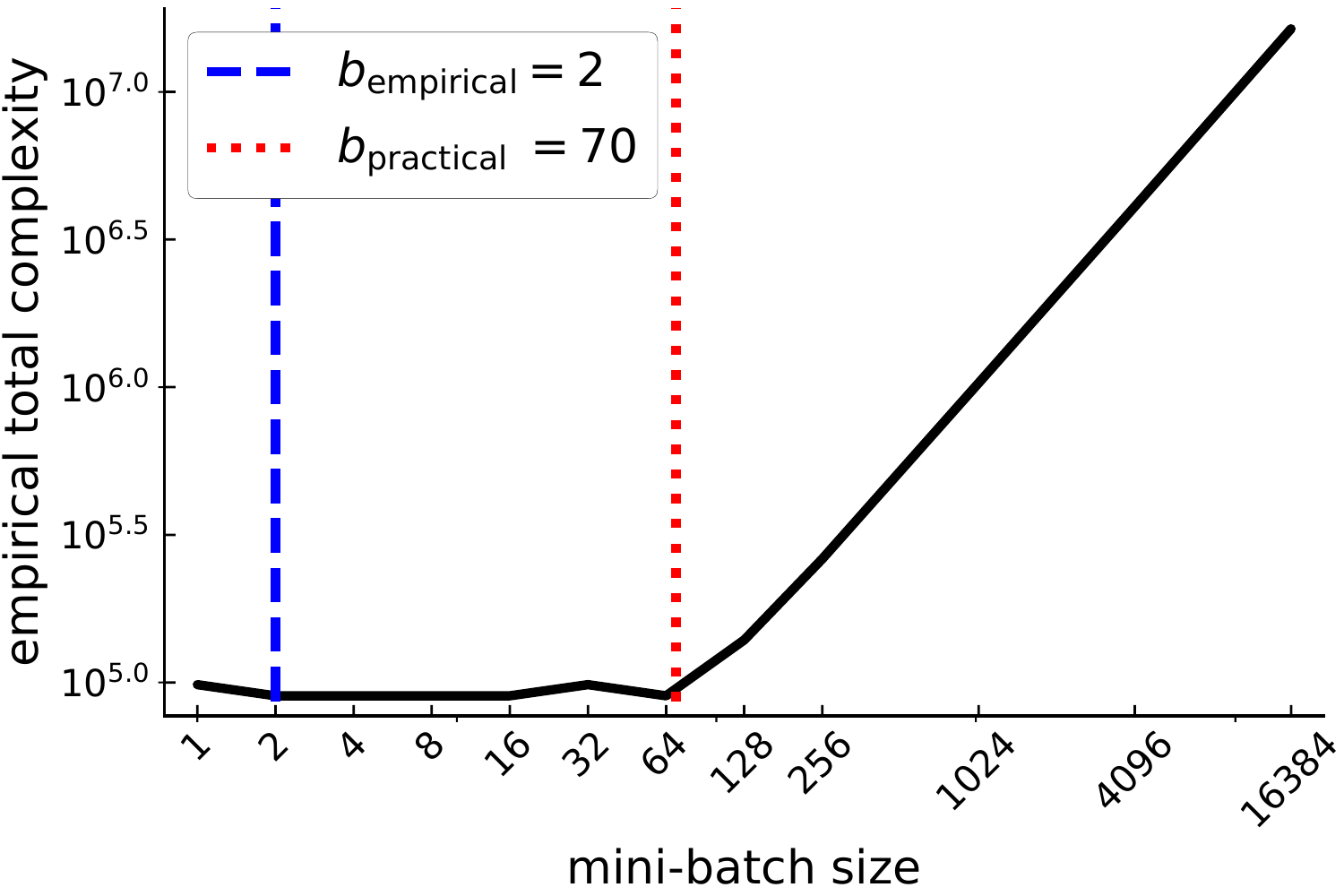}}
    \caption{Total complexity versus the mini-batch size for ridge regression on the unscaled \emph{slice} dataset ($\lambda = 10^{-1}$).}
    \label{exp_4main}
    \end{center}
    \vskip -0.2in
\end{figure}
In Figure~\ref{exp_4main} we can see that the empirical complexity when using $b_{\text{practical}}$ is almost the same as one for the optimal mini-batch size calculated through grid search. Yet, our $b_{\text{practical}}$ being always larger than the mini-batch obtained by grid search, it leads to a faster algorithm for two reasons. Firstly, the corresponding step size is larger and secondly, and secondly, computing the stochastic gradients in parallel improves the running time.
What is even more interesting, is that $b_{\text{practical}}$ always predicts a regime change, where using a larger mini-batch size results in a much larger empirical complexity.


\section{Conclusions}
\label{conclusions}

We have explained the crucial role of the expected smoothness constant $\cL$ in the convergence of a family of stochastic variance-reduced descent algorithms.
We have developed functional upper-bounds of this constant to build larger step sizes and closed-form optimal mini-batch values for the $b$-nice SAGA algorithm.
Our experiments on artificial and real datasets showed the validity of our upper-bounds and the improvement in the total complexity using our step and optimal mini-batch sizes.
Our results suggest a new parameter setting for mini-batch SAGA, that significantly outperforms previous suggested ones, and is even comparable with a grid search approach, without the computational burden of the later.




\newpage
\section*{Acknowledgements}
This work was supported by grants from
DIM Math Innov R\'egion Ile-de-France (ED574 - FMJH)
and by a public grant as part of the Investissement d'avenir project, reference ANR-11-LABX-0056-LMH, LabEx LMH, in a joint call with Gaspard Monge Program for optimization, operations research and their interactions with data sciences.

\bibliography{optimal_mini-batch}
\bibliographystyle{icml2019}
\newpage
\onecolumn 
\appendix

\icmltitle{SUPPLEMENTARY MATERIAL \\ Optimal Mini-Batch and Step Sizes for SAGA}

\section{Proofs of the Upper Bounds of $\cL$}
\label{appendix:proof_of_the_upper_bounds}

\subsection{Master lemma}
\label{sub:master_lemma}

\begin{proof}[Proof of \Cref{lemma:master_lemma}] \label{proof:master_lemma}
Since the $f_i$'s are convex, each realization of $f_v$ is convex, and it follows from equation~2.1.7 in~\cite{NesterovBook}  that
\begin{equation}\label{eq:fvas2}
	\|\nabla f_v(x) - \nabla f_v(y)\|^2_2 \quad \leq \quad 2L_v\left( f_v(x) - f_v(y) - \langle \nabla f_v(y), x-y \rangle \right).
\end{equation}
Taking expectation over the sampling gives
\begin{eqnarray*}
		\mathbb{E} [\|\nabla f_v(x) - \nabla f_v(x^*)\|^2_2 ] &\leq & 2\E{L_v \left( f_v(x) - f_v(x^*) - \langle \nabla f_v(x^*), x-x^* \rangle \right)} \\
		& \overset{\eqref{eq:fvas2}}{=}&
 \frac{2}{n}\E{\sum_{i =1}^n L_v v_i\left( f_i(x) - f_i(x^*) - \langle \nabla f_i(x^*), x-x^* \rangle \right)}  \\
 & =& 	 \frac{2}{n}\sum_{i =1}^n \E{ L_v v_i} \left( f_i(x) - f_i(y) - \langle \nabla f_i(x^*), x-x^* \rangle \right) \\
 & \leq & 2\max_{i=1,\ldots, n} \E{ L_v v_i} \left( f(x) - f(x^*) - \langle \nabla f(x^*), x-x^*\rangle \right) \\
 & =& 2\max_{i=1,\ldots, n} \E{ L_v v_i} \left( f(x) - f(x^*)\right).
\end{eqnarray*}
where in the last equality the full gradient vanishes because it is computed at optimality.
The result now follows by comparing the above with the definition of expected smoothness in~\eqref{eq:expected_smoothness}.
\end{proof}


\subsection{Proof of the simple bound}
\label{sub:proof_of_the_simple_bound}

\begin{proof}[Proof of~\Cref{thm:upper_bound_exp_smooth_simple_combination}]
  To derive this bound on $\cL$ we use that
  \begin{equation} \label{eq:LcupperLis}
    L_B \leq \frac{1}{b}\sum_{j\in B} L_j\enspace,
  \end{equation}
  which follows from repeatedly applying~\Cref{lem:lemma1}.
  For $b \geq 2$, it follows from~\Cref{eq:expsmooth} and~\Cref{eq:LcupperLis} that
  \begin{align} \label{eq:asd9ka903k}
    \cL &\leq \frac{1 }{b \binom{n-1}{b-1}} \max_{i=1,\ldots, n} \Bigg\{ \sum_{\substack{B \subseteq [n] \,:\\ |B| = b \wedge i \in B}} \; \sum_{j\in B} L_j \Bigg\}\enspace.
  \end{align}
  Using a double counting argument we can show that
  \begin{align}
    \sum_{\substack{B \subseteq [n] \,:\\ |B| = b \wedge i \in B}} \;\sum_{j\in B} L_j
    & = \sum_{j=1}^n \; \sum_{\substack{B \subseteq [n] \,:\\ |B| = b \wedge i,j \in B}} L_j  \nonumber \\
    & = \sum_{j\neq i} \; \sum_{\substack{B \subseteq [n] \,:\\ |B| = b \wedge i,j \in B}} L_j + \sum_{\substack{B \subseteq [n] \,:\\ |B| = b \wedge i \in B}} L_i \nonumber \\
    & = \sum_{j\neq i} \binom{n-2}{b-2} L_j  + \binom{n-1}{b-1} L_i \nonumber \\
    & = \binom{n-2}{b-2} (n\Lbarconst-L_i) + \binom{n-1}{b-1} L_i\enspace.
  \end{align}
  Inserting this into~\Cref{eq:asd9ka903k} gives
  \begin{align}
    \cL  &\leq \frac{1 }{b \binom{n-1}{b-1}} \max_{i=1,\ldots, n} \left\{\binom{n-2}{b-2} n\Lbarconst  + \left(\binom{n-1}{b-1} - \binom{n-2}{b-2} \right) L_{\max} \right\} \nonumber \\
    &= \frac{n \binom{n-2}{b-2} }{b \binom{n-1}{b-1}}\Lbarconst + \frac{\binom{n-1}{b-1} -\binom{n-2}{b-2} }{b \binom{n-1}{b-1}}L_{\max} \nonumber \\
    &= \frac{n }{b}\frac{b -1}{n-1} \Lbarconst + \frac{1}{b} \frac{n-b}{n-1} L_{\max} \enspace.
  \end{align}
  We also verify that this bound is valid for $1$-nice sampling. Indeed, we already have that in this case $\cL = L_{\max}$.
\end{proof}

\subsection{Proof of the Bernstein bound}
\label{sub:proof_of_the_bernstein_bound}

To start the proof of \cref{thm:cL_bound_direct_expectation}, we re-write the expected smoothness constant as the maximum over an expectation. Let $S^i$ be a $(b-1)$-nice sampling over $[n]\setminus\{i\}.$
We can write
\begin{eqnarray}
  \cL &= & \frac{1}{\binom{n-1}{b-1}} \max_{i=1,\ldots, n} \Bigg\{\sum_{\substack{B \subseteq [n] \,:\\ |B| = b \wedge i \in B}}  L_B \Bigg\}\nonumber \\
  &= &  \max_{i=1,\ldots, n} \E{  L_{S^i\cup\{i\}} } \nonumber \\
  &\overset{\text{\scriptsize \Cref{lemma:LB_def_1}} }{=}& \max_{i=1,\ldots, n} U \E{ \lambda_{\max}\left(\frac{1}{b}\sum_{j \in S^i\cup\{i\}} a_j a_j^\top \right) } \enspace. \label{eq:cL1_rewrite_expectation_lambda_max}
\end{eqnarray}

One can come back to the definition of the subsample smoothness constant~\Cref{eq:LBdef} and interpret previous expression as an expectation of the largest eigenvalue of a sum of matrices.
This insight allows us to apply a matrix Bernstein inequality, see~\Cref{thm:mat_bernstein_expect_without_replac}, to bound $\cL$.

For the proof of~\Cref{thm:cL_bound_direct_expectation}, we first need the two following results.
\begin{lemma} \label{lem:lemma_expect_B_minus_i}
  Let $a_j \in \R^d$, $i \in \{1,\ldots, n \}$ and let $S^i$ be a $(b-1)$-nice sampling over the set $[n]\setminus\{i\}$. It follows that
  \begin{equation} \label{eq:equality_expect_B_minus_i}
    \E{\sum_{j \in S^i\cup\{i\}} a_j a_j^\top } = a_i a_i^\top + \frac{b -1}{n-1} \sum_{j=1, j\neq i}^{n} a_j a_j^\top \enspace.
  \end{equation}
\end{lemma}

\begin{proof}[Proof of~\Cref{lem:lemma_expect_B_minus_i}]
  This results follows using a double-counting argument at the fourth line of the computation.
  \begin{align*}
    \E{\sum_{j \in S^i\cup\{i\}} a_j a_j^\top } &= \frac{1}{\binom{n-1}{b-1}} \sum_{\substack{B \subseteq [n]\setminus \{i\} :\\ |B| = b-1}} \; \sum_{j \in B\cup \{i\}} a_j a_j^\top\\
    &= \frac{1}{\binom{n-1}{b-1}} \left(\binom{n-1}{b-1} a_ia_i^\top + \sum_{\substack{B \subseteq [n]\setminus \{i\} :\\ |B| = b-1}} \; \sum_{j \in B} a_j a_j^\top \right) \\
    &= a_ia_i^\top+ \frac{1}{\binom{n-1}{b-1}} \sum_{\substack{B \subseteq [n]\setminus \{i\} :\\ |B| = b-1}} \; \sum_{j \in B} a_j a_j^\top \\
    &= a_ia_i^\top+ \frac{1}{\binom{n-1}{b-1}} \sum_{j=1, j\neq i}^n \; \sum_{\substack{B \subseteq [n]\setminus \{i\} :\\ |B| = b-1 \wedge j \in B}} a_j a_j^\top\\
    &= a_ia_i^\top+ \frac{\binom{n-2}{b -2}}{\binom{n-1}{b-1}}  \sum_{j=1, j\neq i}^n   a_j a_j^\top\\
    &= a_i a_i^\top + \frac{b -1}{n-1} \sum_{j=1, j\neq i}^{n} a_j a_j^\top \enspace.
  \end{align*}
\end{proof}
We then introduce another two lemmas which give a first intermediate bound.
\begin{lemma} \label{lem:lemma_bound_expect_B}
  Let  $a_j \in \R^d$ for $j \in [n]$, let $i \in [n]$ and let $S^i$ be a $(b-1)$-nice sampling over $[n]\setminus\{i\}.$ We have
  \begin{align}
    U\E{\lambda_{\max}\left(\frac{1}{b}\sum_{j \in S^i\cup\{i\}} a_j a_j^\top \right)} \leq& \frac{1}{b (n-1)} \left( (n-b)L_i + n(b -1) L\right) \nonumber \\
    &+ U\E{\lambda_{\max} \left( \frac{1}{b} \sum_{ j \in S^i} a_j a_j^\top - \frac{1}{b} \frac{b -1}{n-1} \sum_{j\in [n]\setminus\{i\}} a_j a_j^\top \right)}\enspace. \label{eq:bound_expect_B}
  \end{align}
\end{lemma}
\begin{proof}[Proof of~\Cref{lem:lemma_bound_expect_B}]
  Expanding the expectation we have
  \begin{align*}
    &\E{\lambda_{\max}\left(\frac{1}{b}\sum_{j \in S^i\cup\{i\}} a_j a_j^\top \right)}\\
    &\leq \lambda_{\max} \left( \E{\frac{1}{b}\sum_{j \in S^i\cup\{i\}} a_j a_j^\top } \right)
    + \E{ \lambda_{\max} \left( \frac{1}{b} \sum_{j \in S^i\cup\{i\}} a_j a_j^\top - \E{\sum_{j \in S^i\cup\{i\}} a_j a_j^\top } \right) }\\
    &= \frac{1}{b} \lambda_{\max} \left( a_i a_i^\top + \frac{b -1}{n-1} \sum_{j=1,\, j\neq i }^{n} a_j a_j^\top \right)
    + \E{\lambda_{\max} \left( \frac{1}{b}\sum_{j \in S^i\cup\{i\}} a_j a_j^\top - \frac{1}{b} \left( a_i a_i^\top + \frac{b -1}{n-1} \sum_{j=1,\, j\neq i }^{n} a_j a_j^\top \right) \right) } \\
    &= \frac{1}{b} \lambda_{\max} \left( \frac{1}{n-1} \left( (n-b) a_i a_i^\top + (b -1) \sum_{j=1}^{n} a_j a_j^\top \right) \right)
    + \mathbb{E} \left[ \lambda_{\max} \left( \frac{1}{b} \sum_{ j \in S^i} a_j a_j^\top - \frac{1}{b} \frac{b -1}{n-1} \sum_{j\in [n]\setminus\{i\}} a_j a_j^\top \right) \right]\\
    &\leq \frac{1}{b (n-1)} \left( (n-b) \frac{L_i}{U} + n(b -1) \frac{L}{U} \right)
    + \E{\lambda_{\max} \left( \frac{1}{b} \sum_{ j \in S^i} a_j a_j^\top - \frac{1}{b} \frac{b -1}{n-1} \sum_{j\in [n]\setminus\{i\}} a_j a_j^\top \right)}\ ,
  \end{align*}
  where in the first inequality we add and remove the mean and then apply \Cref{lem:lemma1}. In the second equality we explicit the mean with \Cref{lem:lemma_expect_B_minus_i} and in the last inequality we use again \Cref{lem:lemma1} for the left-hand side term.
  Finally, we multiply by $U$ on both sides of the inequality.
\end{proof}

We recall the following lemma used to introduced the \emph{practical estimate} given by
\begin{equation*}
  \cL_{\text{practical}} (b) \overset{\text{\scriptsize \Cref{def:practical_bound}} }{\eqdef} \frac{n }{b}\frac{b -1}{n-1} L + \frac{1}{b} \frac{n-b}{n-1} L_{\max} \enspace.
\end{equation*}
\practicalmotivate*

\begin{proof}[Proof of~\Cref{lem:practicalmotivate}] \label{proof:proof_practical}
  The result comes from applying re-writing $\cL$ as an expectation of the largest eigenvalue of a sum of matrices. Then we apply \Cref{lem:lemma_bound_expect_B} and then taking the maximum over all $i \in [n]$. Thus, we have
  \begin{align*}
    \cL &\overset{\text{\scriptsize \eqref{eq:cL1_rewrite_expectation_lambda_max}} }{=} \max_{i=1,\ldots, n} U \E{ \lambda_{\max}\left(\frac{1}{b}\sum_{j \in S^i\cup\{i\}} a_j a_j^\top \right) } \\
    &\overset{\text{\scriptsize \Cref{lem:lemma_bound_expect_B}} }{\leq} \max_{i=1,\ldots, n} \left\{
      \frac{1}{b (n-1)} \left( (n-b)L_i + n(b -1) L\right) + U \E{\lambda_{\max} \left( \frac{1}{b} \sum_{ j \in S^i} a_j a_j^\top - \frac{1}{b} \frac{b -1}{n-1} \sum_{j\in [n]\setminus\{i\}} a_j a_j^\top \right)}
    \right\}\\
    &\leq \frac{n }{b}\frac{b -1}{n-1} L + \frac{1}{b} \frac{n-b}{n-1} L_{\max} + \max_{i=1,\ldots, n}U\E{\lambda_{\max} \left( \frac{1}{b} \sum_{ j \in S^i} a_j a_j^\top - \frac{1}{b} \frac{b -1}{n-1} \sum_{j\in [n]\setminus\{i\}} a_j a_j^\top \right)} \enspace.
  \end{align*}
\end{proof}

\begin{proof}[Proof of~\Cref{thm:cL_bound_direct_expectation}]
  Applying the previous lemma we get
  \begin{equation}\label{eq:backup_for_dead_link}
    \cL \overset{\text{\scriptsize \eqref{eq:withpracticalmotivate}} }{\leq} \frac{n }{b}\frac{b -1}{n-1} L + \frac{1}{b} \frac{n-b}{n-1} L_{\max} + \max_{i=1,\ldots, n}U\E{\lambda_{\max} \left( N \right)} \enspace,
  \end{equation}
  with $N \eqdef \frac{1}{b} \sum_{ j \in S^i} a_j a_j^\top - \frac{1}{b} \frac{b -1}{n-1} \sum_{j\in [n]\setminus\{i\}} a_j a_j^\top$.

  To further our argument, we will encode different samplings using unit coordinate vectors. Let $e_1,\ldots, e_n$ $\in \R^n$ be the unit coordinate vectors.
  Let $S^i = \{ S_{1}^i, \ldots, S_{b}^i \}$ denote an arbitrary but fixed ordering of the elements of $S^i$. With this we can encode the sampling without replacement as
  \begin{equation}\label{eq:encodesamp}
    \sum_{ j \in S^i} a_j a_j^\top = \sum_{k=1}^{b -1} \sum_{j\in [n]\setminus\{i\}} (e_{j})_{S^i_k}  a_j a_j^\top \enspace.
  \end{equation}
  Using this notation, the matrix $N$ which can be further decomposed as
  \begin{align*}
    N &= \frac{1}{b} \sum_{ j \in S^i} a_j a_j^\top - \frac{1}{b} \frac{b -1}{n-1} \sum_{j\in [n]\setminus\{i\}} a_j a_j^\top \\
    &= \frac{1}{b} \sum_{k=1}^{b -1} \sum_{j\in [n]\setminus\{i\}} (e_{j})_{S^i_k} a_j a_j^\top - \frac{1}{b} \frac{b -1}{n-1} \sum_{j\in [n]\setminus\{i\}} a_j a_j^\top \\
    &= \sum_{k=1}^{b -1} \frac{1}{b} \sum_{j\in [n]\setminus\{i\}} \left( (e_{j})_{S^i_k} - \frac{1}{n-1} \right) a_j a_j^\top \\
    &\eqdef \sum_{k=1}^{b -1} M_k \enspace.
  \end{align*}
  where we have encoded the sampling $S^i$ using unit coordinate vectors. The matrices $M_1, \ldots, M_{b -1}$ are sampled \emph{without} replacement from the set
  \begin{equation} \label{eq:samplMseti}
    \vast\{ \vast. \sum_{j \in [n]\setminus\{i\}} \frac{1}{b} \left(x_j - \frac{1}{n-1} \right)a_j a_j^\top \; : \; x \in \{e_1, \ldots, e_{i-1},e_{i+1}, \ldots ,e_n \} \vast. \vast\}\enspace.
  \end{equation}
  Now let $X_1, \ldots, X_{b}$ be matrices sampled \emph{with} replacement from~\eqref{eq:samplMseti} and let $X_k \eqdef \frac{1}{b}\sum_{j\in [n]\setminus\{i\}} \left(z_{j}^k - \frac{1}{n-1} \right)a_j a_j^\top$ and $Y \eqdef \sum_{k=1}^{b -1} X_k$
  thus the vectors $z^k$ are sampled with replacement from $\{e_1, \ldots, e_{i-1},e_{i+1}, \ldots ,e_n \}.$
  Consequently
  \begin{equation*}
    \Prb{z_j^k = 1} = \frac{1}{n-1} \enspace, \quad \forall j \in \{1,\ldots, i-1,i+1,\ldots, n\} \enspace.
  \end{equation*}
  We are now in a position to apply the Bernstein matrix inequality.
  To this end we have
  \begin{itemize}[leftmargin=*]
    \item A sum of centered random matrices: $\E{X_k}=0$.
    \item Let $k^*$ be the unique index such that $z^k_{k^*}=1.$ We have a  uniform bound of the largest eigenvalue of our $X_k$
          \noindent \begin{align}
            \lambda_{\max} (X_k)
            & = \frac{1}{b} \lambda_{\max} \left(\sum_{j\in [n]\setminus\{i\}} z_{j}^k a_j a_j^\top - \frac{1}{n-1} \sum_{j\in [n]\setminus\{i\}} a_j a_j^\top \right) \nonumber \\
            &\leq \frac{1}{b} \lambda_{\max} \left( \sum_{j\in [n]\setminus\{i\}} z_{j}^k a_j a_j^\top \right) \nonumber \\
            &= \frac{1}{b} \lambda_{\max} \left( a_{k^*} a_{k^*}^\top \right) \nonumber \\
            &\leq \frac{1}{b} \frac{L_{\max}}{U} \enspace, \label{eq:bound_lambda_max_x_j_3}
          \end{align}
          where we applied the \Cref{lem:lemma2} in the first inequality.
    \item And a bound on the variance too
          \begin{align}
            \E{X_k^2} &= \mathbb{E} \left( \frac{1}{b} \sum_{j\in [n]\setminus\{i\}} \left(z_{j}^k - \frac{1}{n-1} \right) a_j a_j^\top \right) ^2 \nonumber \\
            &= \frac{1}{b ^2} \mathbb{E} \, \left( \sum_{j,p \in [n]\setminus\{i\} } z_{j}^k z_{p}^k a_j a_j^\top a_p a_p^\top
            - \frac{2}{n-1} \sum_{j,p \in [n]\setminus\{i\} } z_{j}^k a_j a_j^\top a_p a_p^\top
            + \frac{1}{(n-1)^2} \sum_{j,p \in [n]\setminus\{i\} } a_j a_j^\top a_p a_p^\top \right) \nonumber \\
            &= \frac{1}{b ^2} \sum_{j,p \in [n]\setminus\{i\} } \left(( \E{z_{j}^k z_{p}^k} a_j a_j^\top a_p a_p^\top
            - \frac{2}{n-1} \E{z_{j}^k} a_j a_j^\top a_p a_p^\top
            + \frac{1}{(n-1)^2} a_j a_j^\top a_p a_p^\top \right) \nonumber \\
            &= \frac{1}{b ^2} \sum_{j,p \in [n]\setminus\{i\} } \left( \E{z_{j}^k z_{p}^k} a_j a_j^\top a_p a_p^\top
            - \frac{2}{(n-1)^2} a_j a_j^\top a_p a_p^\top
            + \frac{1}{(n-1)^2} a_j a_j^\top a_p a_p^\top \right) \nonumber \\
            &= \frac{1}{b^2} \left( \frac{1}{n-1} \sum_{j\in [n]\setminus\{i\}} a_j a_j^\top a_j a_j^\top
            - \frac{1}{(n-1)^2} \sum_{j,p \in [n]\setminus\{i\} } a_j a_j^\top a_p a_p^\top \right) \enspace, \label{eq:intermediate_step_variance_3}
          \end{align}
          where, in the last equality, we used that $z_{j}^k z_{p}^k = 0$ if $j \neq p$ and $\E{z_{j}^k z_{j}^k} = \E{z_{j}^k} = \frac{1}{n-1}$, so that
          \begin{align*}
            \sum_{j,p \in [n]\setminus\{i\} } \E{z_{j}^k z_{p}^k} a_j a_j^\top a_p a_p^\top
            &= \E{\sum_{j,p \in [n]\setminus\{i\} } z_{j}^k z_{p}^k} a_j a_j^\top a_p a_p^\top \\
            &= \sum_{j\in [n]\setminus\{i\}} \E{z_{j}^k z_{j}^k} a_j a_j^\top a_p a_p^\top \\
            &= \frac{1}{n-1} \sum_{j\in [n]\setminus\{i\}} a_j a_j^\top a_p a_p^\top \enspace.
          \end{align*}
          Summing in \eqref{eq:intermediate_step_variance_3}, taking the largest eigenvalue and applying~\Cref{lem:lemma2} results in
          \begin{align}
            \lambda_{\max} \left( \sum_{k=1}^{b-1} \E{X_k^2} \right) &\leq \lambda_{\max} \left( \sum_{k=1}^{b-1} \frac{1}{b ^2} \frac{1}{n-1} \sum_{j\in [n]\setminus\{i\}} a_j a_j^\top a_j a_j^\top \right) \nonumber \\
            &\leq \frac{b-1}{b^2} \left( \max_{j \in [n]\setminus\{i\}} \lambda_{\max} \left( a_j a_j^\top \right) \right) \cdot \lambda_{\max} \left( \frac{1}{n-1} \sum_{j\in [n]\setminus\{i\}} a_j a_j^\top \right) \nonumber \\
            &\leq \frac{b-1}{b^2} \frac{L_{\max}}{U^2} L_{[n] \setminus \{i \}} \enspace. \label{eq:bound_variance_3}
          \end{align}
  \end{itemize}

  Considering~\Cref{eq:bound_lambda_max_x_j_3,eq:bound_variance_3} and applying the matrix Bernstein concentration inequality in Theorem~\ref{thm:mat_bernstein_expect_without_replac} we get
  \begin{equation*}
    U\E{\lambda_{\max}(N)} \leq \sqrt{2 \frac{b-1}{b^2} L_{\max} L_{[n] \setminus \{i \}} \log d} + \frac{1}{3} \frac{L_{\max}}{b} \log d
  \end{equation*}
  Taking the maximum over $i$ and using $L_{[n] \setminus \{i \}} \leq \frac{n}{n-1}L$ we have that
  \begin{equation*}
    \max_{i=1,\ldots, n} U\E{\lambda_{\max}(N)} \leq \sqrt{2 \frac{b-1}{b^2}\frac{n}{n-1} L_{\max} L \log d} + \frac{1}{3} \frac{L_{\max}}{b} \log d
  \end{equation*}
  Combining the above result with~\eqref{eq:backup_for_dead_link} leads us to
  \begin{align*}
    \cL & \leq  \frac{(n-b)L_{\max}}{b (n-1)} + \frac{n(b -1) L }{b (n-1)}
    + \sqrt{2 \left(\frac{b-1}{b} \frac{n}{n-1} L\right) \cdot \left(\frac{1}{b}L_{\max} \log d\right)}
    + \frac{1}{3} \frac{L_{\max}}{b} \log d \\
    &\leq \frac{(n-b)L_{\max}}{b (n-1)} + \frac{n(b -1) L }{b (n-1)}
    + \frac{b-1}{b} \frac{n}{n-1}L + \frac{4}{3}\frac{L_{\max}}{b}\log(d)\\
    &= 2\frac{b -1 }{b }\frac{n}{n-1}L + \frac{1}{b}\left(\frac{4}{3}\log(d) +\frac{n-b}{n-1} \right)L_{\max}.
  \end{align*}
  where in the second inequality we used the inequality $\sqrt{2ab} \leq a+ b$.
\end{proof}

\section{Linear Algebra Tools}
\label{appendix:linear_algebra_tools}

This appendix is dedicated to the presentation of useful results to manipulate more easily the smoothness constants.

\subsection{Spectral Lemmas}
\label{sub:spectral_lemmas}

Let us recall some useful spectral results on Hermitian and positive semi-definite matrices.
\begin{lemma}{(Weyl's inequality)} \label{lem:Weyl}
    Let $A, B \in \R^{n \times n}$ symmetric matrices.
    Assume that the eigenvalues of $A$ (resp. $B$) are sorted \ie $\lambda_1 (A) \geq \dots \geq \lambda_n (A)$ (resp. $\lambda_1 (B) \geq \dots \geq \lambda_n (B)$).
    Then, we have
    \begin{equation}
        \lambda_{i+j-1} (A+B) \leq \lambda_i (A) + \lambda_j (B)\enspace.
    \end{equation}
    whenever $i,j \geq 1$ and $i+j-1 \leq n \enspace.$
\end{lemma}

Moreover, as a direct consequence of the variational characterization of eigenvalues, namely
\begin{equation}
    \lambda_{\max} (A) = \max_{v \neq 0} \frac{v^\top A v}{\norm{v}_2^2} \enspace,
\end{equation}
we have an inequality between the maximum diagonal term of a positive semi-definite matrices and its maximum eigenvalue.
\begin{lemma}
    Let $A \in \R^{n \times n}$ positive semi-definite matrix and the vector containing its diagonal $d \eqdef \mathrm{diag}(A)$. Then, we have
    \begin{equation}
        \max_{i = 1,\dots,n} d_i \leq \lambda_{\max} (A)\enspace.
    \end{equation}
\end{lemma}

The following lemma is a direct consequence of Weyl's inequality for $i=j=1.$
\begin{lemma} \label{lem:lemma1}
    Let $A, B \in \R^{n \times n}$ symmetric matrices. Then, we have
    \begin{equation}
        \lambda_{\max} (A+B) \leq \lambda_{\max} (A) + \lambda_{\max}(B) \enspace.
    \end{equation}
\end{lemma}

Lastly, we present a result arising from previous lemma.
\begin{lemma} \label{lem:lemma2}
    Let $A, B \in \R^{n \times n}$ symmetric matrices such that $B$ is positive semi-definite. Then, we have
    \begin{equation}
        \lambda_{\max} (A-B) \leq \lambda_{\max} (A)\enspace.
    \end{equation}
\end{lemma}
\begin{proof}
    Let $A, B \in \R^{n \times n}$ symmetric matrices such that $B$ is positive semi-definite. We get directly
    \begin{align*}
      \lambda_{\max} (A-B) &\leq \lambda_{\max} (A) + \lambda_{\max} (-B)\\
                           &= \lambda_{\max} (A) - \lambda_{\min} (B)\\
                           &\leq \lambda_{\max} (A) \enspace,
    \end{align*}
    where the first inequality stems from \Cref{lem:lemma1} and the second from $B \succeq 0.$
\end{proof}

\subsection{Basic properties of the smoothness constants}
\label{sub:smooth_lemmas}

The complexity results of \citet{gower2018stochastic} depends on smoothness constants defined in \Cref{sub:assumptions_and_smoothness_constants}.
Here are some inequalities giving an idea of the order of those constants.
\begin{lemma} \label{lem:lemma_smoothness_cst_ineq}
    Let $\emptyset \neq B \subseteq [n] = \{1,\dots,n\}$ a batch set drawn randomly without replacement. The following inequalities hold
    \begin{enumerate}[label=(\roman*)]
        \item \begin{equation} L_i \leq L_{\max} \quad \forall i=1,\dots,n \enspace. \end{equation}
        \item \label{enum:avg_B_upper_bound} \begin{equation} L_B \leq \frac{1}{|B|} \sum_{i \in B} L_i \quad \forall i=1,\dots,n \enspace. \end{equation}
        \item \begin{equation} \label{eq:smoothness_cst_chain_ineq}
                L \overset{\text{(a)}}{\leq} \Lbarconst \overset{\text{(b)}}{\leq} L_{\max} \overset{\text{(c)}}{\leq} nL \overset{\text{(d)}}{\leq} n\Lbarconst \enspace.
              \end{equation}
    \end{enumerate}
\end{lemma}

\begin{proof}
    \begin{enumerate}[label=(\roman*)]
        \item One directly gets that $L_i \leq \max_{j=1,\ldots,n} L_j = \Lmax$.
        \item This inequality states that the smoothness constant $L_B$ of the averaged function $f_B$ is upper bounded by the average of the corresponding smoothness constants $L_i$, over the batch $B$.
              The proof consists in $|B|$ repetitive calls of \Cref{lem:lemma1}.
        \item $(a)$ Direct implication of \ref{enum:avg_B_upper_bound} for $B=[n]$.

            $(b)$ Direct calculation
            \begin{equation*}
                \Lbarconst = \frac{1}{n} \sum_{i=1}^n L_i \leq \frac{1}{n} \sum_{i=1}^n L_{\max} \leq L_{\max} \enspace.
            \end{equation*}

            $(c)$ Let us first recall the matrix formulation of our smoothness constants:
            \begin{equation*}
                L = \frac{U}{n}\lambda_{\max}(A A^\top) = \frac{U}{n}\lambda_{\max}(A^\top A)
            \end{equation*}
            and
            \begin{equation*}
                L_{\max} = U\max_{i=1,\dots,n} e_i^\top A^\top A e_i \enspace,
            \end{equation*}
            Using the min-max theorem, we have that
            \begin{equation*}
                \lambda_{\max}(A^\top A) = \max_{x\neq 0} \frac{x^\top A^\top A x}{\norm{x}_2^2} \geq \max_{i=1,\ldots, n} e_i^\top A^\top A e_i.
            \end{equation*}
           Dividing the above by $n$ on both sides gives
            \begin{equation*}
                L \geq \frac{L_{\max}}{n} \enspace.
            \end{equation*}

            $(d)$ Direct consequence of $(a)$.
    \end{enumerate}
\end{proof}

\section{Matrix Bernstein Inequality: Sampling Without Replacement}
\label{appendix:matrix_bernstein_inequality}

In this appendix, we present the matrix Bernstein inequality for independent Hermitian matrices from \citet{tropp2015introduction}.
We also provide another version of this theorem for matrices sampled \emph{without} replacement and prove it as explicitly as possible, taking our inspiration from \citet{Tropp2011}.
The proof is based the possibility of transferring the results from sampling \emph{with} to \emph{without} through the inequality \eqref{eq:sampling_mgf_ineq} due to \citet{Gross2010}.
The exact same work can be done for the tail bound, which is for instance used in \citet{Bach2013}.

\subsection{Original Bernstein inequality for independent matrices}
\label{sub:Bernstein_with_replacement}

We first present \Cref{thm:mat_bernstein_expect} which gives a Bernstein inequality for a sum of random and independent Hermitian matrices whose eigenvalues are upper bounded.
If the matrices $X_k$ are sampled from a finite set $\cX$, one can interpret this random sampling of independent matrices as a random sampling \emph{with} replacement.
\begin{theorem}[\citet{tropp2015introduction}, Theorem 6.6.1: Matrix Bernstein Inequality]
    \label{thm:mat_bernstein_expect}
  Consider a finite sequence $\{X_k\}_{k=1,\dots,n}$ of $n$ independent, random, Hermitian matrices with dimension $d$. Assume that
  \begin{equation*}
    \mathbb{E} \, X_k = 0 \quad \mathrm{and} \quad \lambda_{\max} (X_k) \leq L  \quad \text{for each index} \,k \enspace.
  \end{equation*}
  Introduce the random matrix
  \begin{equation*}
    S_X \eqdef \sum_{k=1}^n X_k \enspace.
  \end{equation*}
  Let $v(S_X)$ be the matrix variance statistic of the sum:
  \begin{equation}
    v(S_X) \eqdef \norm{\mathbb{E} \, S_X^2} = \norm{\sum_{k=1}^n \mathbb{E}\, X_k^2} = \lambda_{\max} \left( \sum_{k=1}^n \mathbb{E} \, X_k^2 \right) \ .
  \end{equation}
  Then
  \begin{equation}
    \mathbb{E} \, \lambda_{\max} (S_X) \leq \sqrt{2v(S_X) \log d} + \frac{1}{3} L \log d \enspace.
  \end{equation}
\end{theorem}
This theorem is the one we extend in \Cref{thm:mat_bernstein_expect_without_replac} to the case when the random matrices $X_k$ are sampled \emph{without} replacement from a finite set $\cX$.
We drew our inspiration from the proof of the matrix Chernoff inequality in \citet{Tropp2011} and the one of the matrix Bernstein tail bound in \citet{Bach2013}, both in the case of sampling without replacement.

\subsection{Technical random matrices prerequisites}
\label{sub:technical_random_matrices_prerequisites}

Before proving \Cref{thm:mat_bernstein_expect_without_replac}, which extends the matrix Bernstein inequality to sampling without replacement, we need to introduce the key tools of the matrix Laplace transform technique.
This technique is precious to prove tail bounds for sums of random matrices such as Chernoff, Hoeffding or Bernstein bounds, as presented in \citep{Tropp2012}.

Here, $\norm{\cdot}$ denotes the spectral norm, which is defined for any Hermitian matrix $H$ by
\begin{equation} \label{eq:spectral_norm}
  ||H|| = \max \left\{ \lambda_{\max} (H), -\lambda_{\min} (H) \right\} \enspace.
\end{equation}
We also introduce the moment generating function (mgf) and the cumulant generating function (cgf) of a random matrix, which are essential in the Laplace transform method approach.
\begin{definition}[Matrix Mgf and Cgf] \label{def:mgf_and_cgf}
  Let $X$ be a random Hermitian matrix. For all $\theta \in \R$, the matrix generating function $M_X$ and the matrix cumulant generating function $\Xi_X$ are given by
  \begin{equation*}
    M_X (\theta) \eqdef \mathbb{E} \, e^{\theta X}
  \end{equation*}
  and
  \begin{equation*}
    \Xi_X (\theta) \eqdef \log \mathbb{E} \, e^{\theta X} \enspace.
  \end{equation*}
\end{definition}
\begin{remark}
    These expectations may not exist for all values of $\theta$.
\end{remark}

\begin{proposition}[\citet{tropp2015introduction}, Proposition 3.2.2: Expectation Bound of the Maximum Eigenvalue] \label{prop:expect_bounds_eigenvalues}
  Let $X$ be a random Hermitian matrix. Then
  \begin{equation}
    \mathbb{E} \, \lambda_{\max} \left( X \right) \leq \inf_{\theta > 0} \left\{ \frac{1}{\theta} \log \mathbb{E} \, \mathrm{tr} \,  \mathrm{e}^{ \theta X } \right\} \enspace.
  \end{equation}
\end{proposition}
\begin{remark}
    This proposition is an adaptation of the Laplace transform method to obtain a bound of the expectation of the maximum eigenvalue of a random Hermitian matrix. Contrary to the tail bounds, there is no exact analog of the expectation bounds in the scalar setting.
\end{remark}

\begin{proof}[Proof of~\Cref{prop:expect_bounds_eigenvalues}]
Fix a positive number $\theta$. Because $\lambda_{\max} ( \cdot )$ is a positive-homogeneous map, we have
\begin{align*}
\mathbb{E} \, \lambda_{\max} \left(X \right)
& = \frac{1}{\theta} \mathbb{E} \, \lambda_{\max} \left(\theta X \right)\\
& = \frac{1}{\theta} \mathbb{E} \, \log \mathrm{e}^{\lambda_{\max} \left( \theta X \right)}\\
& \leq \frac{1}{\theta} \log \mathbb{E} \, \mathrm{e}^{\lambda_{\max} \left( \theta X \right) }\\
& = \frac{1}{\theta} \log \mathbb{E} \, \lambda_{\max} \left( \mathrm{e}^{ \theta X } \right)\\
& \leq \frac{1}{\theta} \log \mathbb{E} \, \mathrm{tr} \, \mathrm{e}^{ \theta X } \enspace,
\end{align*}
where in the third line we used the Jensen's inequality, in the fourth one the spectral mapping theorem and in the last line the domination by the trace of a positive-definite matrix.
\end{proof}

\begin{theorem}[\citet{tropp2015introduction}, Theorem 8.1.1: Lieb] \label{thm:lieb}
  Let $H$ be a fixed Hermitian matrix with dimension $d$. The function
  \begin{equation*}
    X \rightarrow \mathrm{tr} \exp \left( H + X \right)
  \end{equation*}
  is a concave map on the the convex cone of $d \times d$ positive-definite matrices.
\end{theorem}
\begin{proof}[Proof of~\Cref{thm:lieb}]
  See Chapter 8 in \citet{tropp2015introduction}.
\end{proof}
\begin{corollary} \label{cor:jensen_lieb}
  Let $H$ be a fixed Hermitian matrix with dimension $d$. Let $X$ be a random Hermitian matrix of same dimension. The following inequality holds
  \begin{equation*}
    \mathbb{E} \, \mathrm{tr} \exp \left( H + X \right) \leq \mathrm{tr} \exp \left( H + \log \mathbb{E} \, e^{X} \right)
  \end{equation*}
  is a concave map on the the convex cone of $d \times d$ positive-definite matrices.
\end{corollary}
\begin{proof}[Proof of~\Cref{cor:jensen_lieb}]
  Introducing $Y=e^{X}$, we have directly
  \begin{align*}
    \mathbb{E} \, \mathrm{tr} \exp \left( H + X \right) &= \mathbb{E} \, \mathrm{tr} \exp \left( H + \log e^{X} \right) \\
    &= \mathbb{E} \, \mathrm{tr} \exp \left( H + \log Y \right) \\
    &\leq \mathrm{tr} \exp \left( H + \log \mathbb{E} \, Y \right)\\
    &= \mathrm{tr} \exp \left( H + \log \mathbb{E} \, e^{X} \right) \enspace.
  \end{align*}
  where the inequality comes from the application of \Cref{thm:lieb} and Jensen's inequality.
\end{proof}
\begin{lemma}[\citet{tropp2015introduction}, Lemma 3.5.1 or \citet{Tropp2012}, Lemma 3.4: Subadditivity of Matrix Cgfs] \label{lem:subbaditivity_mat_cfgs}
  Consider a finite sequence $\{X_k\}$ of independent, random, Hermitian matrices of the same dimension. Let $\theta \in \R$, then
  \begin{align*}
    \mathrm{tr} \exp \left( \Xi_{\sum_k X_k} (\theta) \right) &= \mathbb{E} \, \mathrm{tr} \exp \left( \theta \sum\nolimits_k X_k \right)\\
    &\leq \mathrm{tr} \exp \left( \sum\nolimits_k \log \mathbb{E} \, \mathrm{e}^{\theta X_k} \right)\\
    &= \mathrm{tr} \exp \left( \sum\nolimits_k \Xi_{X_k} (\theta) \right) \enspace.
  \end{align*}
\end{lemma}
\begin{proof}[Proof of~\Cref{lem:subbaditivity_mat_cfgs}]
  Let us assume, without loss of generality, that $\theta = 1$. Let a finite sequence $\{X_k\}_{k=1}^n$ of $n$ independent, random, Hermitian matrices of the same dimension.
  We write down $\mathbb{E}_k$ the expectation with respect only to the $k$-th random matrix $X_k$.
  \begin{align*}
    \mathrm{tr} \exp \left( \Xi_{\sum_{k=1}^n X_k} (1) \right)
    &= \mathrm{tr} \exp \left( \log \mathbb{E} \, \exp \left( \sum_{k=1}^{n} X_k \right) \right)\\
    &= \mathbb{E} \, \mathrm{tr} \exp \left( \sum_{k=1}^{n} X_k \right) \\
    &= \mathbb{E}_1 \hspace{-1pt} \dots \mathbb{E}_{n-1} \mathbb{E}_{n} \, \mathrm{tr} \exp \left( \sum_{k=1}^{n-1} X_k + X_{n+1} \right) \\
    &\leq \mathbb{E}_1 \dots \mathbb{E}_{n-1} \, \mathrm{tr} \exp \left( \sum_{k=1}^{n-1} X_k + \log \mathbb{E}_{n} \, e^{X_{n+1}} \right)\\
    &= \mathbb{E}_1 \hspace{-1pt} \dots \mathbb{E}_{n-1} \, \mathrm{tr} \exp \left( \sum_{k=1}^{n-2} X_k + X_{n-1} + \log \mathbb{E}_{n} \, e^{X_{n+1}} \right) \\
    &\leq \mathbb{E}_1 \hspace{-1pt} \dots \mathbb{E}_{n-2} \, \mathrm{tr} \exp \left( \sum_{k=1}^{n-2} X_k + \log \mathbb{E}_{n-1} \, e^{X_{n-1}} + \log \mathbb{E}_{n} \, e^{X_{n}} \right)\\
    &\leq \dots \leq \mathrm{tr} \exp \left( \sum\nolimits_k \log \mathbb{E} \, \mathrm{e}^{\theta X_k} \right)\\
    &= \mathrm{tr} \exp \left( \sum_k \Xi_{X_k} (\theta) \right) \enspace.
  \end{align*}
  where first and second inequalities result from \Cref{cor:jensen_lieb}, the last one comes the fact that $\mathbb{E}_{k} \, e^{X_{k}} = \mathbb{E} \, e^{X_{k}}, \forall k \in [n]$ and the final equality directly comes from an indentification of \Cref{def:mgf_and_cgf}.
\end{proof}


\begin{lemma}[\citet{tropp2015introduction}, Lemma 6.6.2: Matrix Bernstein Mgf and Cgf Bounds] \label{lem:matrix_bernstein_mgf_and_cgf_bound}
  Let $X$ a random Hermitian matrix such that
  \begin{equation*}
    \mathbb{E} \, X = 0 \quad \text{and} \quad \lambda_{\max} \left( X \right) \leq L \enspace.
  \end{equation*}
  Then, for $0 < \theta < 3 / L$,
  \begin{equation*}
    M_X (\theta):= \mathbb{E} \, e^{\theta X} \preceq \exp \left( \frac{\theta^2 / 2}{1 - \theta L / 3} \cdot \mathbb{E} \, X^2 \right)
  \end{equation*}
  and
  \begin{equation*}
    \Xi_X (\theta):= \log \mathbb{E} \, e^{\theta X} \preceq \frac{\theta^2 / 2}{1 - \theta L / 3} \cdot \mathbb{E} \, X^2 \enspace.
  \end{equation*}
\end{lemma}
\begin{proof}[Proof of~\Cref{lem:matrix_bernstein_mgf_and_cgf_bound}]
  See \citet{tropp2015introduction}.
\end{proof}

\subsection{Extended results for sampling without replacement}
\label{sub:random_mat_ineq_without_replacement}

This section is dedicated to the main result, \Cref{lem:trace_mgf_without_dominated_by_with_replacement},
needed for transferring results from sampling \emph{with} to \emph{without} replacement.
This lemma is actually the matrix version of a classical result from \citet{Hoeffding1963}.
We then combine it with previous results of \Cref{sub:technical_random_matrices_prerequisites} to produce a new master bound in \Cref{thm:master_bound_without},
which is the key inequality of the proof of \Cref{thm:mat_bernstein_expect_without_replac}.
\begin{lemma}[\citet{Gross2010}, Domination of the Trace of the Mgf of a Sample Without Replacement] \label{lem:trace_mgf_without_dominated_by_with_replacement}
  Consider two finite sequences, of same length $n$, $\{X_k\}_{k=1,\dots,n}$ and $\{Y_k\}_{k=1,\dots,n}$ of Hermitian random matrices sampled respectively \emph{with} and \emph{without} replacement from a finite set $\cX$.
  Let $\theta \in \R$, $S_{X}:=\sum_{k=1}^n X_k$ and $S_{Y}:=\sum_{k=1}^n Y_k$, then
  \begin{align} \label{eq:sampling_mgf_ineq}
    \mathrm{tr} \, M_{S_{Y}} (\theta) :=
    \mathbb{E} \, \mathrm{tr} \exp \left( \theta S_{Y} \right)
    \leq \mathbb{E} \, \mathrm{tr} \exp \left( \theta S_{X} \right) \enspace.
  \end{align}
\end{lemma}
\begin{proof}[Proof of~\Cref{lem:trace_mgf_without_dominated_by_with_replacement}]
  The left-hand side equality directly arises from \Cref{def:mgf_and_cgf} and the fact that the trace commutes with the expectation because it is a linear operator.
  For the right-hand side inequality, see the proof in \citet{Gross2010}.
\end{proof}

\begin{theorem}[Master Bound for a Sum of Random Matrices Sampled Without Replacement] \label{thm:master_bound_without}
  Consider two finite sequences, of same length $n$, $\{X_k\}_{k=1,\dots,n}$ and $\{Y_k\}_{k=1,\dots,n}$ of Hermitian random matrices of same size sampled respectively \emph{with} and \emph{without} replacement from a finite set $\cX$. Then
  \begin{align}
    & \mathbb{E} \, \lambda_{\max} \left( \sum_{k=1}^n Y_k \right)
    \leq \inf_{\theta>0} \left\{ \frac{1}{\theta} \log \mathrm{tr} \exp \left( \sum_{k=1}^n \log \mathbb{E} \, \mathrm{e}^{\theta X_k} \right) \right\} \enspace.
  \end{align}
\end{theorem}
\begin{remark}
    This theorem is a modified version of Theorem 3.6.1 in \citet{tropp2015introduction} for a sum of matrices sampled without replacement.
\end{remark}

\begin{proof}[Proof of~\Cref{thm:master_bound_without}]
  Consider two finite sequences, of same length, $\{X_k\}$ and $\{Y_k\}$ of Hermitian random matrices of same size sampled respectively \emph{with} and \emph{without} replacement from a finite set $\cX$.
  Let $\theta$ a positive number.
  \begin{align*}
    \mathbb{E} \, \lambda_{\max} \left( \sum_{k=1}^n Y_k \right)
    &\leq \inf_{\theta > 0} \left\{ \frac{1}{\theta} \log \mathbb{E} \, \mathrm{tr} \, \exp \left( \theta \sum_{k=1}^n Y_k \right) \right\}
     \leq \inf_{\theta > 0} \left\{ \frac{1}{\theta} \log \mathbb{E} \, \mathrm{tr} \, \exp \left( \theta \sum_{k=1}^n X_k \right) \right\}\\
    & \leq \inf_{\theta > 0} \left\{ \frac{1}{\theta} \log \mathrm{tr} \exp \left( \sum_{k=1}^n \log \mathbb{E} \, \mathrm{e}^{\theta X_k} \right) \right\} \enspace.
  \end{align*}
  where we used successively \Cref{prop:expect_bounds_eigenvalues}, \Cref{lem:trace_mgf_without_dominated_by_with_replacement} and \Cref{lem:subbaditivity_mat_cfgs}.
  First, we use the expectation bound for the  maximum eigenvalue.
  We then use the main result of \citet{Gross2010} and invoked in \citet{Tropp2011} to extend the matrix Chernoff bound for matrices sampled \emph{without} replacement.
  This lemma allows us to transfer our results to sampling \emph{with} replacement.
  And finally, we then apply the subadditivity of matrix cgfs to get the desired result.
\end{proof}

\subsection{Bernstein inequality for sampling without replacement}
\label{sub:Bernstein_without_replacement}

The following theorem is almost the same than \Cref{thm:mat_bernstein_expect}, but in the case of matrices sampled \emph{without} replacement from a finite set.
The proof stems from results established in previous \Cref{sub:random_mat_ineq_without_replacement,sub:technical_random_matrices_prerequisites}.
\begin{theorem}[Matrix Bernstein Inequality Without Replacement] \label{thm:mat_bernstein_expect_without_replac}
  Let $\cX$ be a finite set of Hermitian matrices with dimension $d$ such that
  \begin{equation*}
    \lambda_{\max} (X) \leq L, \quad \forall X \in \cX \enspace.
  \end{equation*}
  Sample two finite sequences, of same length $n$, $\{X_k\}_{k=1,\dots,n}$ and $\{Y_k\}_{k=1,\dots,n}$ uniformly at random from $\cX$ respectively \emph{with} and \emph{without} replacement such that
  \begin{equation*}
    \mathbb{E} \, X_k = 0 \quad \forall k \enspace.
  \end{equation*}
  Introduce the random matrices
  \begin{equation*}
    S_X \eqdef \sum_{k=1}^n X_k \quad \text{and} \quad S_Y \eqdef \sum_{k=1}^n Y_k \enspace.
  \end{equation*}
  Let $v(S_X)$ be the matrix variance statistic of the second sum
  \begin{align}
    v(S_X) &\eqdef \norm{\mathbb{E} \, S_X^2} = \norm{\sum_{k=1}^n \mathbb{E}\, X_k^2} = \lambda_{\max} \left( \sum_{k=1}^n \mathbb{E} \, X_k^2 \right).
  \end{align}
  Then
  \begin{equation}
    \mathbb{E} \, \lambda_{\max} (S_Y) \leq \sqrt{2v(S_X) \log d} + \frac{1}{3} L \log d \enspace.
  \end{equation}
\end{theorem}

\begin{proof}[Proof of~\Cref{thm:mat_bernstein_expect_without_replac}]
  Consider $\cX$ a finite set of Hermitian matrices of dimension $d$ such that
  \begin{equation*}
    \lambda_{\max} (X) \leq L \quad \forall X \in \cX \enspace.
  \end{equation*}
  Sample two finite sequences, of same length, $\{X_k\}$ and $\{Y_k\}$ uniformly at random from $\cX$ respectively \emph{with} and \emph{without} replacement such that
  \begin{equation*}
    \mathbb{E} \, X_k = 0 \quad \forall k \enspace.
  \end{equation*}
  The $\{X_k\}$ matrices are thus independent.
  Introduce the sums $S_X =\sum_{k=1}^n X_k$ and $S_Y = \sum_{k=1}^n Y_k$. Let us bound the expectation of the largest eigenvalue of the latter
\begin{align*}
  \mathbb{E} \, \lambda_{\max} (S_Y)
  &= \mathbb{E} \, \lambda_{\max} \left( \sum_{k=1}^n Y_k \right)
  \leq \inf_{\theta >0} \left\{ \frac{1}{\theta} \log \mathrm{tr} \exp \left( \sum_{k=1}^n \log \mathbb{E} \, \mathrm{e}^{\theta X_k} \right) \right\}\\
  &\leq \inf_{0 < \theta < 3 / L} \left\{ \frac{1}{\theta} \log \mathrm{tr} \exp \left( \frac{\theta^2 / 2}{1 - \theta L / 3} \sum_{k=1}^n \mathbb{E} \, X_k ^2 \right) \right\}\\
  &\leq \inf_{0 < \theta < 3 / L} \left\{ \frac{1}{\theta} \log \left[d \,\lambda_{\max} \left(\exp \left( \frac{\theta^2 / 2}{1 - \theta L / 3} \mathbb{E} \, S_X^2 \right) \right) \right] \right\}\\
  &\leq \inf_{0 < \theta < 3 / L} \left\{ \frac{1}{\theta} \log \left[d \,\exp \left( \frac{\theta^2 / 2}{1 - \theta L / 3} \lambda_{\max} \left( \mathbb{E} \, S_X^2 \right) \right) \right] \right\}\\
  &\leq \inf_{0 < \theta < 3 / L} \left\{ \frac{1}{\theta} \log \left[d \,\exp \left( \frac{\theta^2 / 2}{1 - \theta L / 3} v(S_X) \right) \right] \right\}\\
  &= \inf_{0 < \theta < 3 / L} \left\{ \frac{\log d}{\theta} + \frac{\theta / 2}{1 - \theta L / 3} v(S_X) \right\} \enspace.
\end{align*}
where the inequalities sucessively derive from \Cref{thm:master_bound_without}, \Cref{lem:matrix_bernstein_mgf_and_cgf_bound} combined with the monotony of $\mathrm{tr} \exp (\cdot)$, the fact that $\mathrm{tr} (M) \leq d \,\lambda_{\max} (M) , \,\forall M \in \R^{d \times d}$, the spectral mapping theorem and lastly \eqref{eq:spectral_norm} with $\mathbb{E} \, Y^2 \succeq 0$.
Finally, one can complete the infimum, for instance using a computer algebra system, to finish the proof as it was stated in the original proof by \citet{tropp2015introduction}
\footnote{For instance : \href{https://www.wolframalpha.com/input/?i=Minimize\%5B\%7B(log(d)\%2Fx)+\%2B+((x\%2F2)\%2F(1-(L\%2F3)*x))*v,+x+\%3E0,+x+\%3C+(3\%2FL)\%7D,+x\%5D}{\texttt{Minimize[{(log(d)/x) + ((x/2)/(1-(L/3)*x))*v, x >0, x < (3/L)}, x]}} in Wolfram Alpha.}.
In conclusion,
\begin{equation*}
  \mathbb{E} \, \lambda_{\max} ( S_Y ) \leq \sqrt{2v( S_X ) \log d} + \frac{1}{3} L \log d \enspace.
\end{equation*}
\end{proof}

\section{Miscellaneous}
\label{appendix:miscellaneous}

\begin{lemma}[Double counting]\label{lem:2count}
Let $a_{i,C} \in \R$ for $i=1,\ldots, n$ and $C \in \cC$, where $\cC$ is a collection of subsets of $[n]$. Then
\begin{equation}
\sum_{C \in \cC} \sum_{i\in C} a_{i,C} \quad
=\quad   \sum_{i=1}^n \sum_{C \in \cC \; : \; i \in C} a_{i,C}.
\end{equation}
\end{lemma}

\begin{algorithm}[tb]
    \caption{\textsc{JacSketch practical implementation of $b$-nice SAGA}}
    \label{alg:practical_SAGA_implementation}
 \begin{algorithmic}
    \STATE {\bfseries Input:} mini-batch size $b$, step size $\gamma > 0$
    \STATE {\bfseries Initialize:} $w^0 \in \R^d$, $J^0 \in \R^{d \times n}$, $u^0 = \frac{1}{n} J^0 e$
    \FOR{$k=0, 1, 2, \ldots$}
    \STATE Sample a fresh batch $B \subseteq [n]$ s.t. $|B| = b$
    \STATE $\mathrm{aux} = \sum_{i\in B} (\nabla f_i(w^k)- J^k_{:i})$ \tcp*[f]{update the auxiliary vector} 
    \STATE $g^k = u^k  + \frac{1}{b}\mathrm{aux}$ \tcp*[f]{update the unbiased gradient estimate} 
    \STATE $u^{k+1} = u^k  + \frac{1}{n}\mathrm{aux}$ \tcp*[f]{update the biased gradient estimate} 
    \STATE  $J^{k+1}_{:i} = \begin{cases} J^k_{:i} & \quad i\notin B\\
    		                              \nabla f_i(w^k) & \quad i\in B.\end{cases}$ \tcp*[f]{update the Jacobian estimate} 
    \STATE $w^{k+1} = w^{k} - \gamma g^k$ \tcp*[f]{take a step}
    \ENDFOR
\end{algorithmic}
\end{algorithm}

\section{Additional Experiments}
\label{appendix:additional_experiments}

\subsection{Experiment 1: estimates of the expected smoothness constant for artificial datasets}
\label{appendix:exp1}

As described in \Cref{sec:numerical_study}, we compute our the \emph{simple} and \emph{Bernstein bounds}, our \emph{practical estimate} and the true $\cL$ for ridge regression applied to small artificial datasets: \emph{uniform} $(n=24, d=50)$, \emph{staircase eigval} $(n=d=24)$ and \emph{alone eigval} $(n=d=24)$. \Cref{fig:exp1_general_conclusion} shows first that the \emph{practical estimate} is a very close approximation of $\cL$. On the one hand, we observe in \Cref{fig:exp1_general_conclusion_unif} that the \emph{Bernstein bound} performs poorly since the feature dimension is very small $d=50$. On the other hand, \Cref{fig:exp1_general_conclusion_staircase} shows a regime change for $b \approx 10$, which highlight the usefulness of combining our bounds to approximate the expected smoothness constant. Finally, we observe that for the \emph{alone eigval} dataset \Cref{fig:exp1_general_conclusion_alone}, which has one very large eigenvalue far from the rest of the spectrum, the \emph{simple bound} matches $\cL$ because the gap between $\Lbarconst$ and $L$ shrinks.
Indeed, in this configuration $\Lbarconst \approx L \approx \frac{L_{\max}}{n}$.
When the spectrum is more concentrated, like for \emph{staircase eigval}, we get a significant gap between $\Lbarconst$ and $L$ as shown in \Cref{fig:exp1_general_conclusion_staircase}, where the \emph{simple bound} is far from $\cL$ when $b=n$.

\begin{figure}[ht]
  \vskip 0.2in
  \begin{center}
      \begin{subfigure}[b]{0.32\textwidth}
        \includegraphics[width=\textwidth]{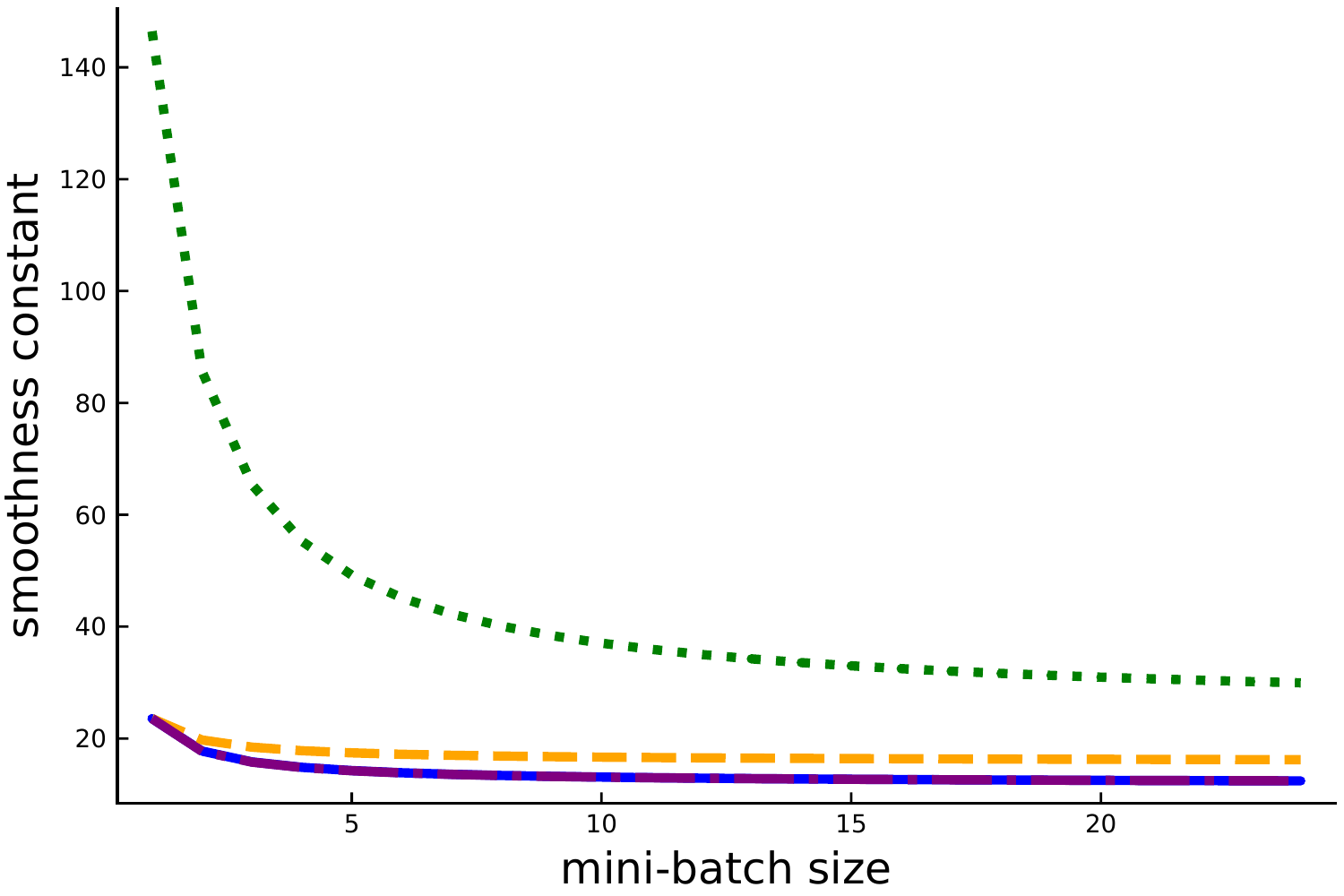}
        \caption{\emph{uniform}.}
        \label{fig:exp1_general_conclusion_unif}
      \end{subfigure}
      \begin{subfigure}[b]{0.32\textwidth}
        \includegraphics[width=\textwidth]{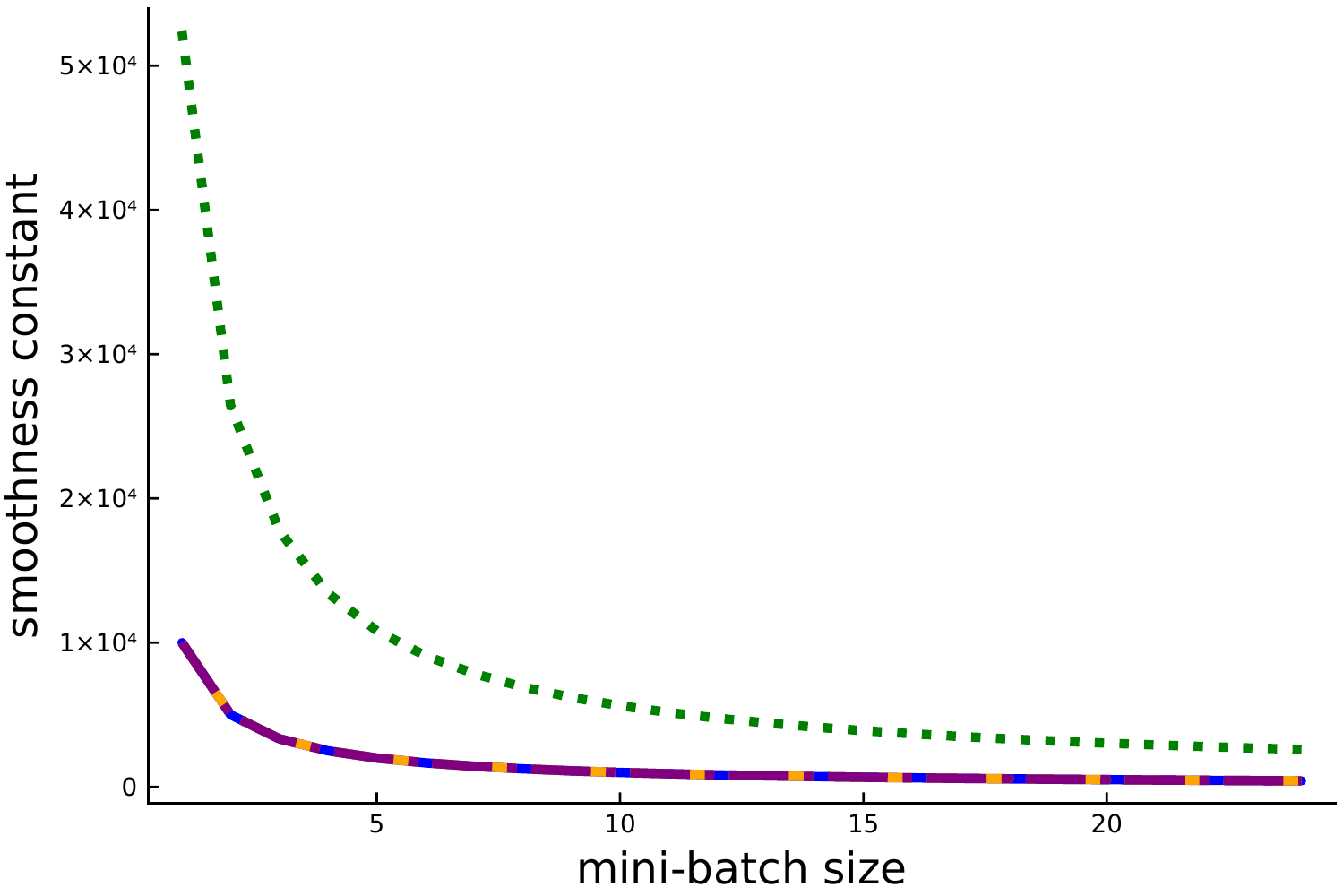}
        \caption{\emph{alone eigval}.}
        \label{fig:exp1_general_conclusion_alone}
      \end{subfigure}
        \begin{subfigure}[b]{0.32\textwidth}
        \centering
        \includegraphics[width=\textwidth]{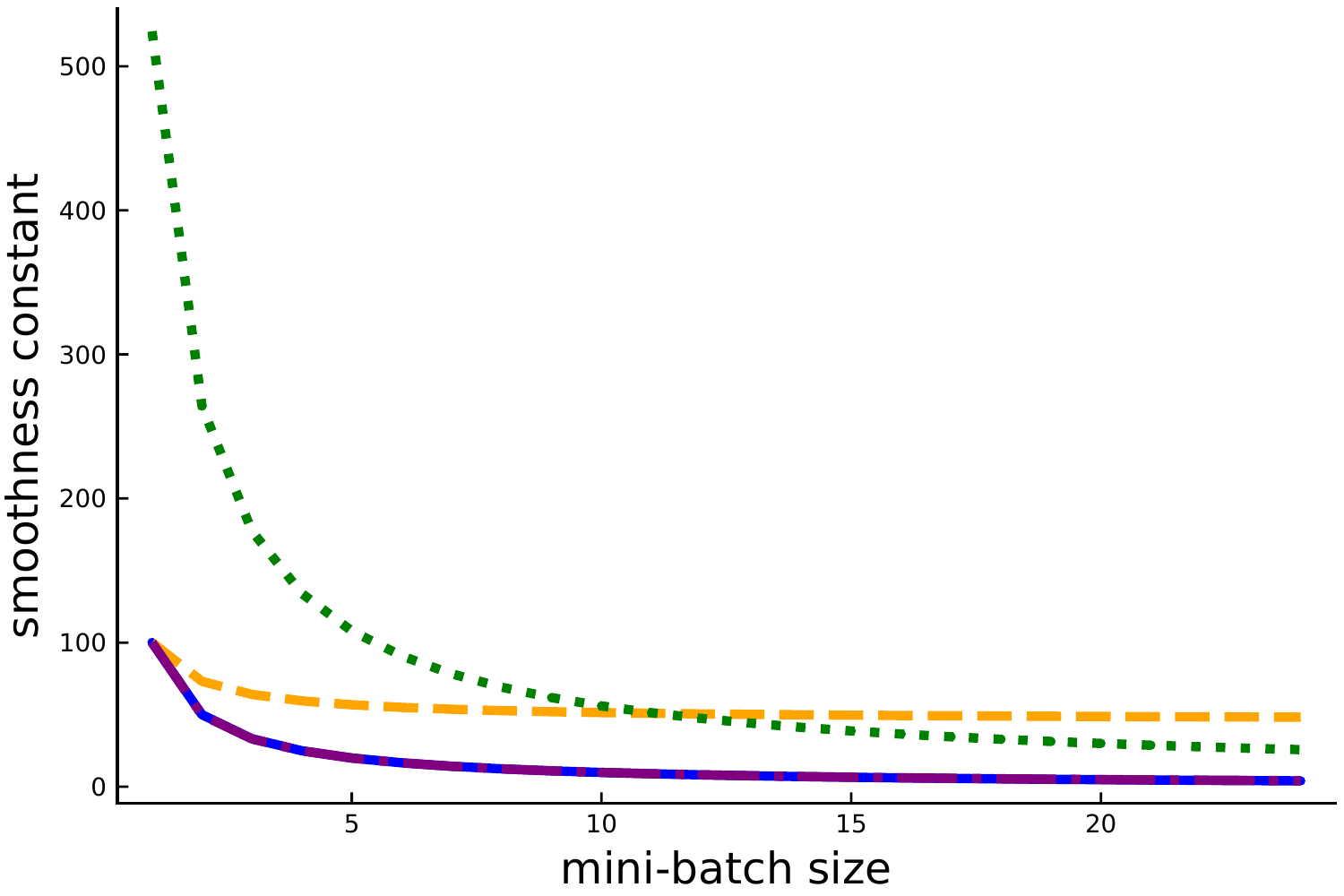}
        \caption{\emph{staircase eigval}.}
        \label{fig:exp1_general_conclusion_staircase}
        \end{subfigure}
        \centerline{\includegraphics[width=0.5\textwidth]{exp1/exp1-expsmoothbounds-legend-with-true}}
  \captionsetup{justification=centering}
  \caption{Expected smoothness constant $\cL$ and its upper-bounds the mini-batch size $b$ varies (unscaled datasets, $\lambda = 10^{-1}$).}
  \label{fig:exp1_general_conclusion}
  \end{center}
\end{figure}

We also report the influence of changing the value of the regularization parameter $\lambda$. \Cref{fig:exp1_influence_lambda} shows that this parameter has little impact on the general shape of the bounds and of $\cL$.

\begin{figure}[ht]
    \vskip 0.2in
    \begin{center}
        \begin{subfigure}[b]{0.31\textwidth}
          \includegraphics[width=\textwidth]{exp1/ridge_gauss-50-24-0_0_seed-1-none-regularizor-1e-01-exp1-expsmoothbounds}
          \caption{\emph{uniform}, $\lambda = 10^{-1}$.}
        \end{subfigure}
        \begin{subfigure}[b]{0.31\textwidth}
          \includegraphics[width=\textwidth]{exp1/ridge_diagalone-24-0_0-100_seed-1-none-regularizor-1e-01-exp1-expsmoothbounds}
          \caption{\emph{alone eigval}, $\lambda = 10^{-1}$.}
        \end{subfigure}
        \begin{subfigure}[b]{0.31\textwidth}
          \includegraphics[width=\textwidth]{exp1/ridge_diagints-24-0_0-100_seed-1-none-regularizor-1e-01-exp1-expsmoothbounds}
          \caption{\emph{staircase eigval}, $\lambda = 10^{-1}$.}
        \end{subfigure}
        \\
        \begin{subfigure}[b]{0.31\textwidth}
          \includegraphics[width=\textwidth]{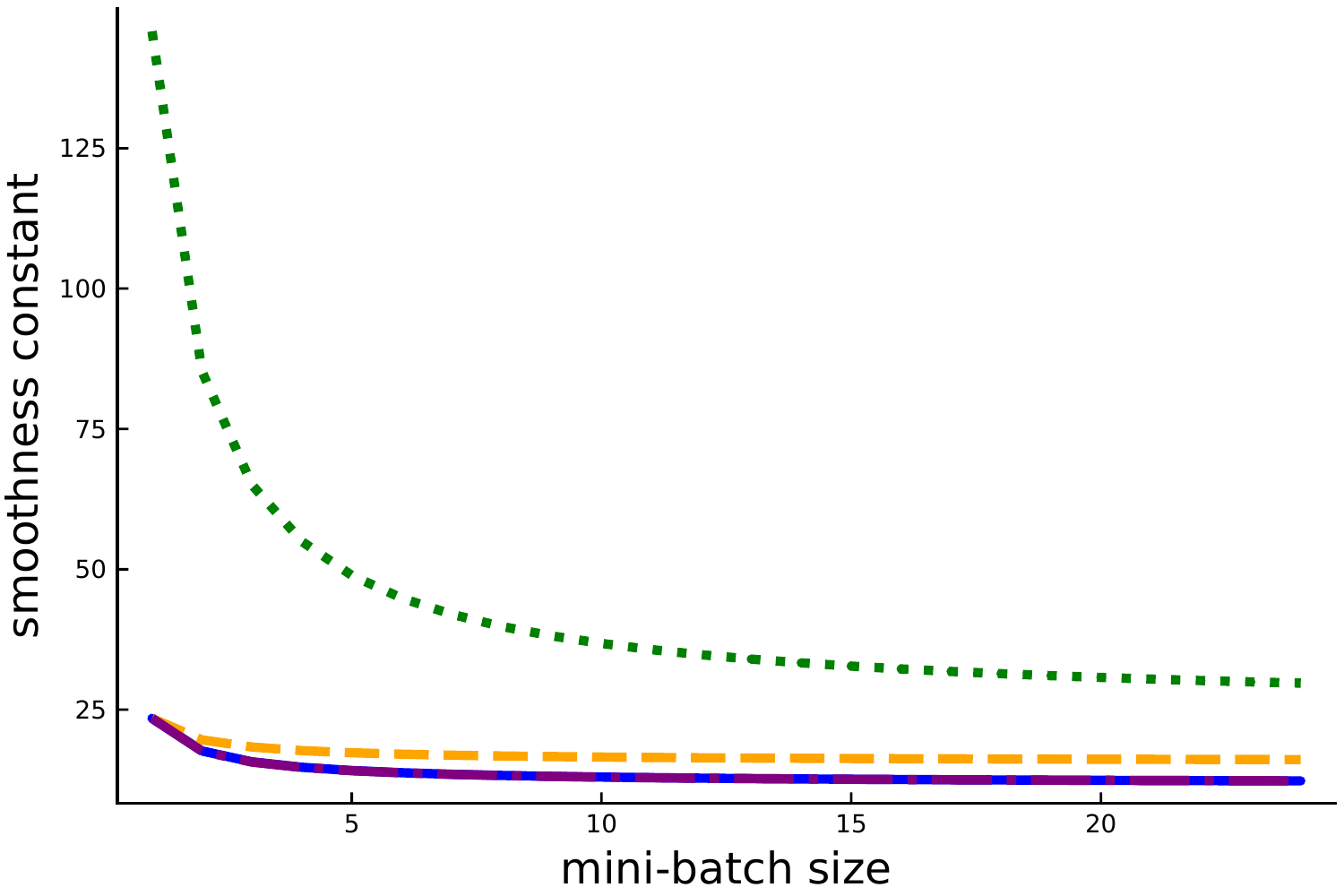}
          \caption{\emph{uniform}, $\lambda = 10^{-3}$.}
        \end{subfigure}
        \begin{subfigure}[b]{0.31\textwidth}
          \includegraphics[width=\textwidth]{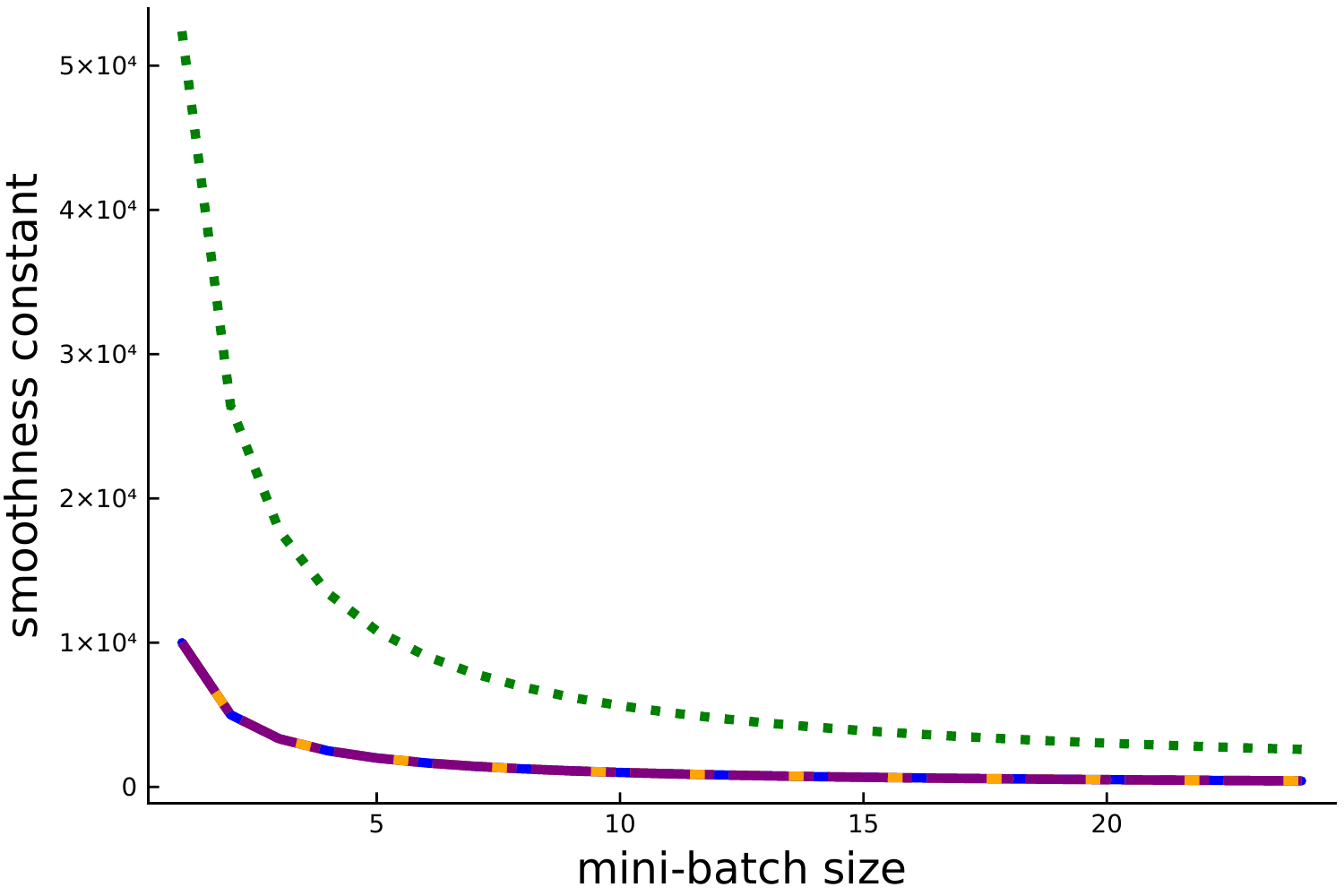}
          \caption{\emph{alone eigval}, $\lambda = 10^{-3}$.}
        \end{subfigure}
        \begin{subfigure}[b]{0.31\textwidth}
          \includegraphics[width=\textwidth]{exp1/ridge_diagints-24-0_0-100_seed-1-none-regularizor-1e-03-exp1-expsmoothbounds}
          \caption{\emph{staircase eigval}, $\lambda = 10^{-3}$.}
        \end{subfigure}
        \begin{subfigure}[b]{\textwidth}
          \centerline{\includegraphics[width=0.5\textwidth]{exp1/exp1-expsmoothbounds-legend-with-true}}
        \end{subfigure}
    \captionsetup{justification=centering}
    \caption{Expected smoothness constant $\cL$ and its upper-bounds as a function of the mini-batch size for unscaled datasets with $\lambda = 10^{-1}$ (top) and $\lambda = 10^{-3}$ (bottom).}
    \label{fig:exp1_influence_lambda}
    \end{center}
    \vskip -0.2in
\end{figure}

Finally, we study the impact of scaling or standardizing (\ie removing the mean and dividing by the standard deviation for each feature) our artificial datasets.
In order not to benefit from the diagonal shape of the \emph{alone eigval} and \emph{staircase eigval} datasets we also give examples of the bounds of $\cL$ after a rotation of the data.
The rotation aims at preserving the spectrum while erasing the diagonal structure of the covariance matrix $A A^\top$.
This rotation procedure consists in transforming $A$ into $Q^\top A Q$, where $Q$ is the orthogonal matrix given by the QR decomposition of a random squared matrix (with dimension the same as the one of $A$) with uniformly random coefficients $M$, such that $M = QR$.

We observe in \Cref{fig:difference_not_rotated_and_rotated_data} that rotations do not affect our estimates of $\cL$, because they preserve the spectrum.
Scaling non-diagonal datasets does not change the general shape neither.
As predicted, scaling diagonal matrices leads to a particular case where the spectrum of the covariance matrix is flattened and for all $i \in [n]$, $L_i \approx L_{\max} \approx \Lbarconst$.
This is why we get a flat \emph{simple bound} in \Cref{fig:staircase_scaled_flat_simple,fig:alone_scaled_flat_simple}.
Even after those different types of preprocessing (rotation and scaling) and with different values of $\lambda$, we end up with the same strong observation that the \emph{practical estimate} is a very sharp approximation of the expected smoothness constant.

\begin{figure}[ht]
  \vskip 0.2in
  \begin{center}
      \begin{subfigure}[b]{0.4\textwidth}
        \includegraphics[width=\textwidth]{exp1/ridge_diagints-24-0_0-100_seed-1-none-regularizor-1e-03-exp1-expsmoothbounds}
        \caption{\emph{staircase eigval}}
      \end{subfigure}
      \begin{subfigure}[b]{0.4\textwidth}
        \includegraphics[width=\textwidth]{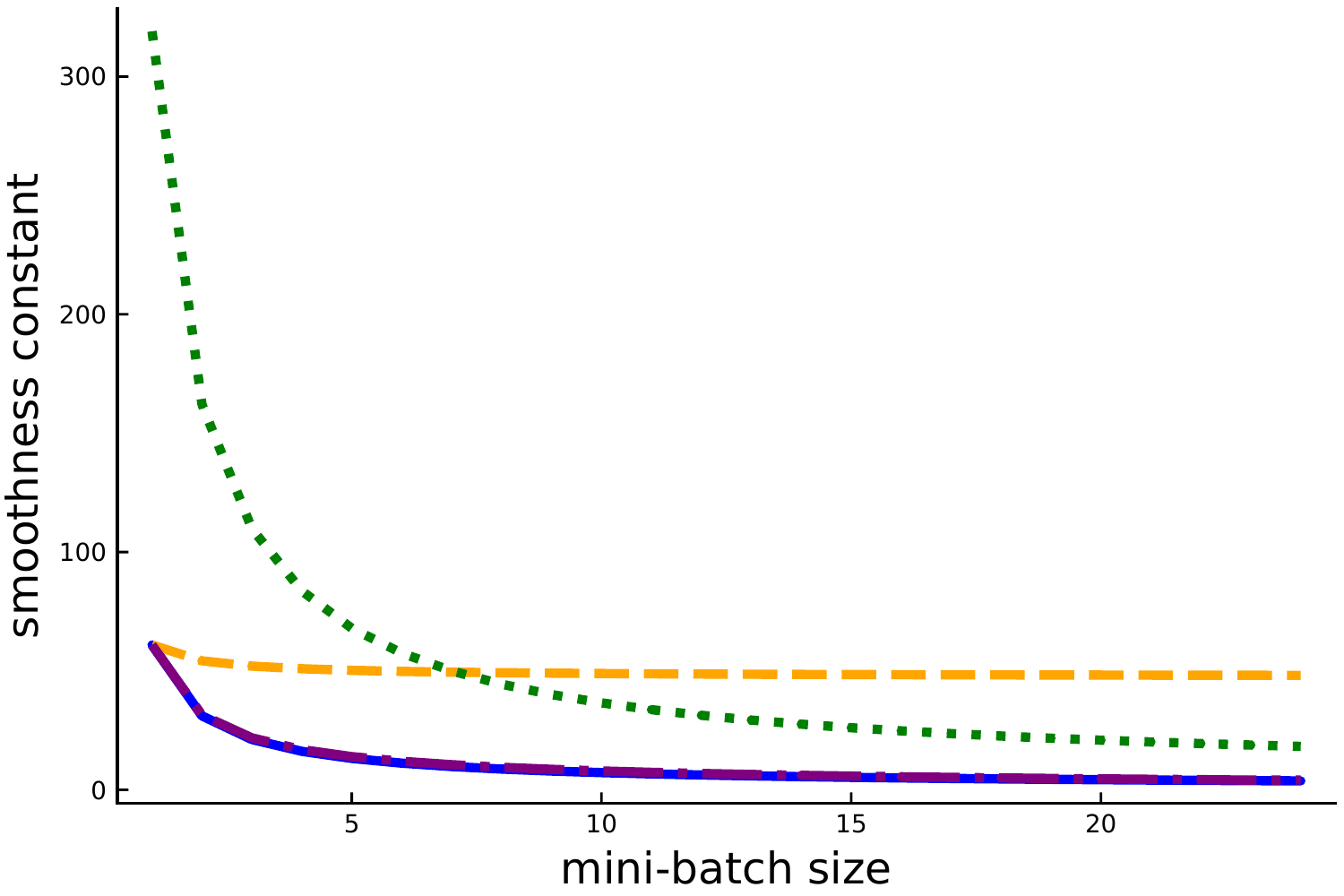}
        \caption{\emph{staircase eigval} rotated}
      \end{subfigure}\\
      \begin{subfigure}[b]{0.4\textwidth}
        \includegraphics[width=\textwidth]{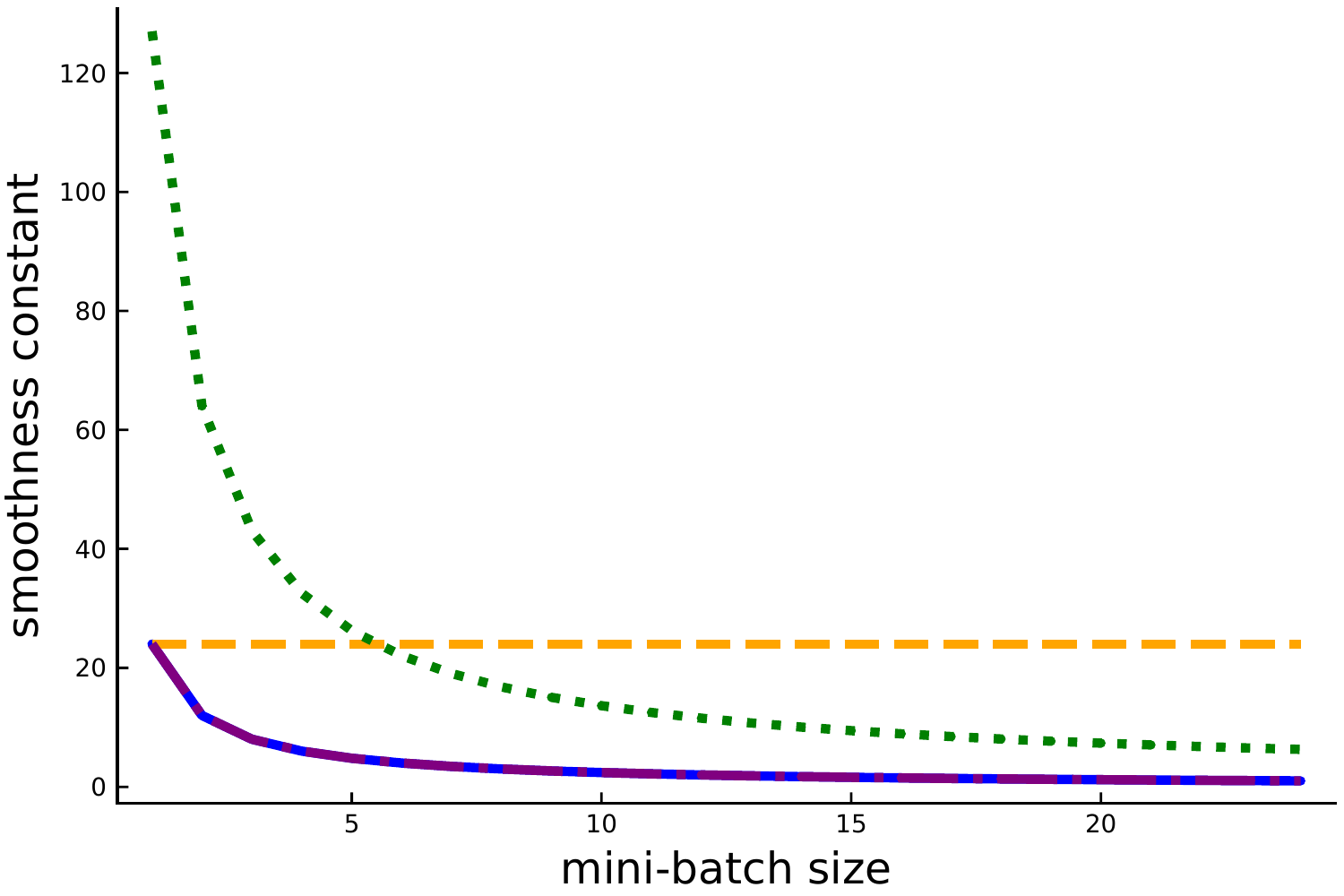}
        \caption{\emph{staircase eigval} scaled}
        \label{fig:staircase_scaled_flat_simple}
      \end{subfigure}
      \begin{subfigure}[b]{0.4\textwidth}
        \includegraphics[width=\textwidth]{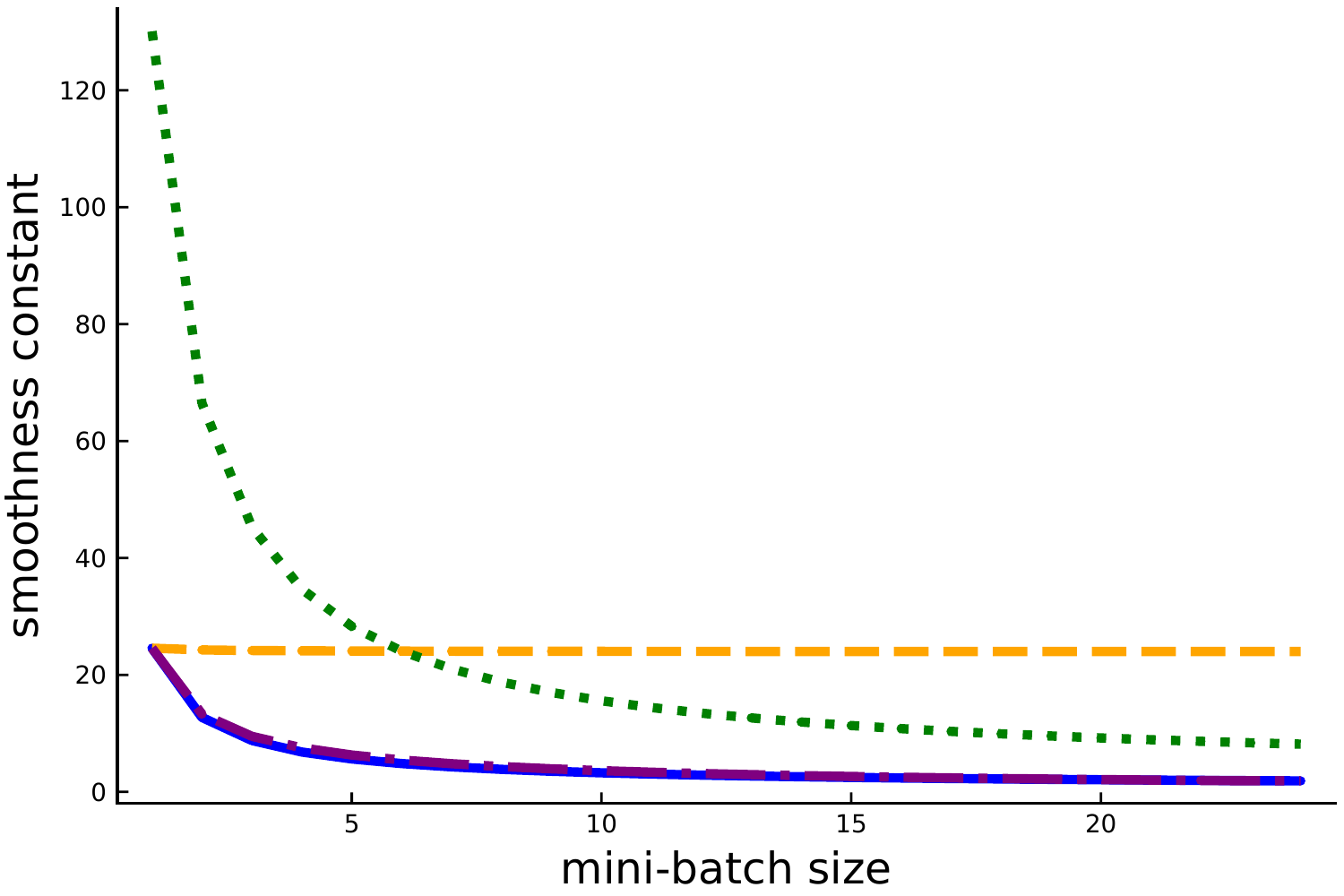}
        \caption{\emph{staircase eigval} rotated then scaled}
      \end{subfigure}\\
      \begin{subfigure}[b]{0.4\textwidth}
        \includegraphics[width=\textwidth]{exp1/ridge_diagalone-24-0_0-100_seed-1-none-regularizor-1e-03-exp1-expsmoothbounds}
        \caption{\emph{alone eigval}}
      \end{subfigure}
      \begin{subfigure}[b]{0.4\textwidth}
        \includegraphics[width=\textwidth]{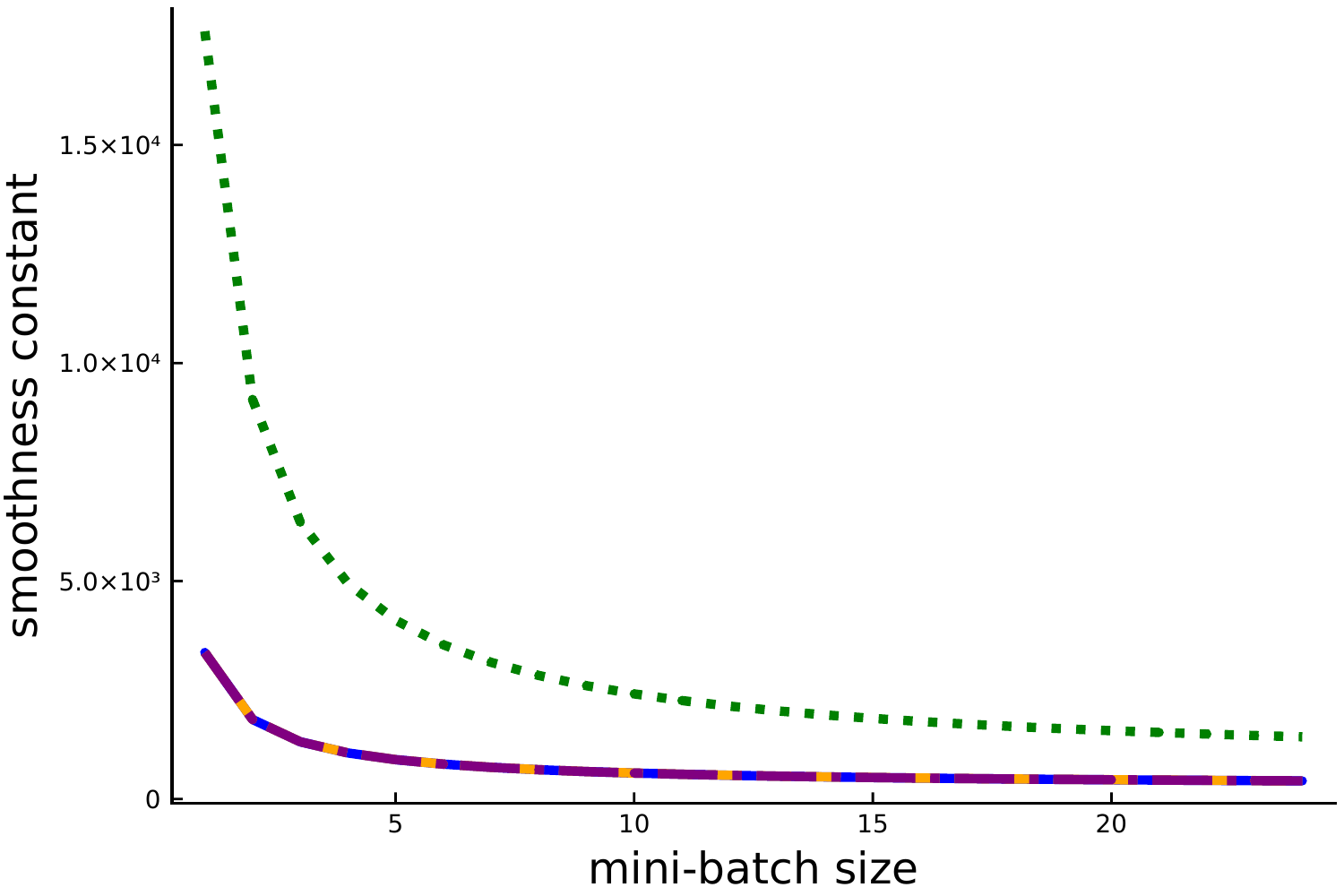}
        \caption{\emph{alone eigval} rotated}
      \end{subfigure}\\
      \begin{subfigure}[b]{0.4\textwidth}
        \includegraphics[width=\textwidth]{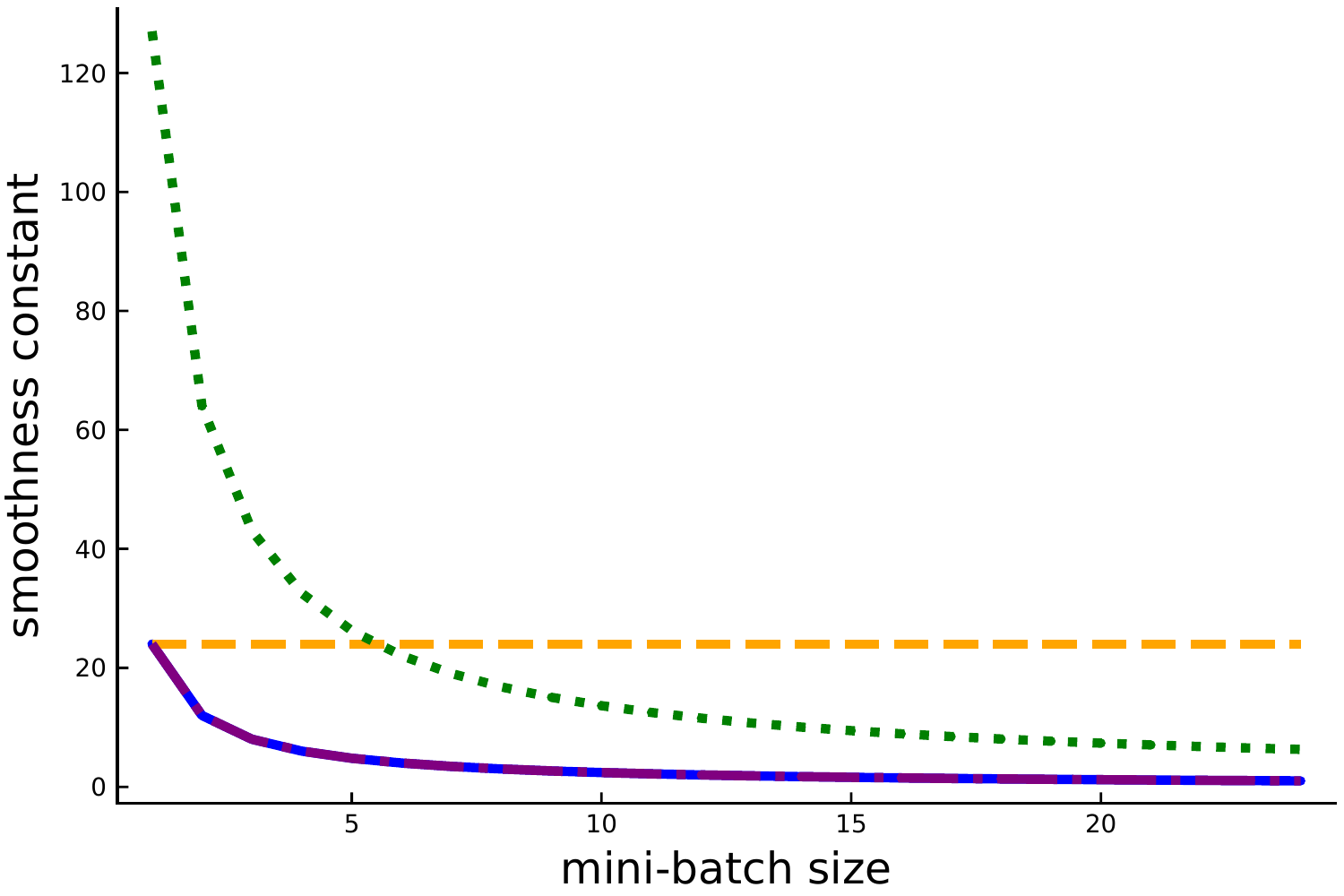}
        \caption{\emph{alone eigval} scaled}
        \label{fig:alone_scaled_flat_simple}
      \end{subfigure}
      \begin{subfigure}[b]{0.4\textwidth}
        \includegraphics[width=\textwidth]{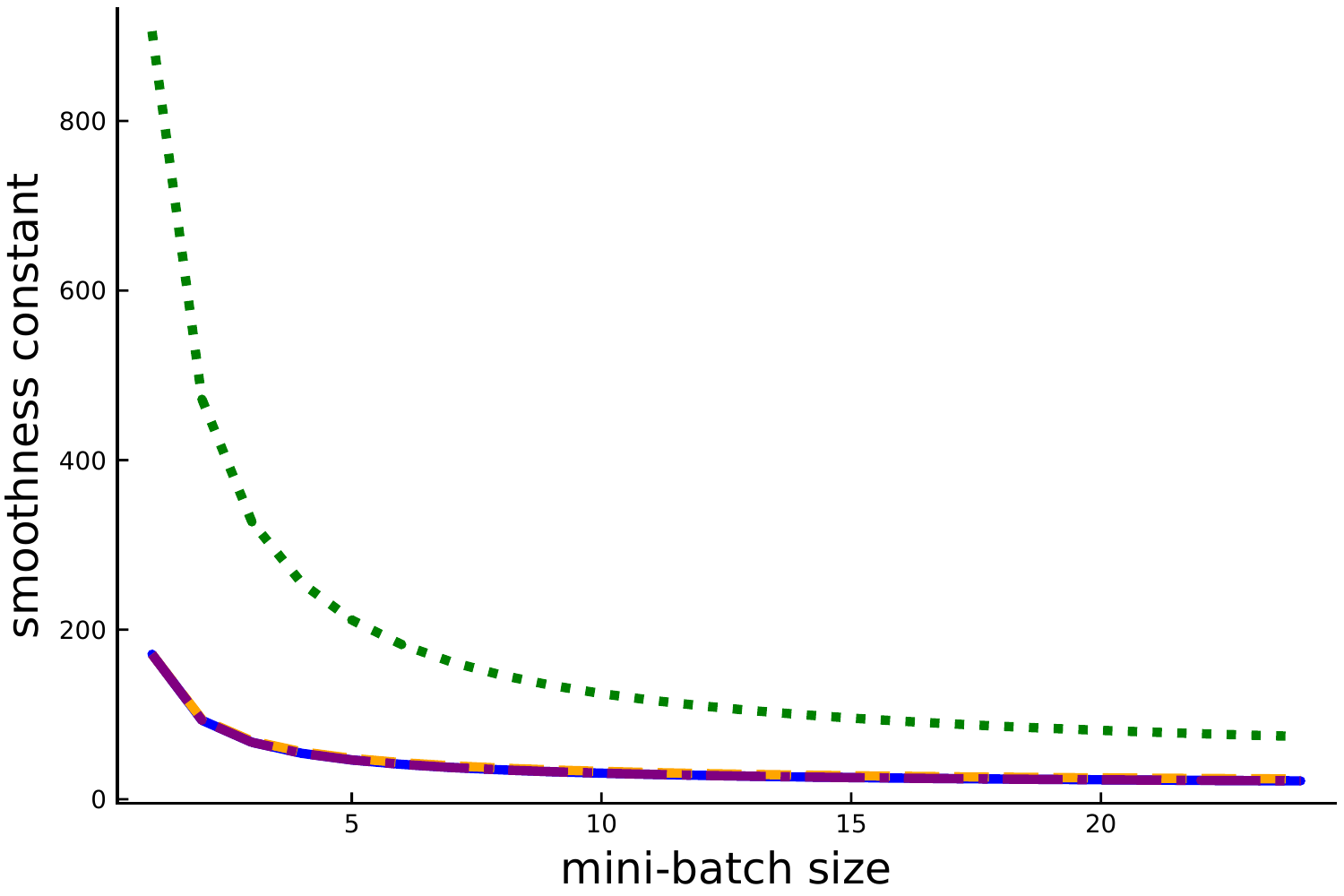}
        \caption{\emph{alone eigval} rotated then scaled}
      \end{subfigure}\\
      \begin{subfigure}[b]{\textwidth}
        \centerline{\includegraphics[width=0.5\textwidth]{exp1/exp1-expsmoothbounds-legend-with-true}}
      \end{subfigure}
  \caption{Upper-bounds of the expected smoothness constant $\cL$ for non-rotated (left) and rotated (right) datasets ($\lambda = 10^{-3}$).}
  \label{fig:difference_not_rotated_and_rotated_data}
  \end{center}
  \vskip -0.2in
\end{figure}

\subsection{Experiment 1: estimates of the expected smoothness constant for real datasets}
\label{appendix:exp_1_real_datasets}
In what folows, we also used publicly available datasets from LIBSVM\footnote{\tiny{\url{https://www.csie.ntu.edu.tw/~cjlin/libsvmtools/datasets/}}} provided by \citet{chang2011libsvm} and from the UCI repository\footnote{\tiny\url{https://archive.ics.uci.edu/ml/datasets/}} provided by \citet{Dua:2017}. We applied ridge regression to the following datasets: \textit{YearPredictionMSD} $(n=515,345,d=90)$ from LIBSVM and \textit{slice} $(n=53,500,d=384)$ from UCI. We also applied regularized logistic regression for binary classification on \textit{ijcnn1} $(n=141,691,d=22)$, \textit{covtype.binary} $(n=581,012,d=54)$, \textit{real-sim} $(n=72,309,d=20,958)$, \textit{rcv1.binary} $(n=697,641,d=47,236)$ and \textit{news20.binary} ($n=19,996,d=1,355,191$) from LIBSVM. When a test set was available, we concatenated it with the train set to have more samples.

One can observe in \Cref{fig:exp_1_real_datasets}, that for unscaled datasets the \emph{Bernstein bound} performs better than the \emph{simple bound}, except for \textit{YearPredictionMSD} $(n=515,345,d=90)$ and \textit{covtype.binary} $(n=581,012,d=54)$. From \Cref{fig:exp_1_real_datasets_scaled}, we observe that after feature-scaling, the \emph{Bernstein bound} is always a tighter upper bound of $\cL$ than the \emph{simple bound}.

\begin{figure}[ht]
  \vskip 0.2in
  \begin{center}
    \begin{subfigure}[b]{0.35\textwidth}
      \includegraphics[width=\textwidth]{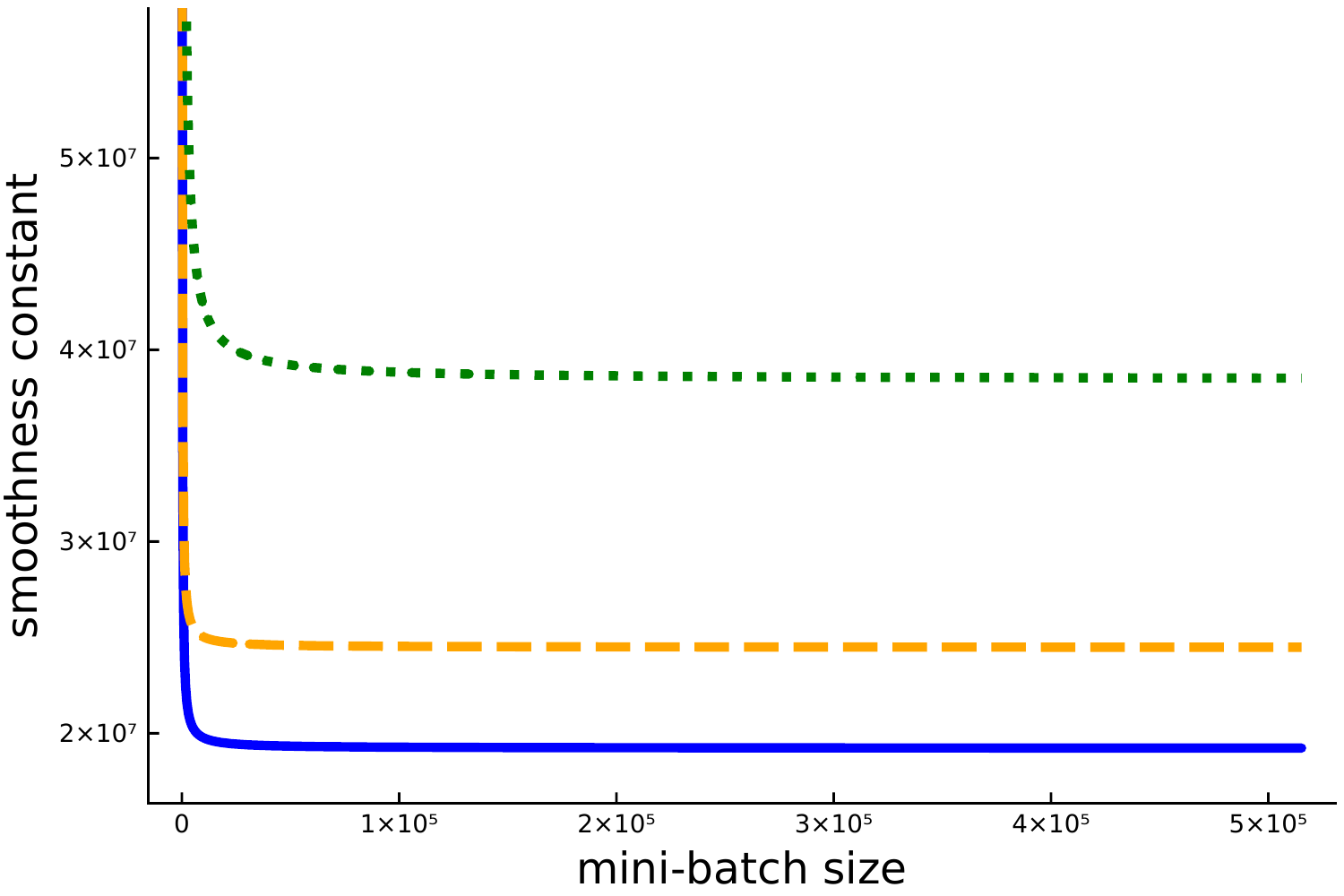}
      \caption{\textit{YearPredictionMSD}}
    \end{subfigure}%
    \begin{subfigure}[b]{0.35\textwidth}
      \includegraphics[width=\textwidth]{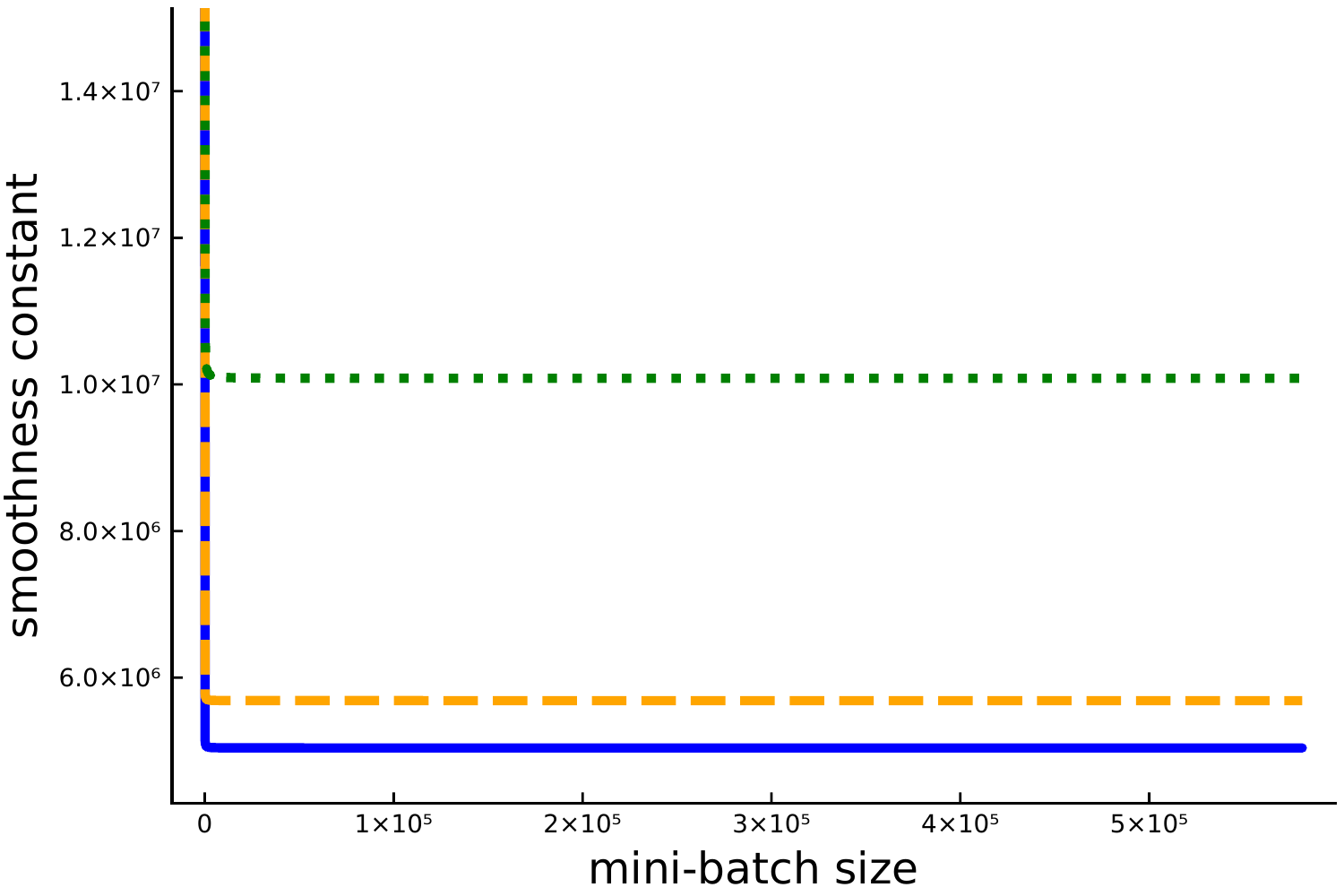}
      \caption{\textit{covtype.binary}}
    \end{subfigure}\\
    \begin{subfigure}[b]{0.35\textwidth}
      \includegraphics[width=\textwidth]{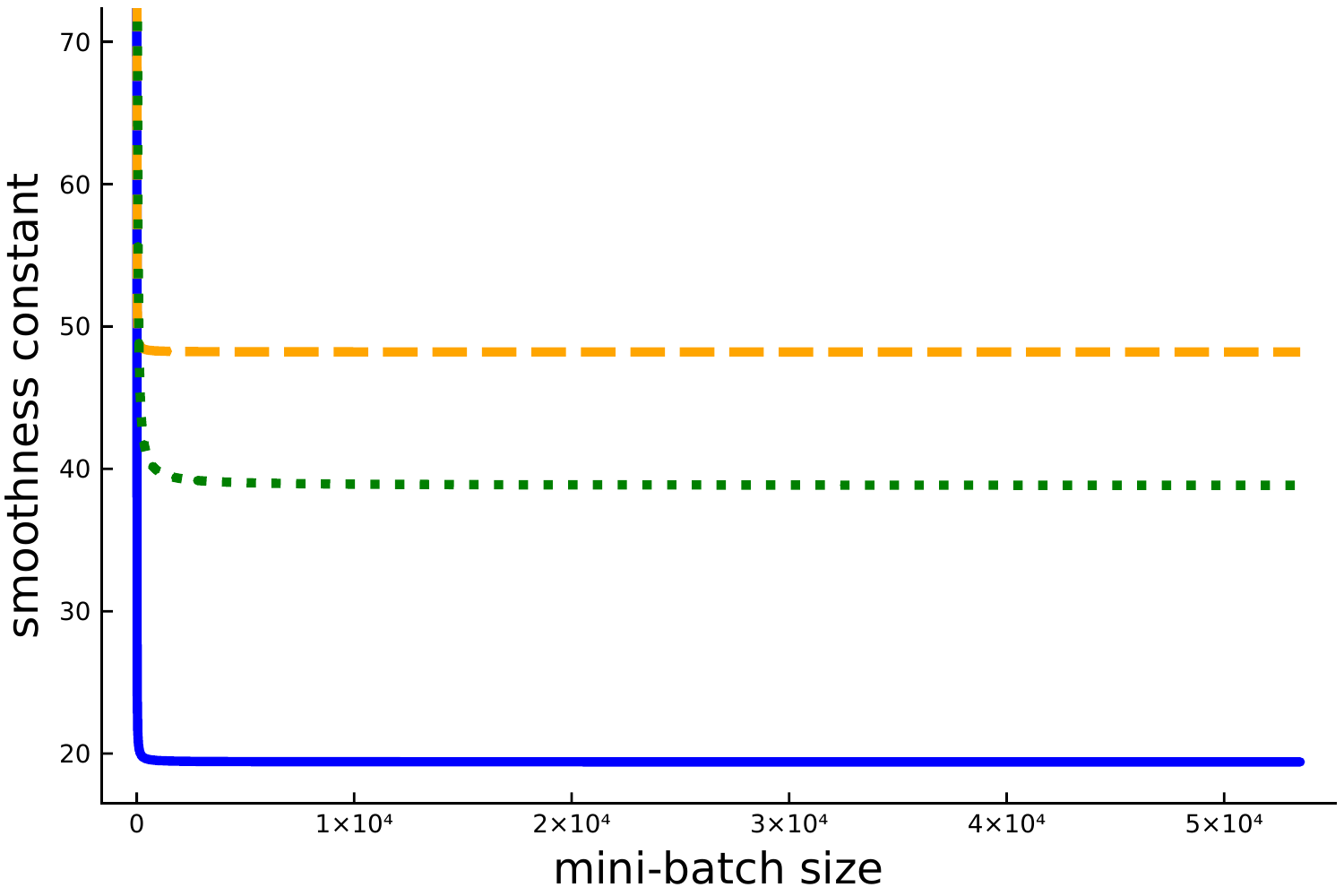}
      \caption{\textit{slice}}
    \end{subfigure}%
    \begin{subfigure}[b]{0.35\textwidth}
      \includegraphics[width=\textwidth]{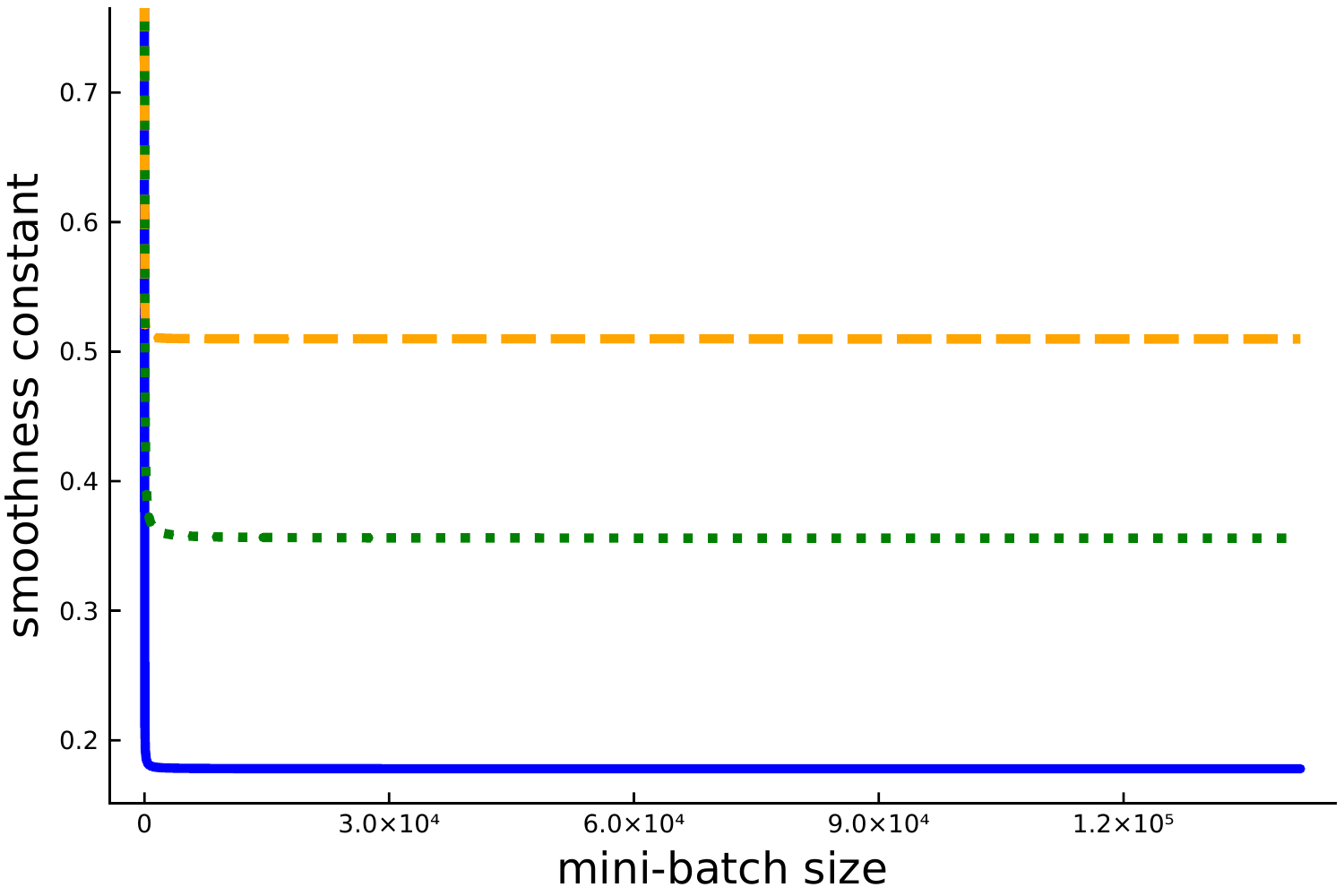}
      \caption{\textit{ijcnn1}}
    \end{subfigure}\\
    \begin{subfigure}[b]{0.35\textwidth}
      \includegraphics[width=\textwidth]{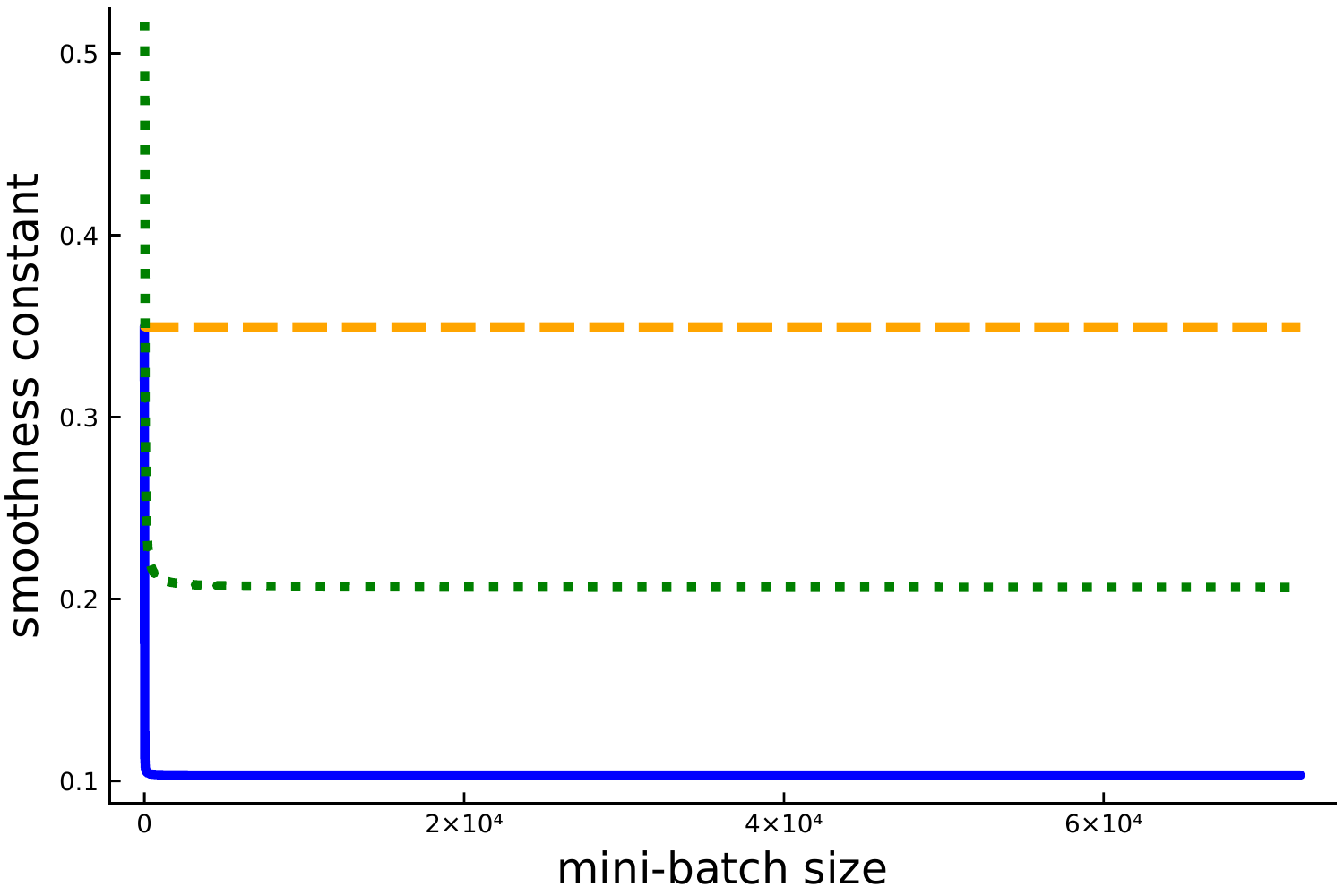}
      \caption{\textit{real-sim}}
    \end{subfigure}%
    \begin{subfigure}[b]{0.35\textwidth}
      \includegraphics[width=\textwidth]{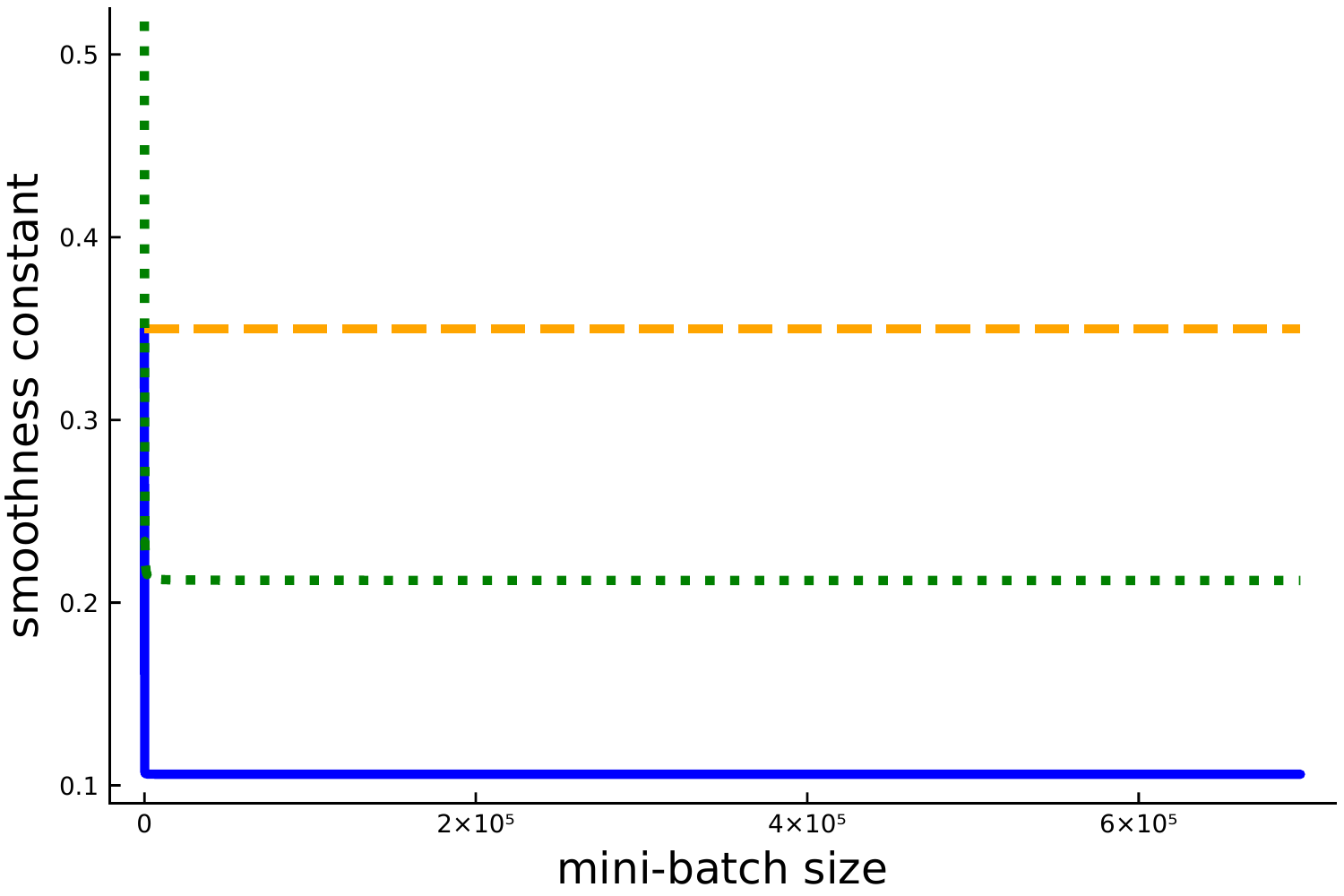}
      \caption{\textit{rcv1.binary}}
    \end{subfigure}\\
    \begin{subfigure}[b]{0.35\textwidth}
      \includegraphics[width=\textwidth]{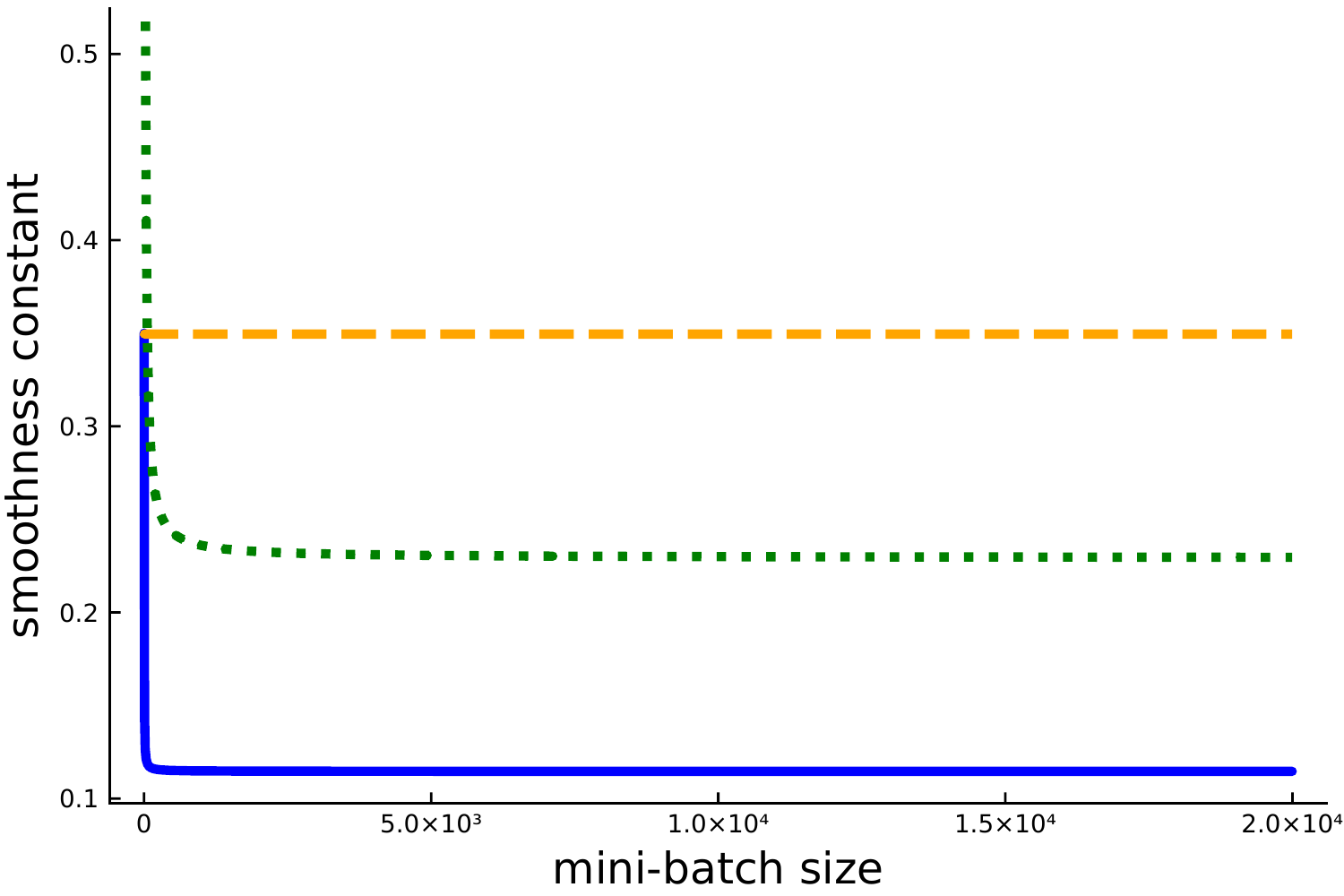}
      \caption{\textit{news20.binary}}
    \end{subfigure}\\
    \begin{subfigure}[b]{\textwidth}
      \centerline{\includegraphics[width=0.5\textwidth]{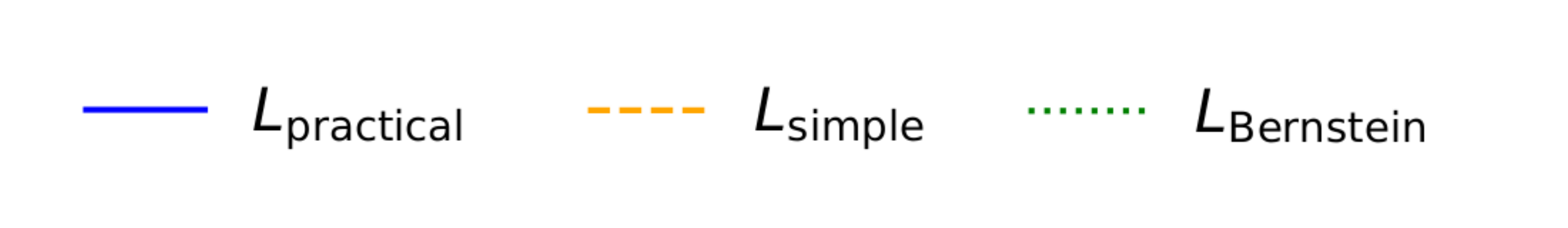}}
    \end{subfigure}
  \caption{Upper-bounds of the expected smoothness constant for real unscaled datasets ($\lambda = 10^{-1}$).}
  \label{fig:exp_1_real_datasets}
  \end{center}
  \vskip -0.2in
\end{figure}

\begin{figure}[ht]
  \vskip 0.2in
  \begin{center}
    \begin{subfigure}[b]{0.35\textwidth}
      \includegraphics[width=\textwidth]{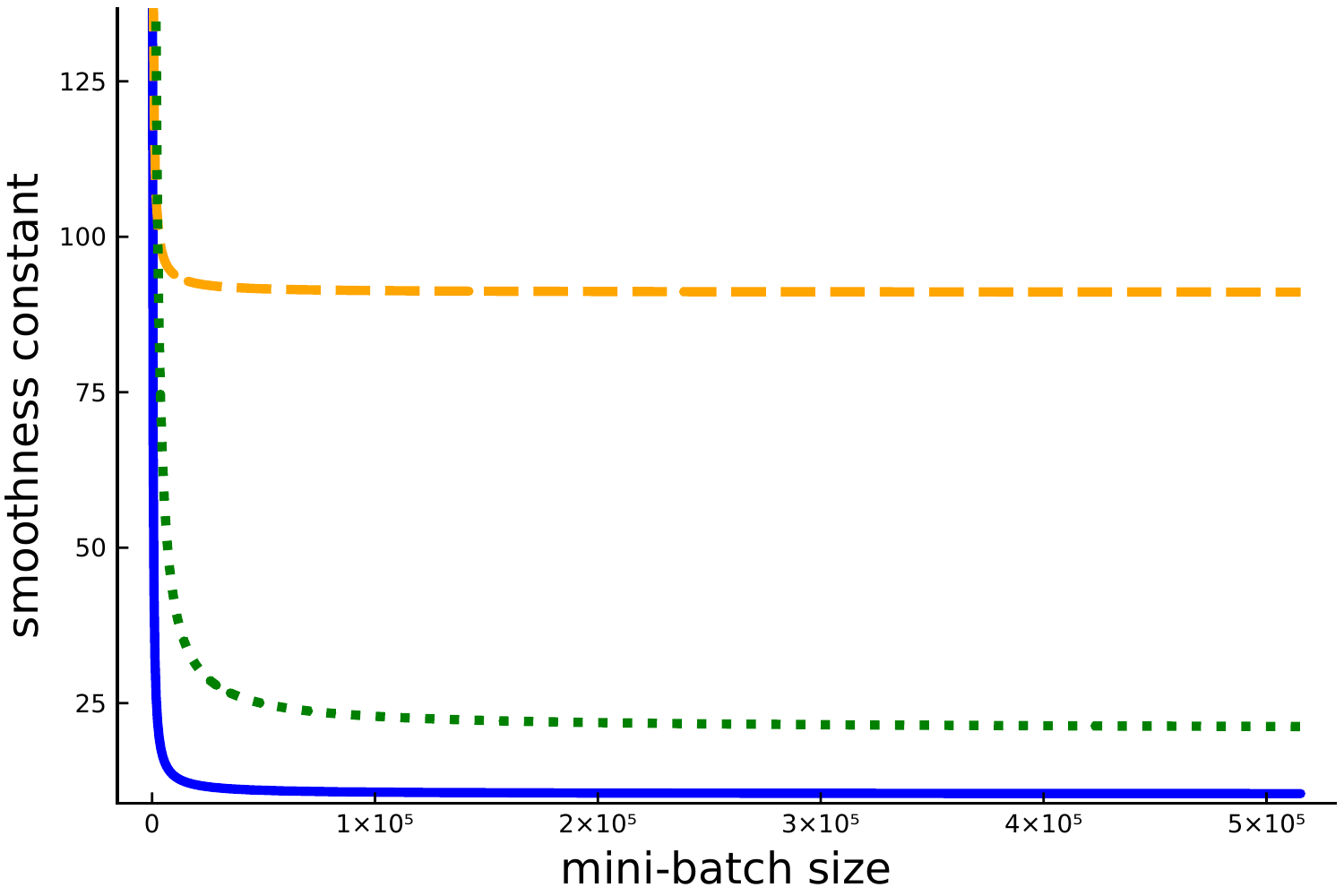}
      \caption{ \textit{YearPredictionMSD}}
    \end{subfigure}%
    \begin{subfigure}[b]{0.35\textwidth}
      \includegraphics[width=\textwidth]{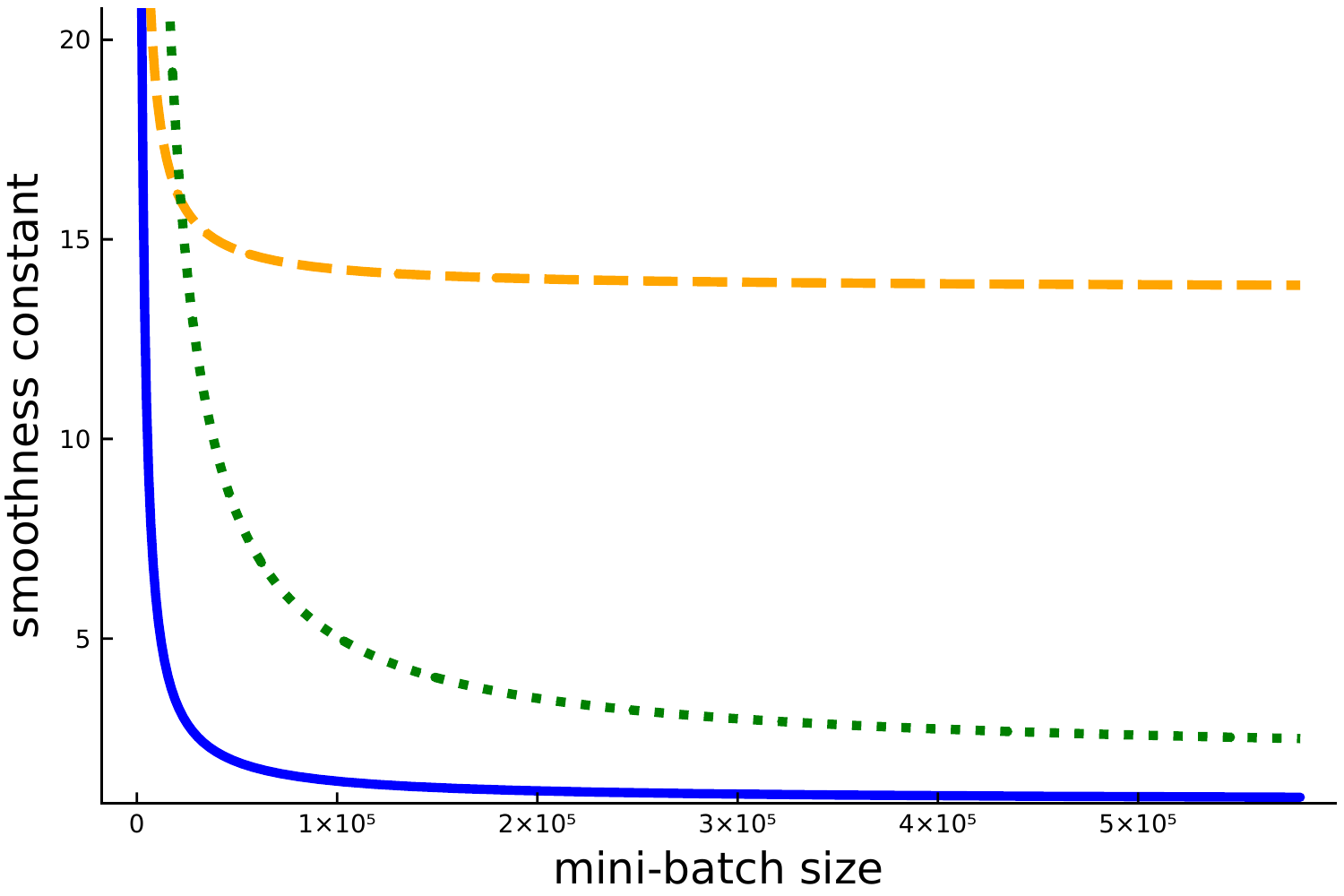}
      \caption{\textit{covtype.binary}}
    \end{subfigure}\\
    \begin{subfigure}[b]{0.35\textwidth}
      \includegraphics[width=\textwidth]{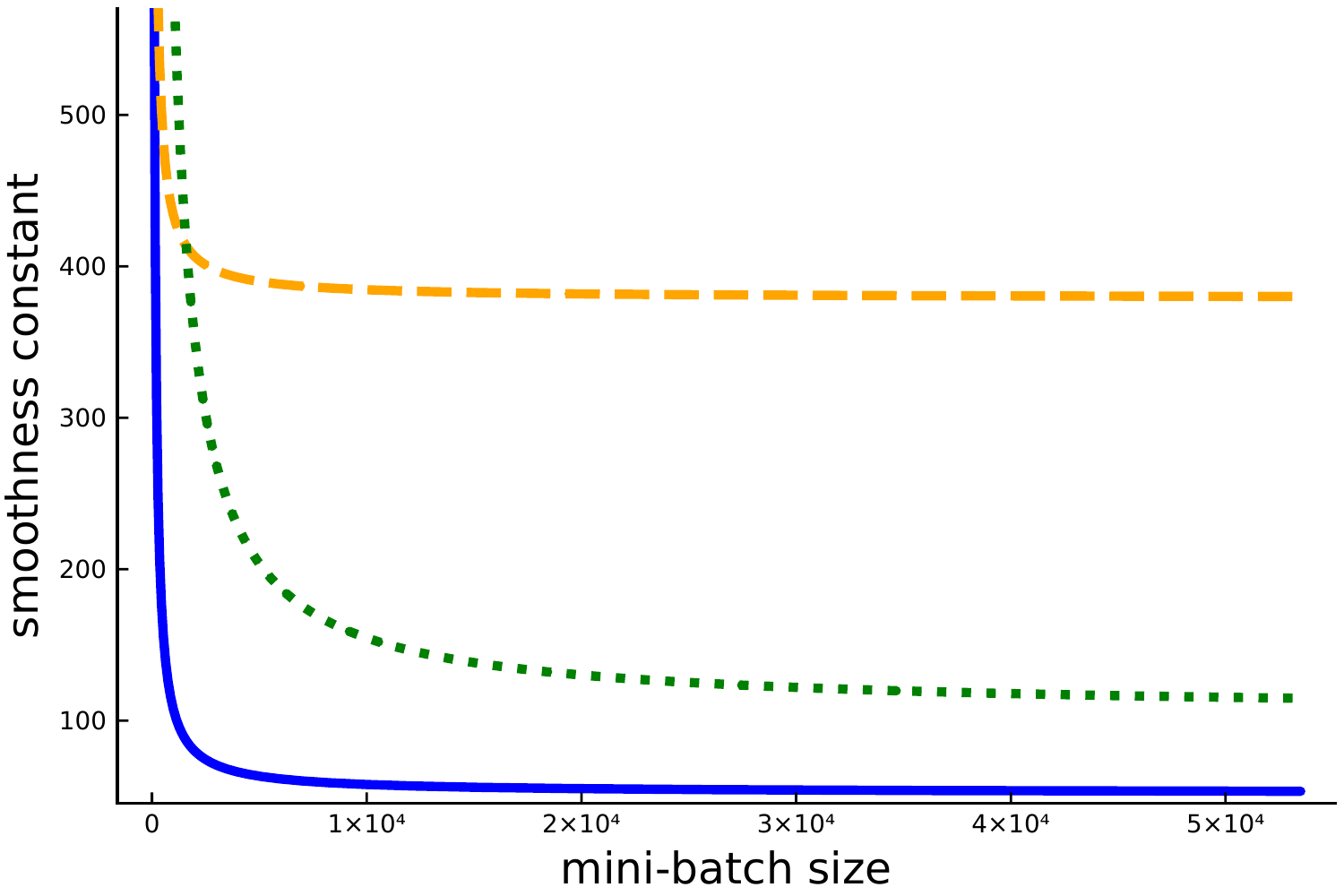}
      \caption{\textit{slice}}
    \end{subfigure}%
    \begin{subfigure}[b]{0.35\textwidth}
      \includegraphics[width=\textwidth]{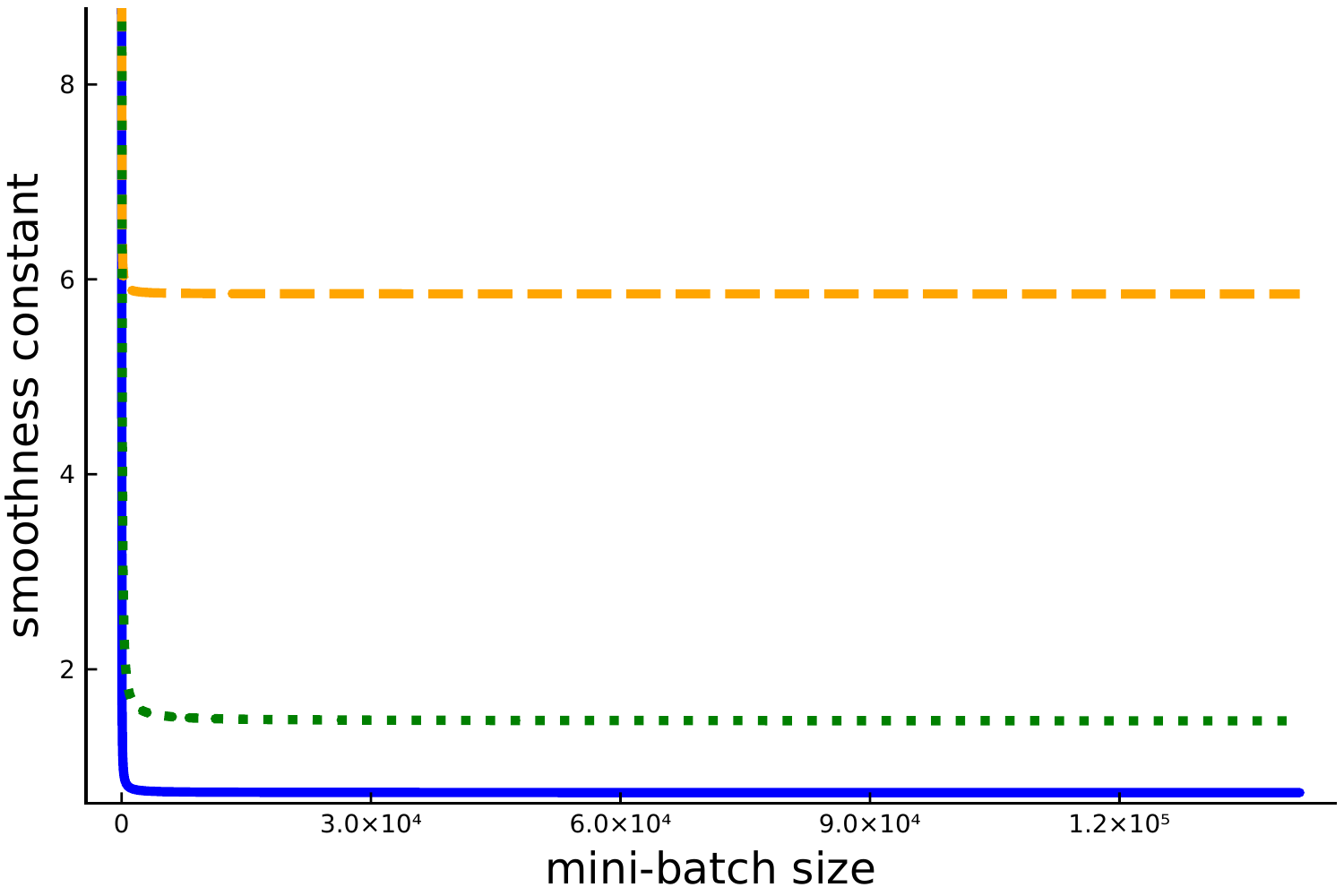}
      \caption{\textit{ijcnn1}}
    \end{subfigure}
    \begin{subfigure}[b]{\textwidth}
      \centerline{\includegraphics[width=0.5\textwidth]{exp1/exp1-expsmoothbounds-legend}}
    \end{subfigure}
  \caption{Upper-bounds of the expected smoothness constant of $\cL$ for real feature-scaled datasets ($\lambda =10^{-1}$).}
  \label{fig:exp_1_real_datasets_scaled}
  \end{center}
  \vskip -0.2in
\end{figure}


\subsection{Experiment 2: step size estimates for artificial datasets}
\label{appendix:exp_2_artificial_datasets}

In this section we give the step sizes estimate corresponding to the expected smoothness constant, the \emph{simple} and \emph{Bernstein} upper-bounds and the \emph{practical} estimate for our small artificial datasets. In \Cref{fig:exp2_stepsizes_artificial_unscaled_data}, we show that the \emph{practical} step size estimate is larger than all others. Moreover, for except for small value sof $b$, our $\gamma_{\text{simple}}$ or $\gamma_{\text{Bernstein}}$ estimates are in most cases larger than the one proposed in \citep{hofmann2015variance}.

\begin{figure}[ht]
  \vskip 0.2in
  \begin{center}
      \begin{subfigure}[b]{0.33\textwidth}
        \includegraphics[width=\textwidth]{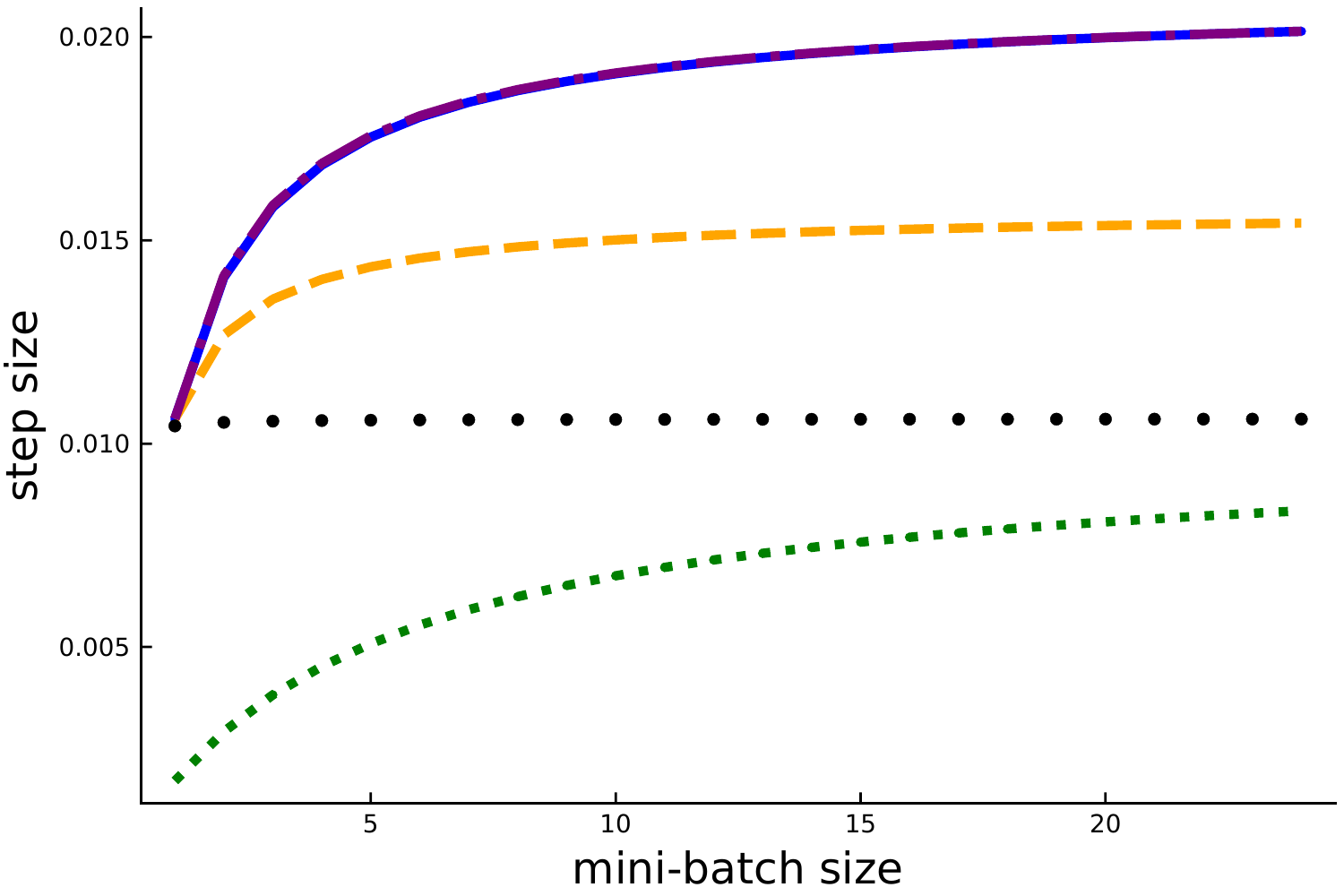}
        \caption{\emph{uniform}.}
      \end{subfigure}
      \begin{subfigure}[b]{0.33\textwidth}
        \includegraphics[width=\textwidth]{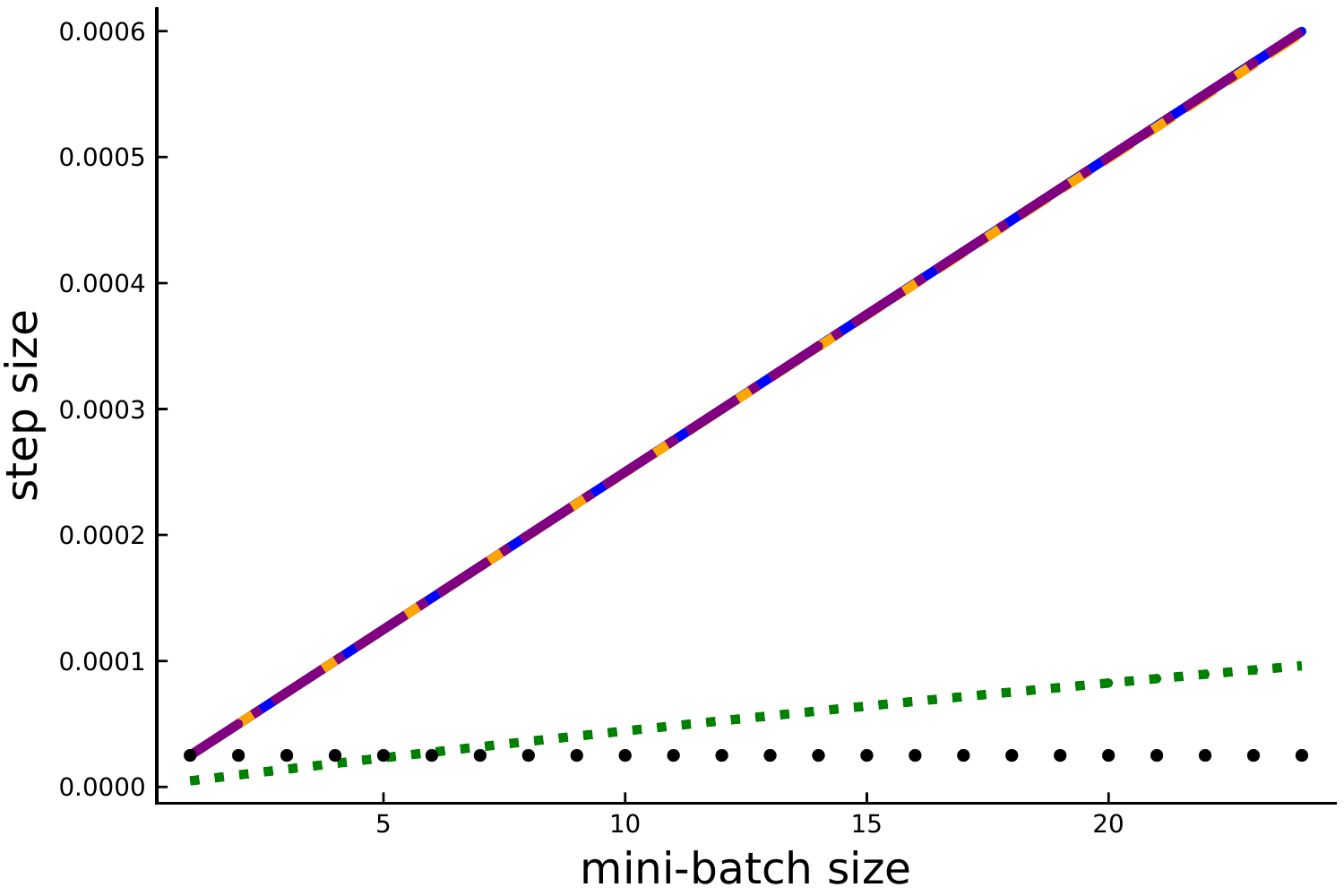}
        \caption{\emph{alone eigval}.}
      \end{subfigure}
      \begin{subfigure}[b]{0.33\textwidth}
        \includegraphics[width=\textwidth]{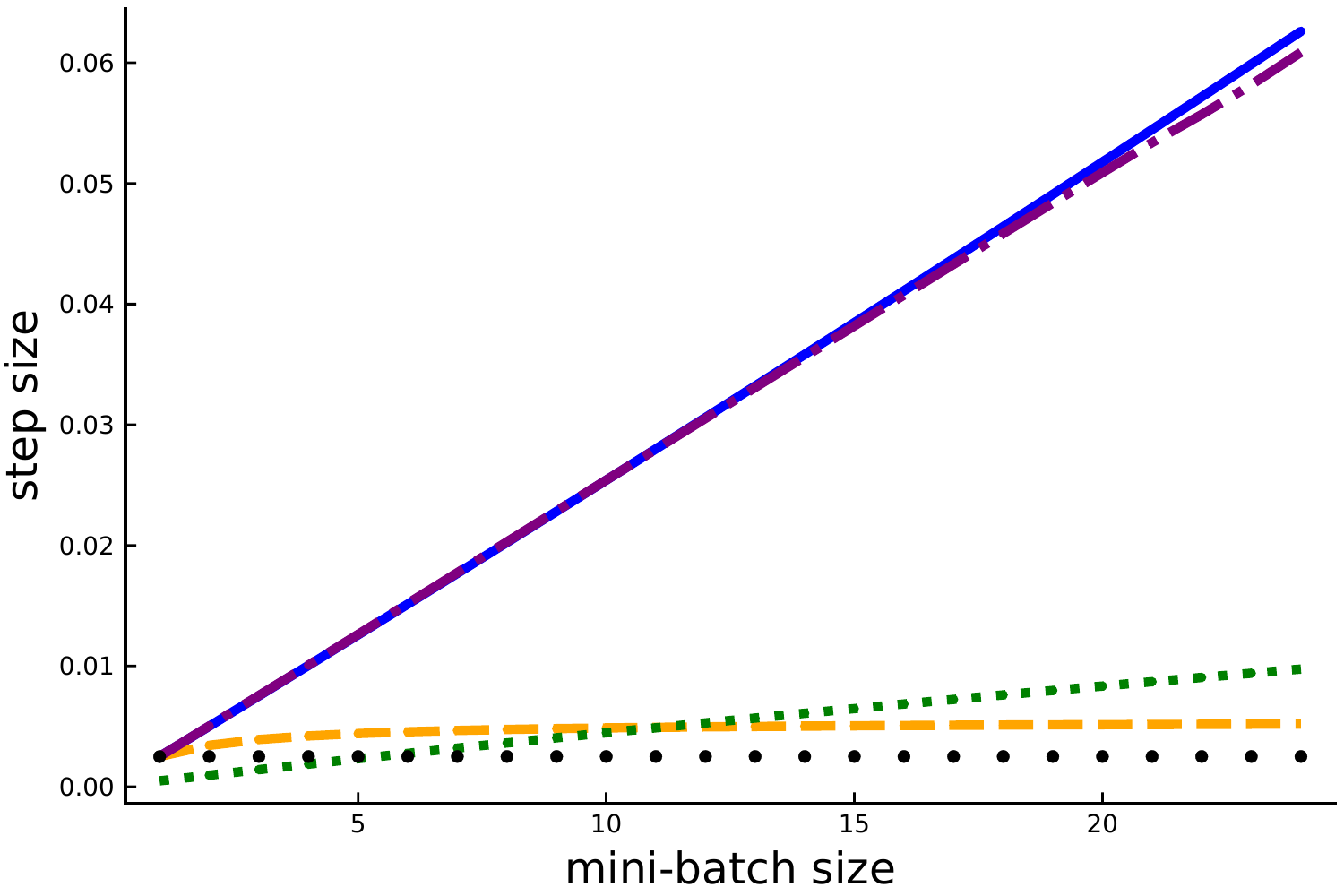}
        \caption{\emph{staircase eigval}.}
      \end{subfigure}\\
      \begin{subfigure}[b]{\textwidth}
        \centerline{\includegraphics[width=0.65\textwidth]{exp2/exp2-stepsizes-legend-with-true}}
      \end{subfigure}
  \caption{Step size estimates as a function the mini-batch size for unscaled artificial datasets ($\lambda = 10^{-1}$).}
  \label{fig:exp2_stepsizes_artificial_unscaled_data}
  \end{center}
  \vskip -0.2in
\end{figure}

\subsection{Experiment 2: step size estimates for real datasets}
\label{appendix:exp_2_real_datasets}

Here we show the step sizes estimate corresponding to the \emph{simple} and \emph{Bernstein} upper-bounds and the \emph{practical} estimate for real datasets detailed in \Cref{appendix:exp_1_real_datasets}. On these real data, unscaled in \Cref{fig:exp2_real_datasets_unscaled} and scaled in \Cref{fig:exp2_real_datasets_scaled}, we see that the gap between our step size estimates and $\gamma_{\text{Hofmann}}$ is even larger. We observe in \Cref{fig:exp2_real_datasets_unscaled}, accordlingly to previous remarks in \Cref{appendix:exp_1_real_datasets}, that \emph{Bernstein bound} leads to larger step sizes than the \emph{simple} one, except for the unscaled \textit{YearPredictionMSD} and \textit{covtype.binary} datasets. Yet, as noticed before, \Cref{fig:exp2_real_datasets_scaled} seems to show that scaling the data leads to $\gamma_{\text{Bernstein}}$ larger than $\gamma_{\text{simple}}$.

\begin{figure}[ht]
  \vskip 0.2in
  \begin{center}
    \begin{subfigure}[b]{0.35\textwidth}
      \includegraphics[width=\textwidth]{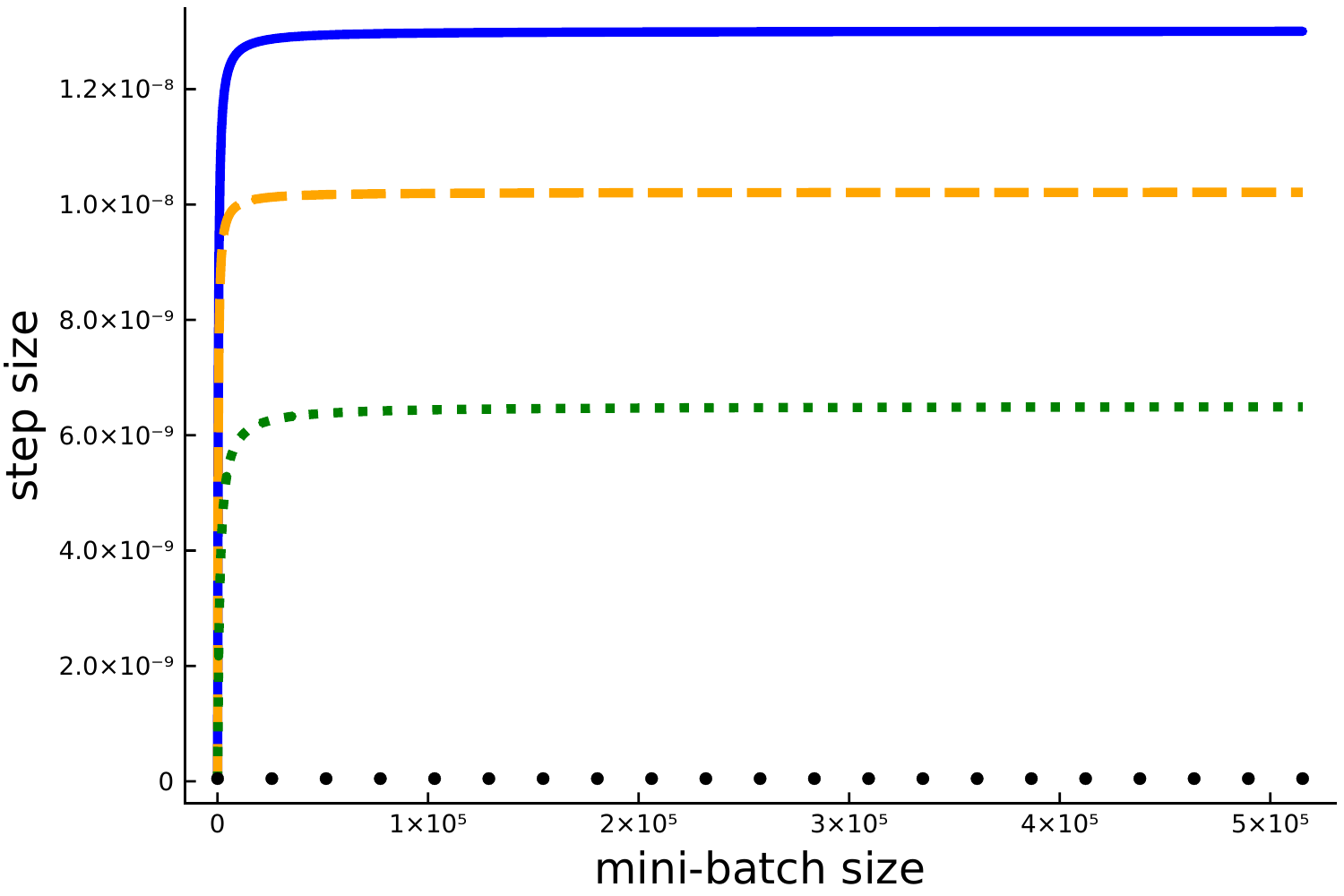}
      \caption{\textit{YearPredictionMSD}}
    \end{subfigure}%
    \begin{subfigure}[b]{0.35\textwidth}
      \includegraphics[width=\textwidth]{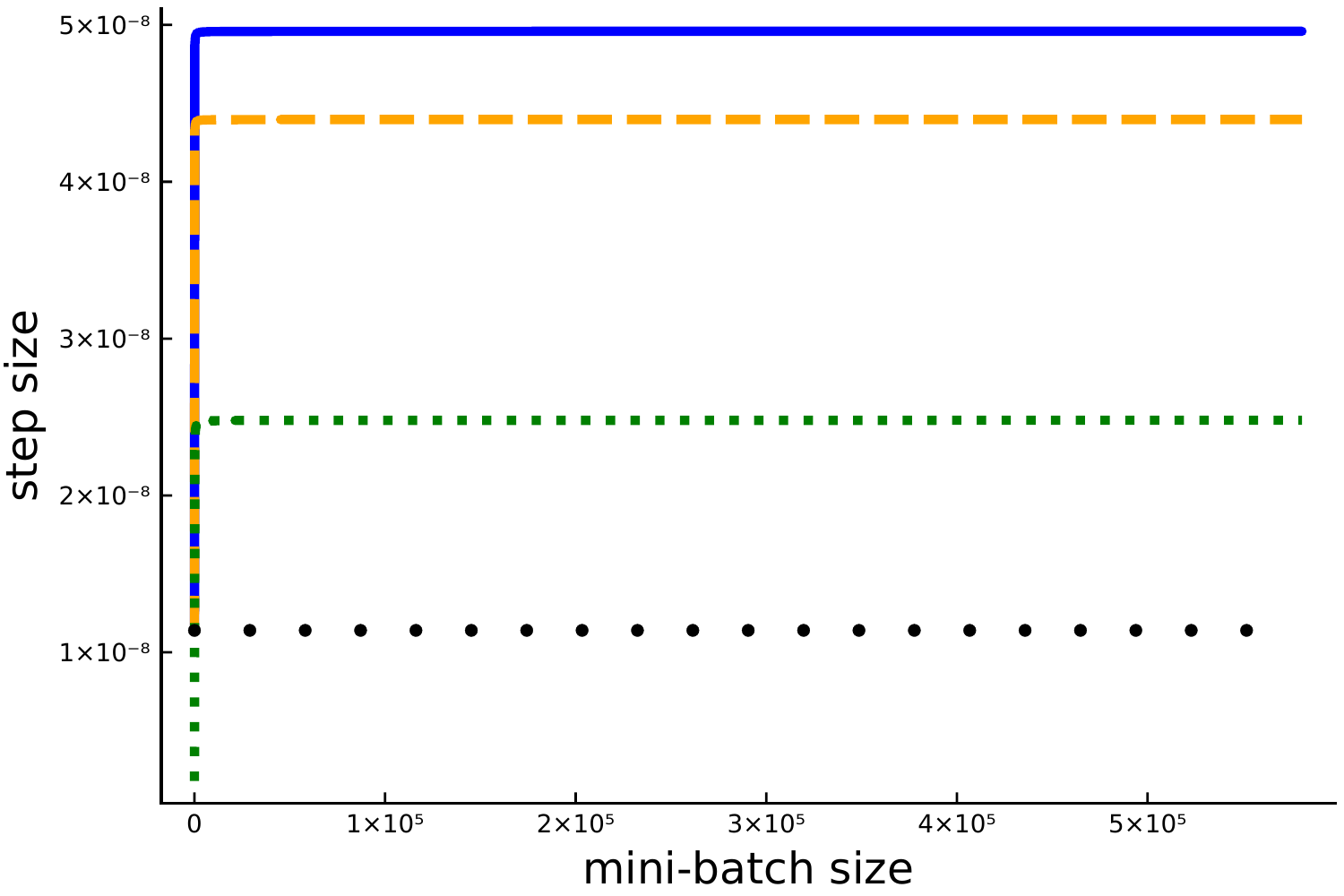}
      \caption{\textit{covtype.binary}}
    \end{subfigure}\\
    \begin{subfigure}[b]{0.35\textwidth}
      \includegraphics[width=\textwidth]{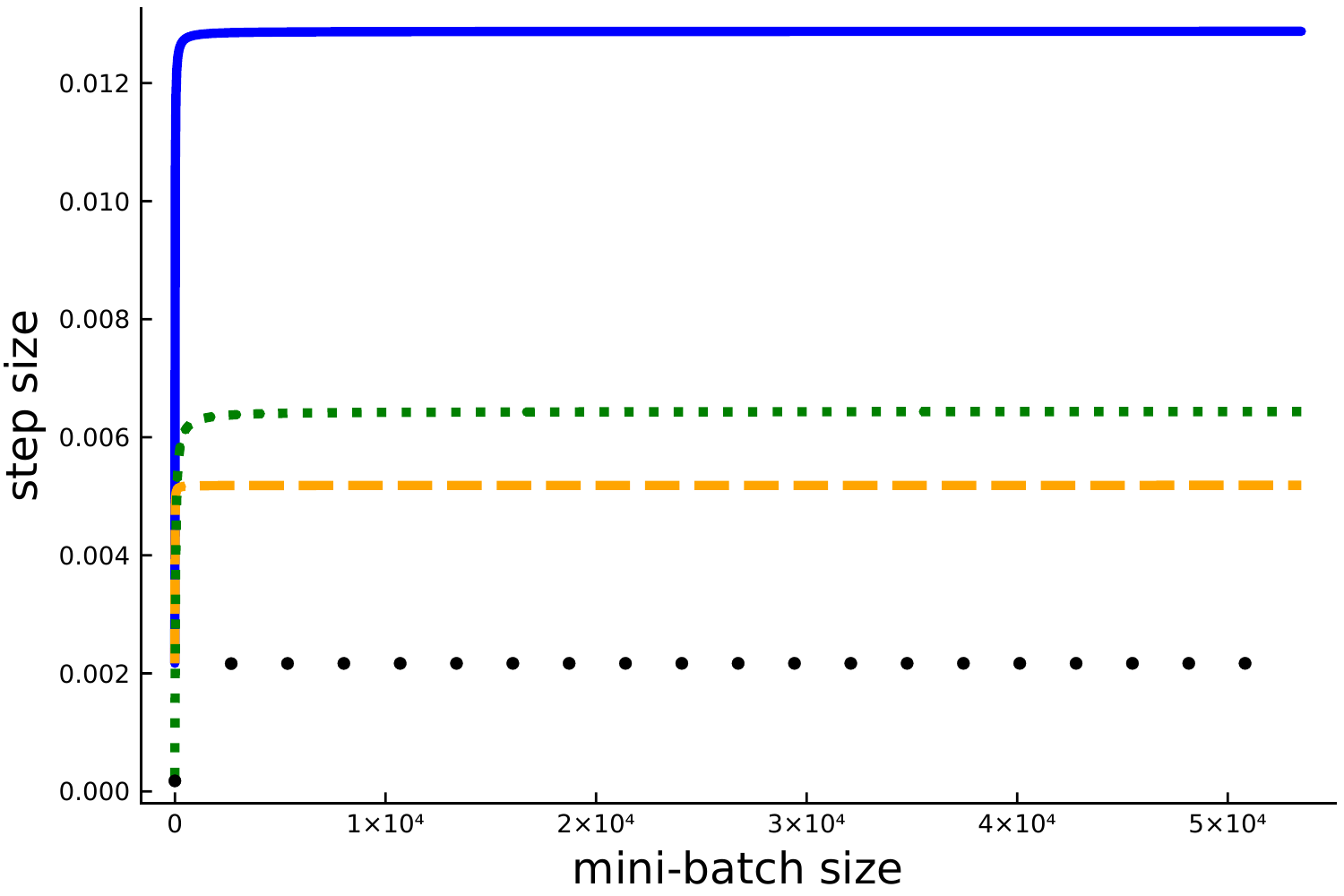}
      \caption{\textit{slice}}
    \end{subfigure}%
    \begin{subfigure}[b]{0.35\textwidth}
      \includegraphics[width=\textwidth]{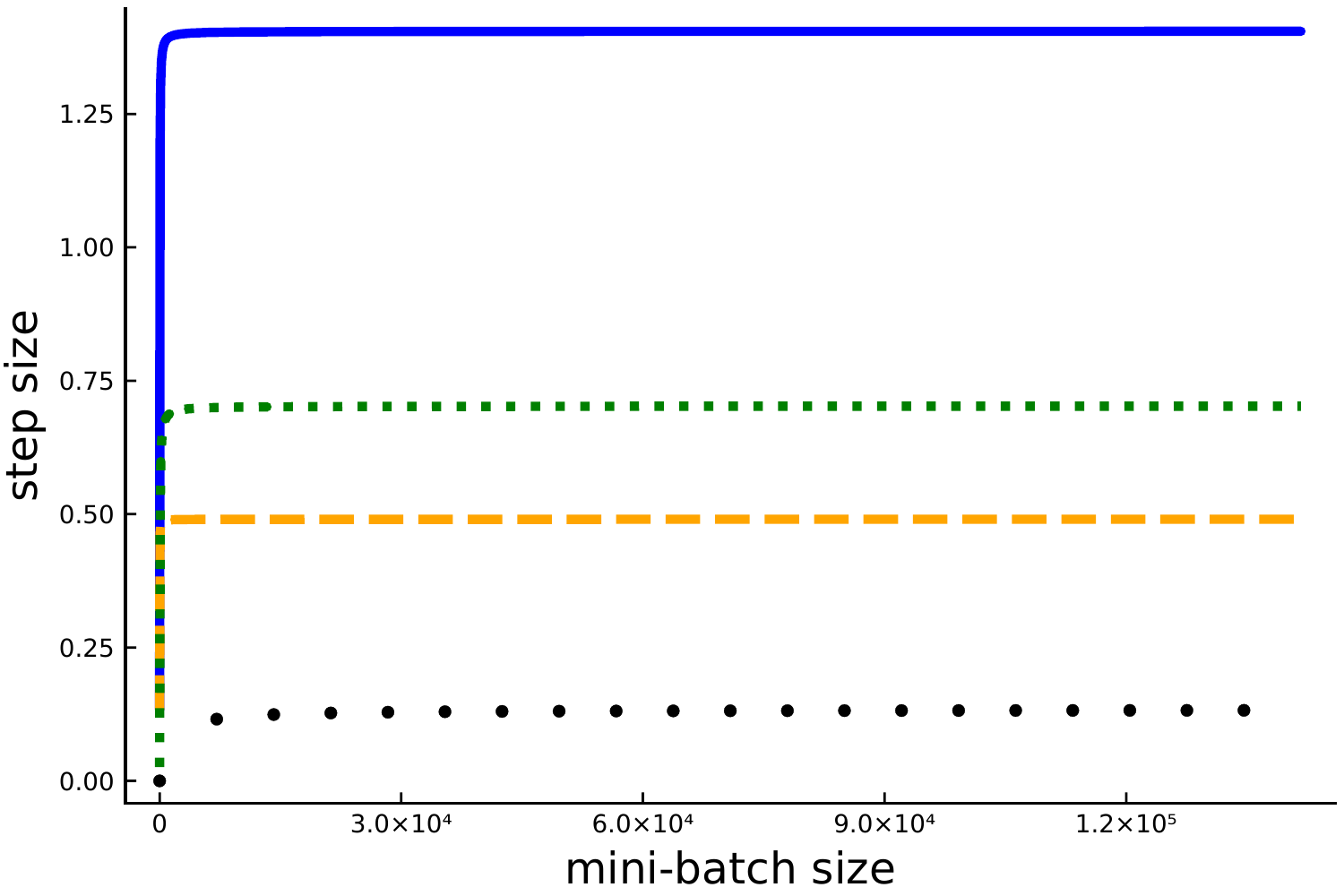}
      \caption{\textit{ijcnn1}}
    \end{subfigure}\\
    \begin{subfigure}[b]{0.35\textwidth}
      \includegraphics[width=\textwidth]{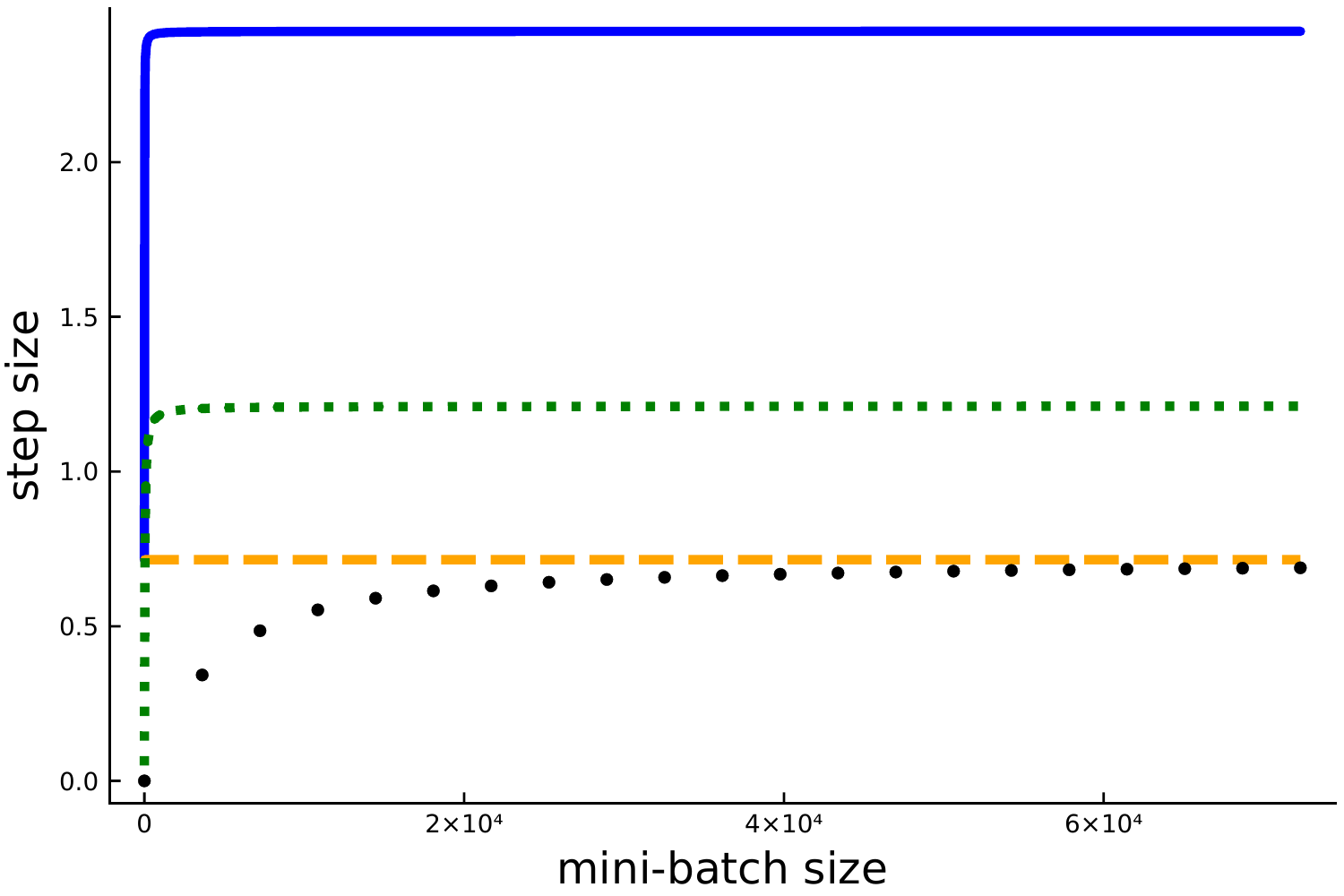}
      \caption{\textit{real-sim}}
    \end{subfigure}%
    \begin{subfigure}[b]{0.35\textwidth}
      \includegraphics[width=\textwidth]{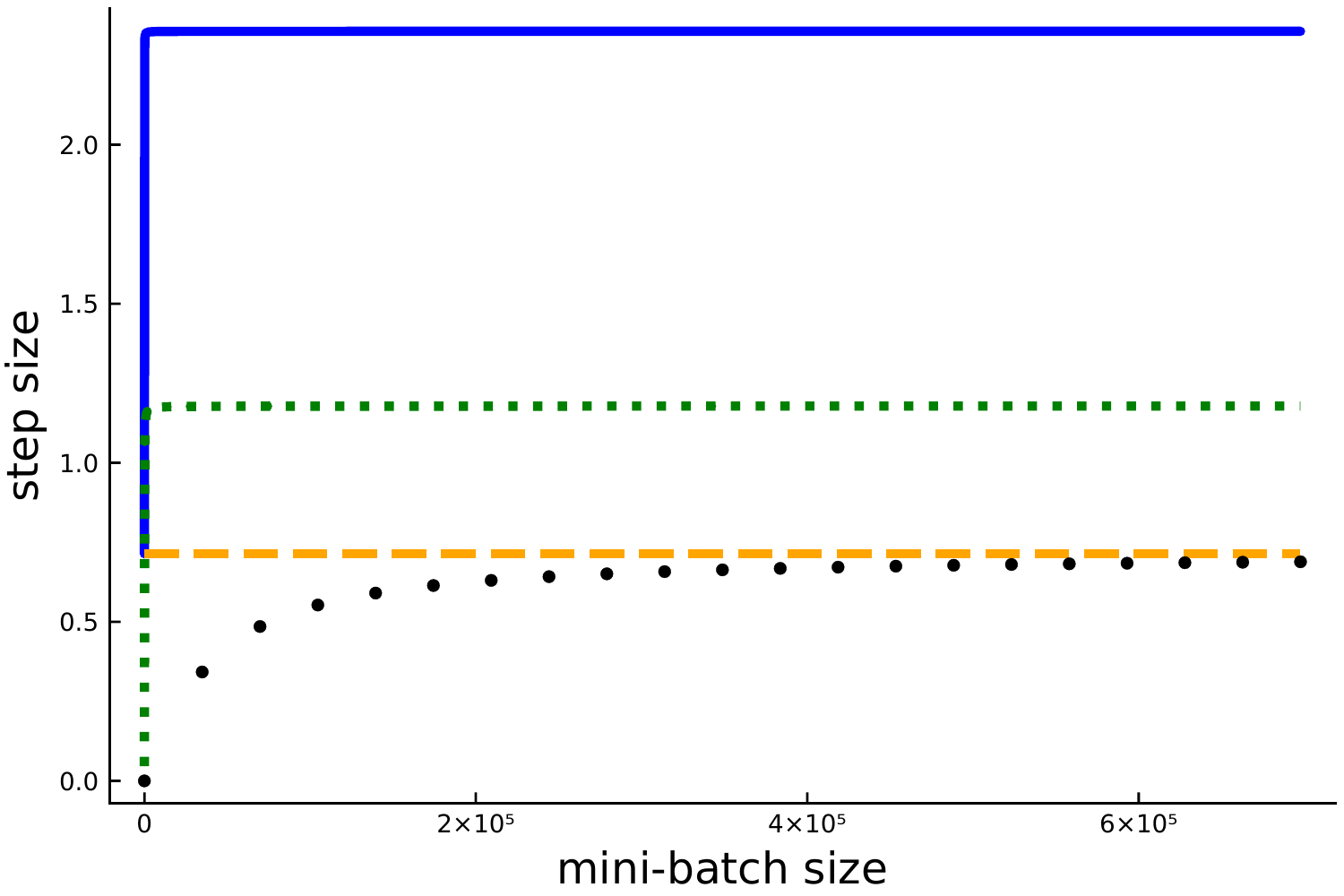}
      \caption{\textit{rcv1.binary}}
    \end{subfigure}\\
    \begin{subfigure}[b]{0.35\textwidth}
      \includegraphics[width=\textwidth]{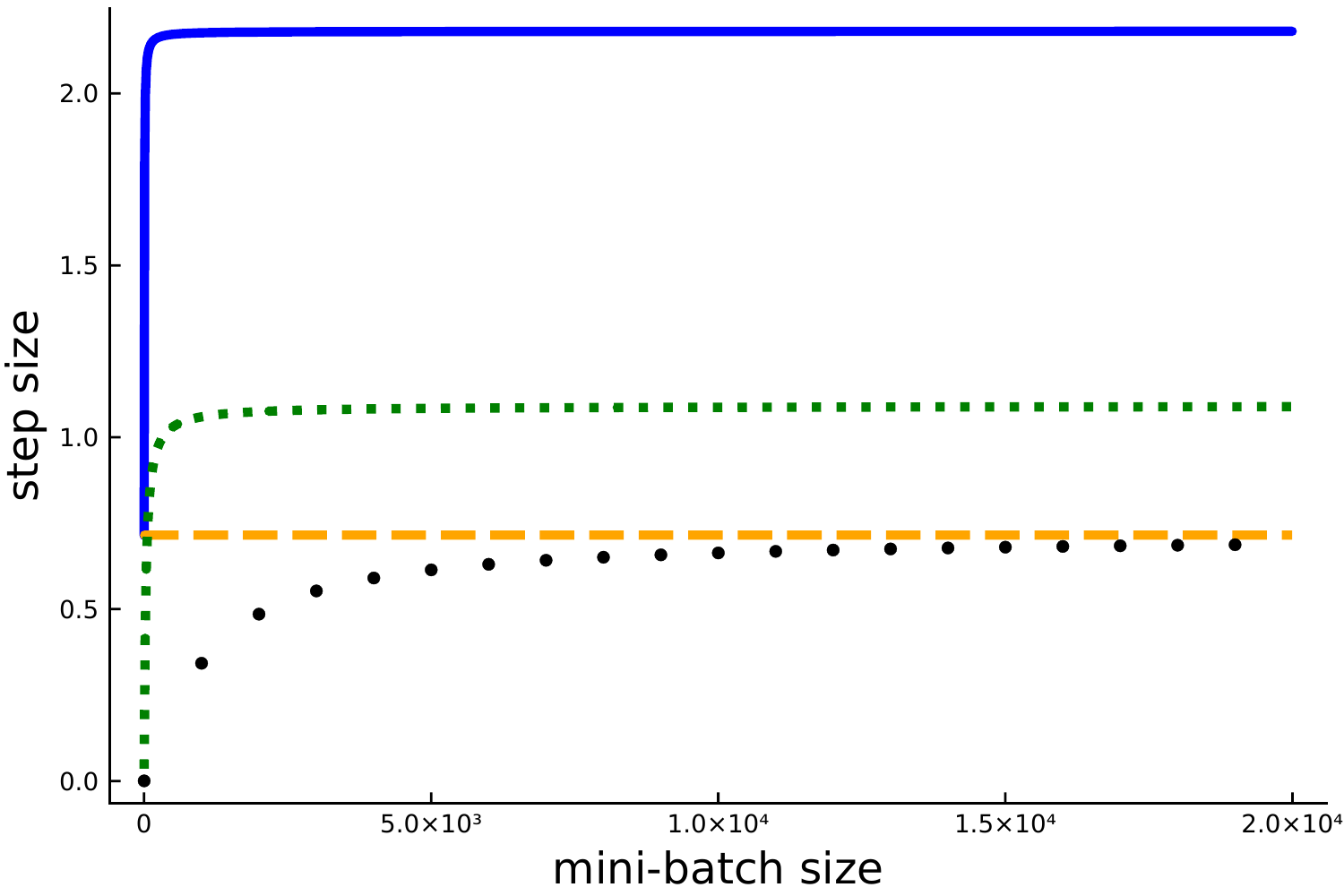}
      \caption{\textit{news20.binary}}
    \end{subfigure}\\
    \begin{subfigure}[b]{\textwidth}
      \centerline{\includegraphics[width=0.5\textwidth]{exp2/exp2-stepsizes-legend}}
    \end{subfigure}
  \caption{Step size estimates as a function the mini-batch size for real unscaled datasets ($\lambda = 10^{-1}$).}
  \label{fig:exp2_real_datasets_unscaled}
  \end{center}
  \vskip -0.2in
\end{figure}

\begin{figure}[!ht]
  \vskip 0.2in
  \begin{center}
    \begin{subfigure}[b]{0.35\textwidth}
      \includegraphics[width=\textwidth]{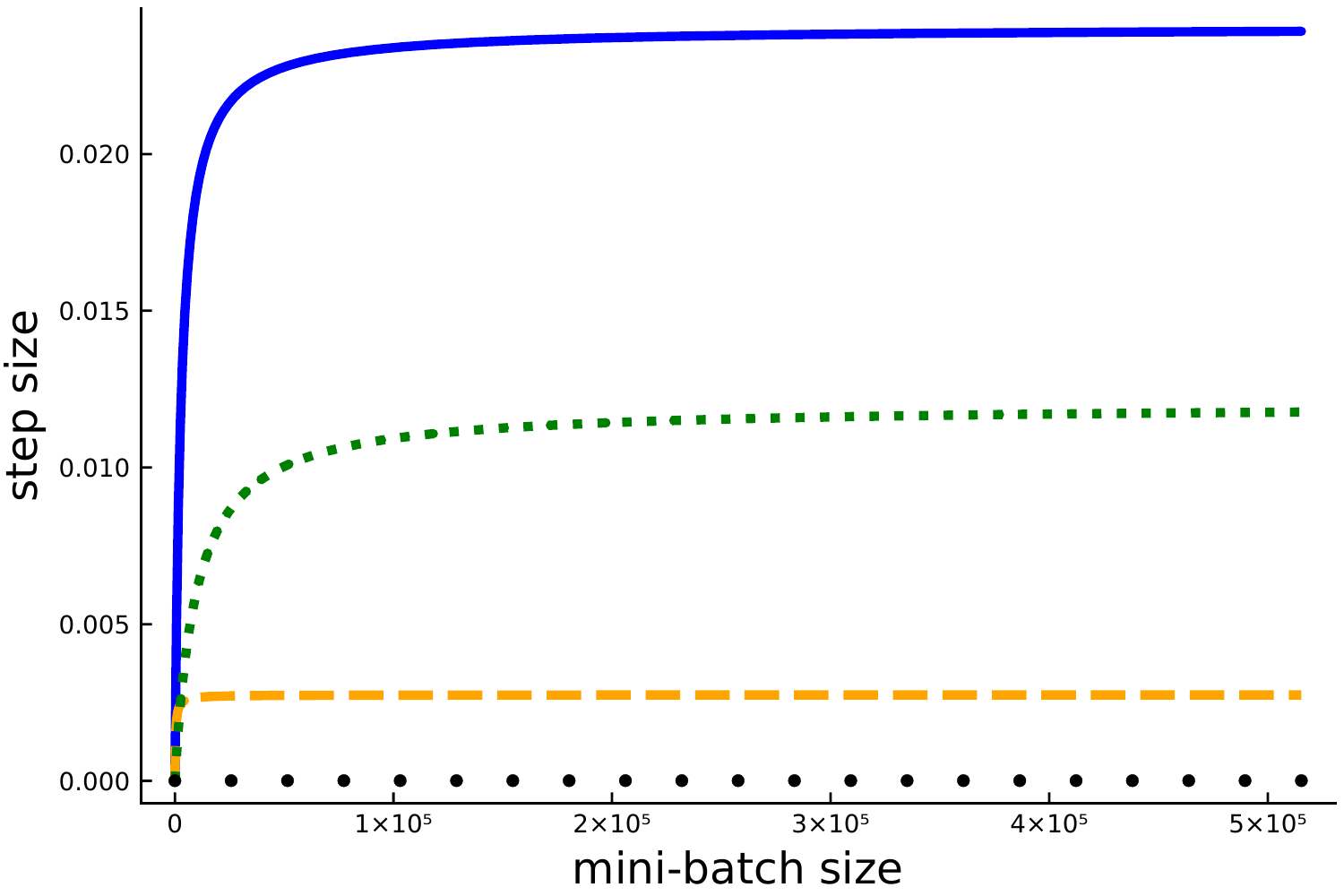}
      \caption{ \textit{YearPredictionMSD}}
    \end{subfigure}%
    \begin{subfigure}[b]{0.35\textwidth}
      \includegraphics[width=\textwidth]{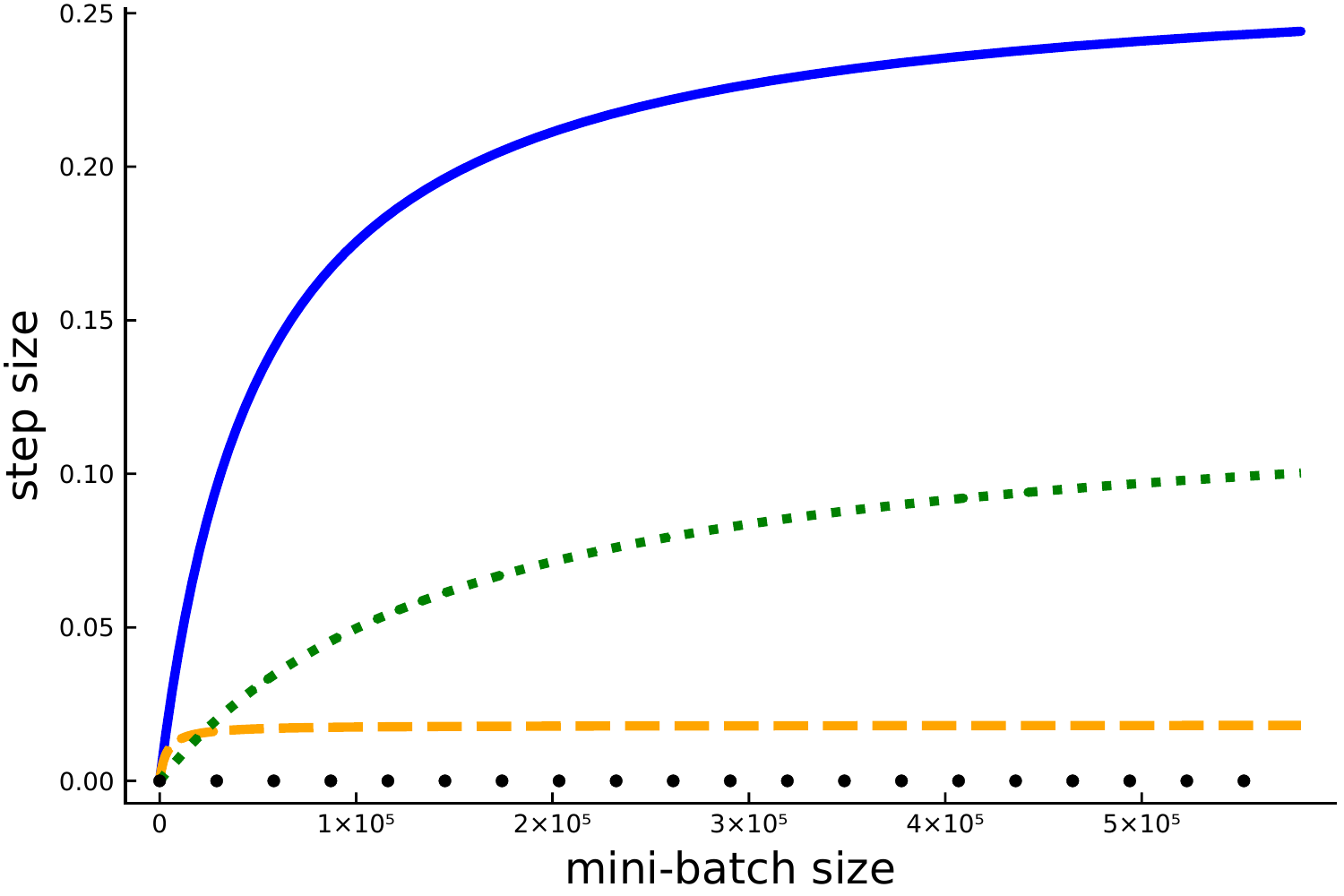}
      \caption{\textit{covtype.binary}}
    \end{subfigure}\\
    \begin{subfigure}[b]{0.35\textwidth}
      \includegraphics[width=\textwidth]{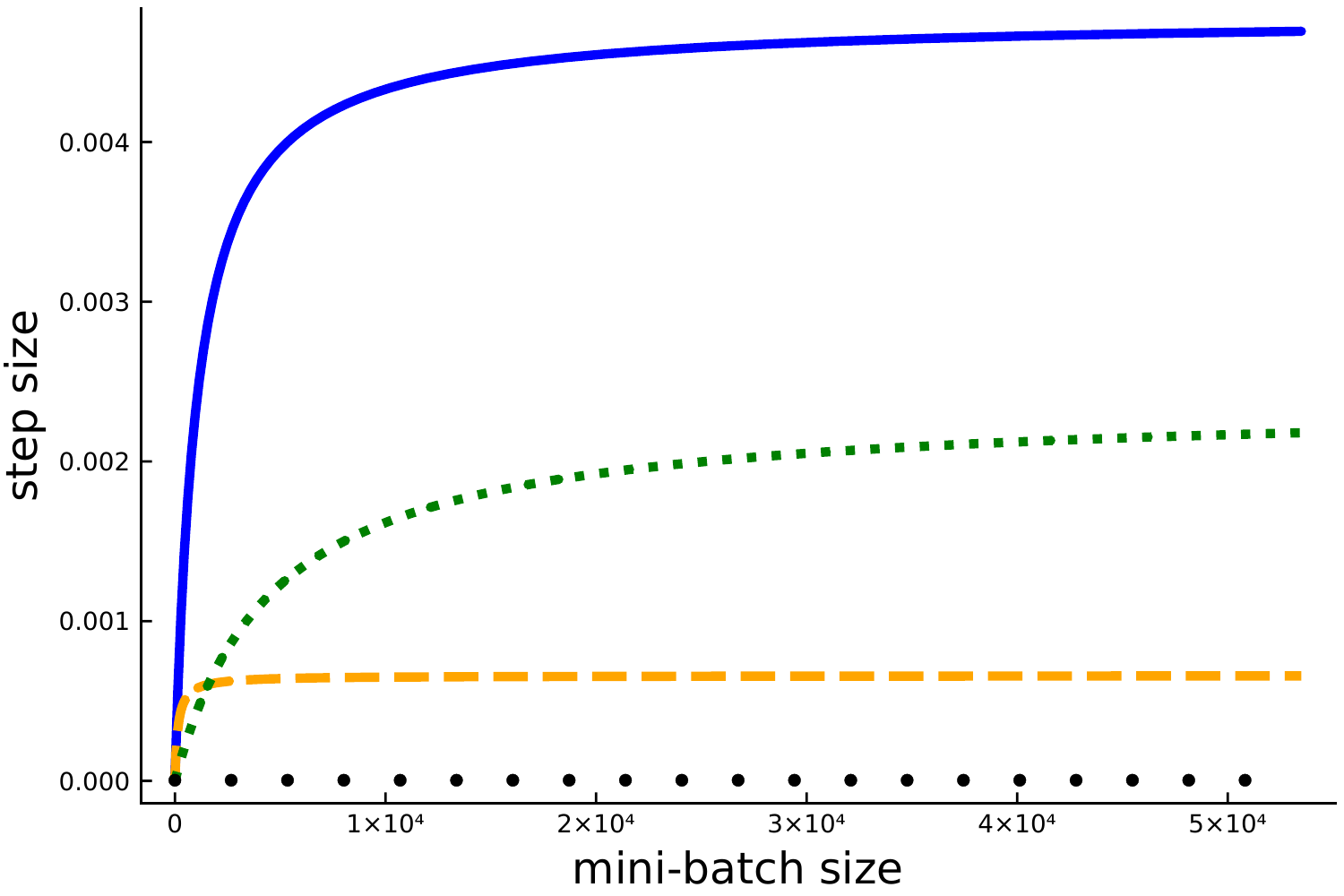}
      \caption{\textit{slice}}
    \end{subfigure}%
    \begin{subfigure}[b]{0.35\textwidth}
      \includegraphics[width=\textwidth]{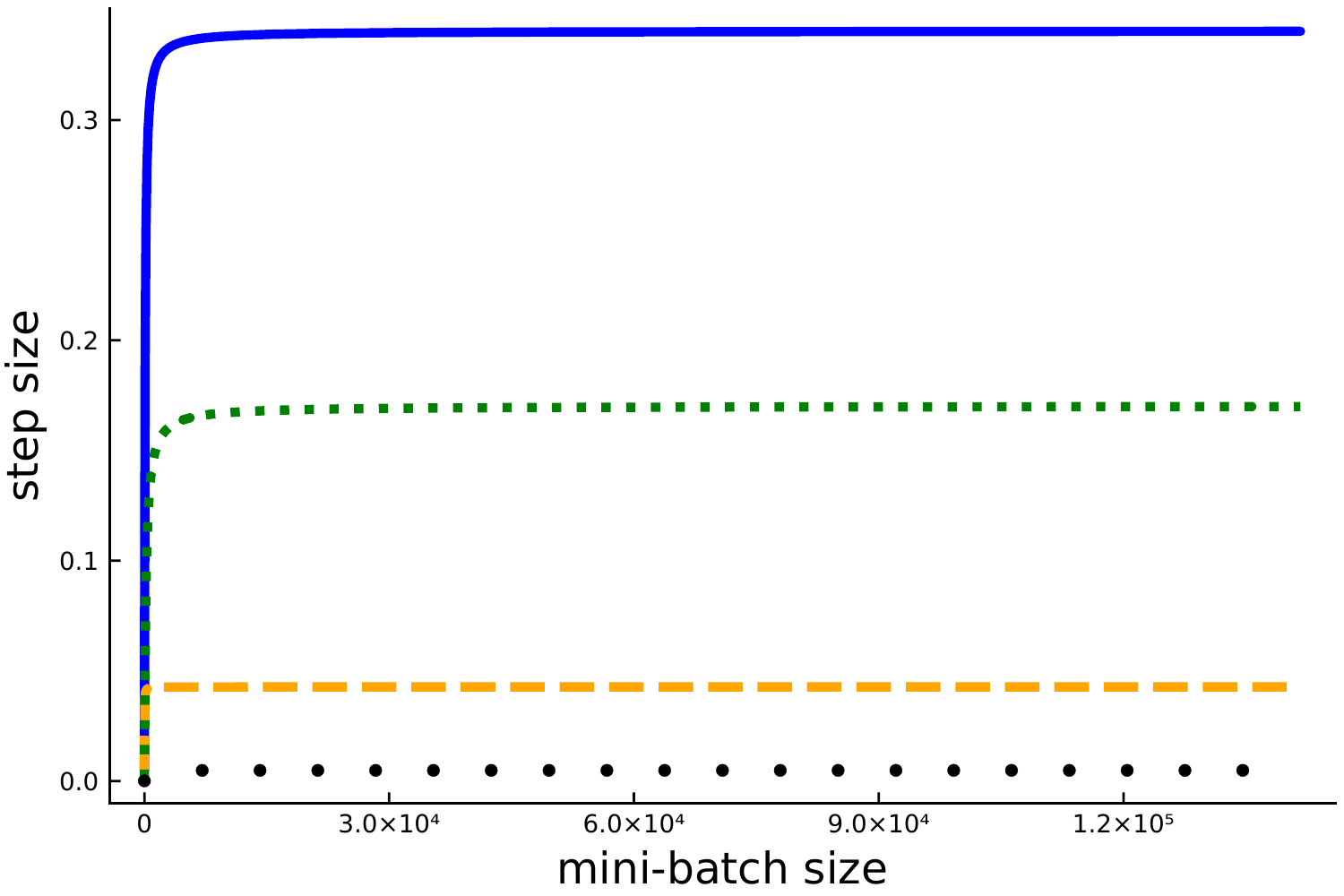}
      \caption{\textit{ijcnn1}}
    \end{subfigure}\\
    \begin{subfigure}[b]{\textwidth}
      \centerline{\includegraphics[width=0.5\textwidth]{exp2/exp2-stepsizes-legend}}
    \end{subfigure}
  \caption{Step size estimates as a function the mini-batch size for real feature-scaled datasets ($\lambda = 10^{-1}$).}
  \label{fig:exp2_real_datasets_scaled}
  \end{center}
  \vskip -0.2in
\end{figure}

\subsection{Experiment 3: comparison with previous SAGA settings}
\label{appendix:exp_3}

In this section we provide more example of the performance of our practical settings compared to previously known SAGA settings. In \Cref{fig:exp3_scaled_ijcnn1,fig:exp3_scaled_covtype,fig:exp3_scaled_YearPredicitonMSD,fig:exp3_scaled_slice,fig:exp3_unscaled_slice,fig:exp3_unscaled_realsim} we run our experiments on real datasets introduced in detail in \Cref{appendix:exp1}. SAGA implementations are run until the suboptimality reaches a relative error $(f(w^k)-f(w^*))/(f(w^0)-f(w^*))$ of $10^{-4}$, except in some cases where the Hofmann's runs exceeded our maximal number of epochs like in \Cref{fig:exp3_scaled_covtype}. In \Cref{fig:exp3_scaled_slice}, the curves corresponding to Hofmann's settings are not displayed because they achieve a total complexity which is too large. \Cref{fig:exp3_hofmann_fail} shows an example of such a configuration.

These experiments show that our settings ($b_{\text{practical}}, \gamma_{\text{practical}}$) most of the time outperforms whether the classical ($b=1, \gamma_{\text{Defazio}}$) or the ($b=20, \gamma_{\text{Hofmann}}$) settings both in terms of epochs and running time.

\begin{figure}[ht]
  \vskip 0.2in
  \begin{center}
  \includegraphics[width=0.5\textwidth]{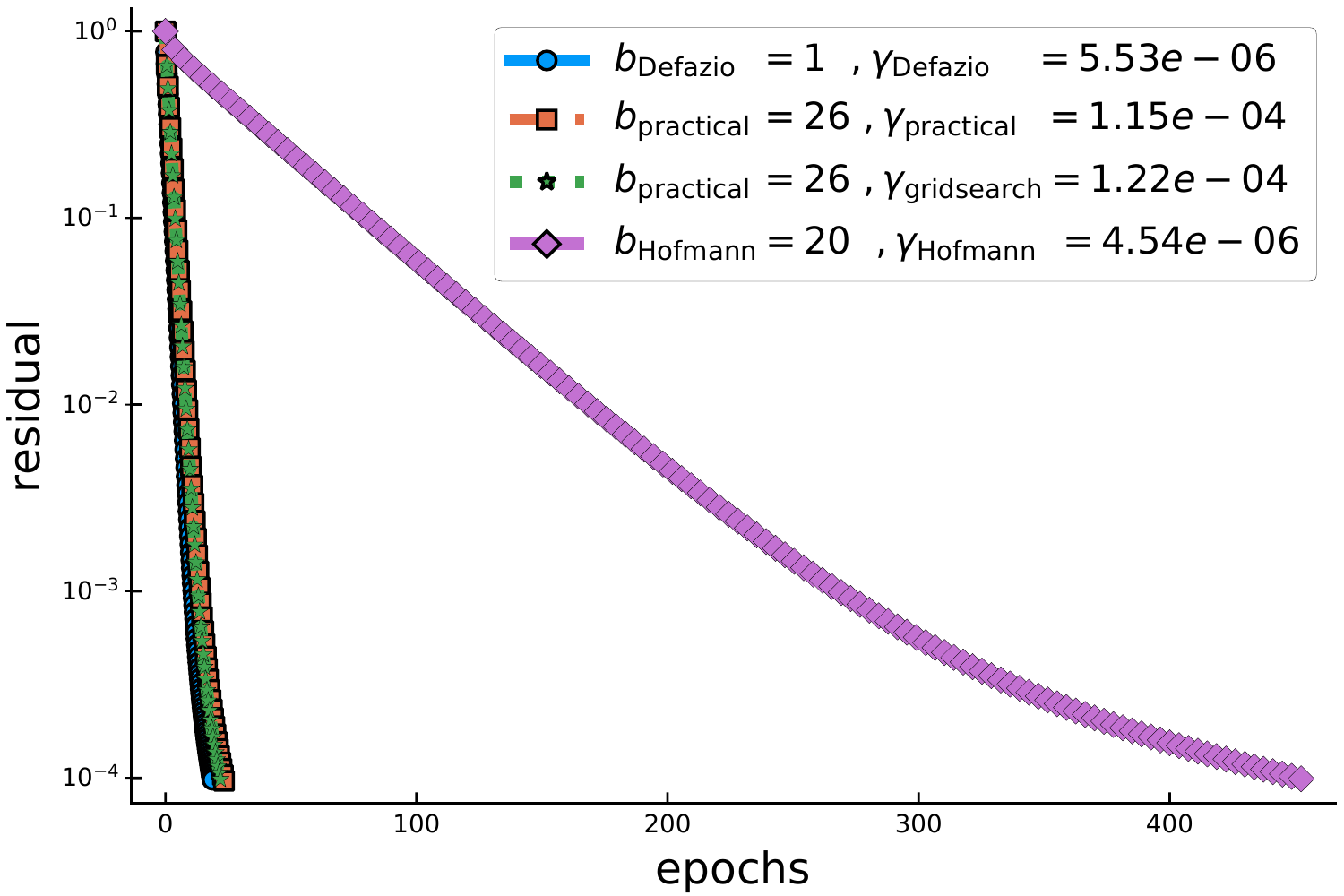}
  \caption{Poor performance of Hofmann's settings for the feature-scaled dataset \textit{slice} ($\lambda = 10^{-1}$).}
  \label{fig:exp3_hofmann_fail}
  \end{center}
  \vskip -0.2in
\end{figure}

\begin{figure}[!ht]
  \vskip 0.2in
  \begin{center}
    \begin{subfigure}[b]{\textwidth}
      \centering
      \includegraphics[width=0.35\textwidth]{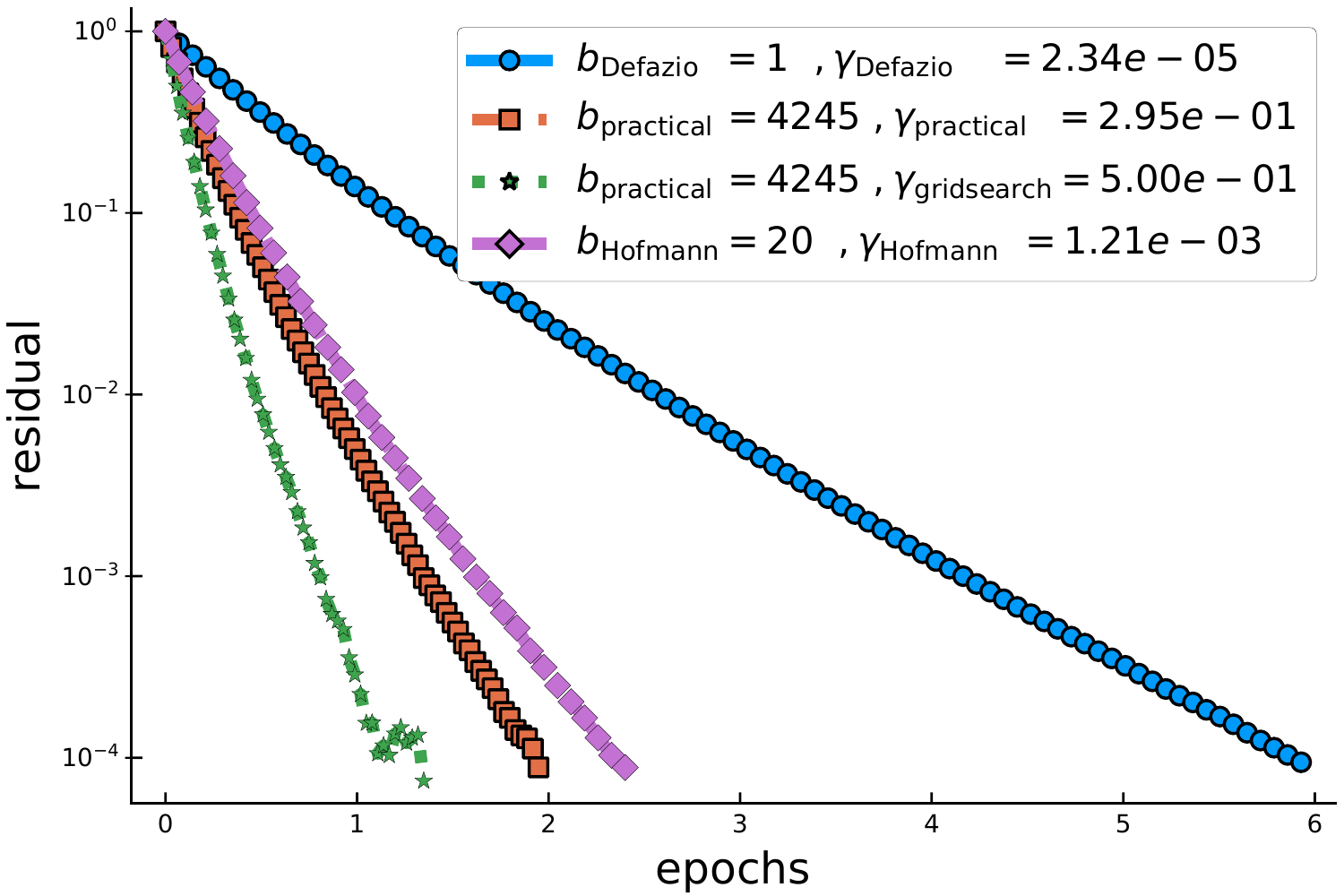}%
      \includegraphics[width=0.35\textwidth]{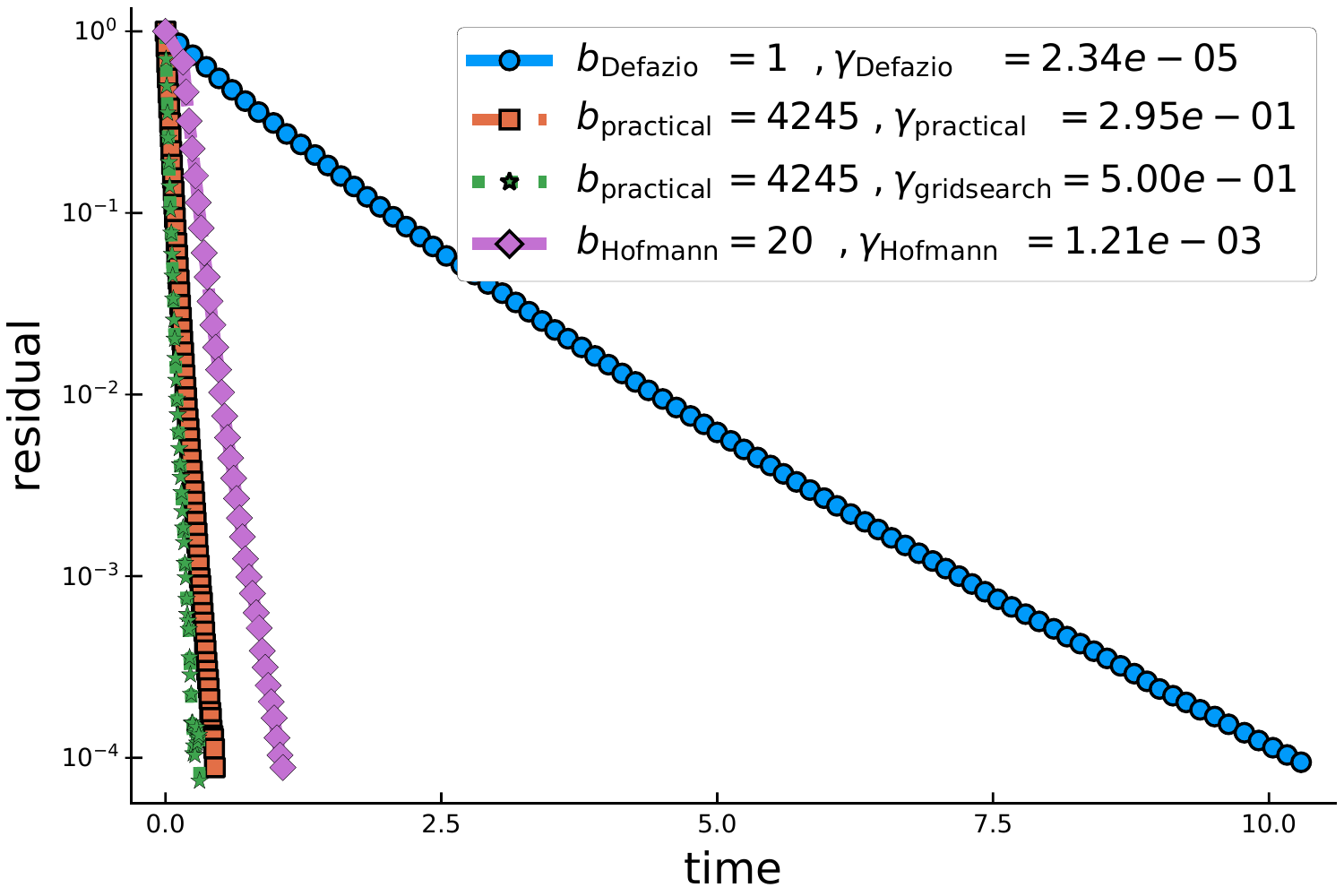}
      \caption{$\lambda = 10^{-1}$}
    \end{subfigure}\\
    \begin{subfigure}[b]{\textwidth}
      \centering
      \includegraphics[width=0.35\textwidth]{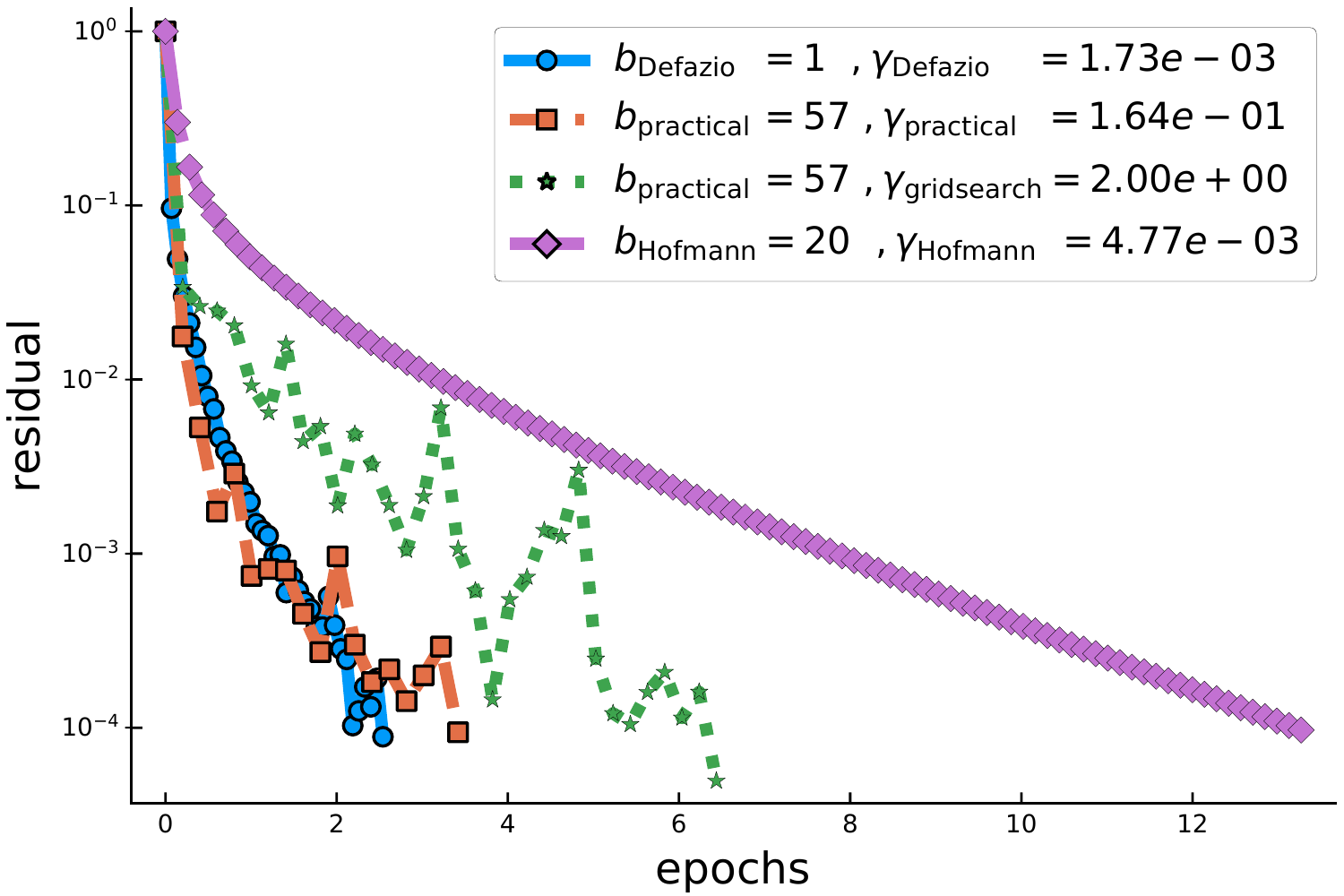}%
      \includegraphics[width=0.35\textwidth]{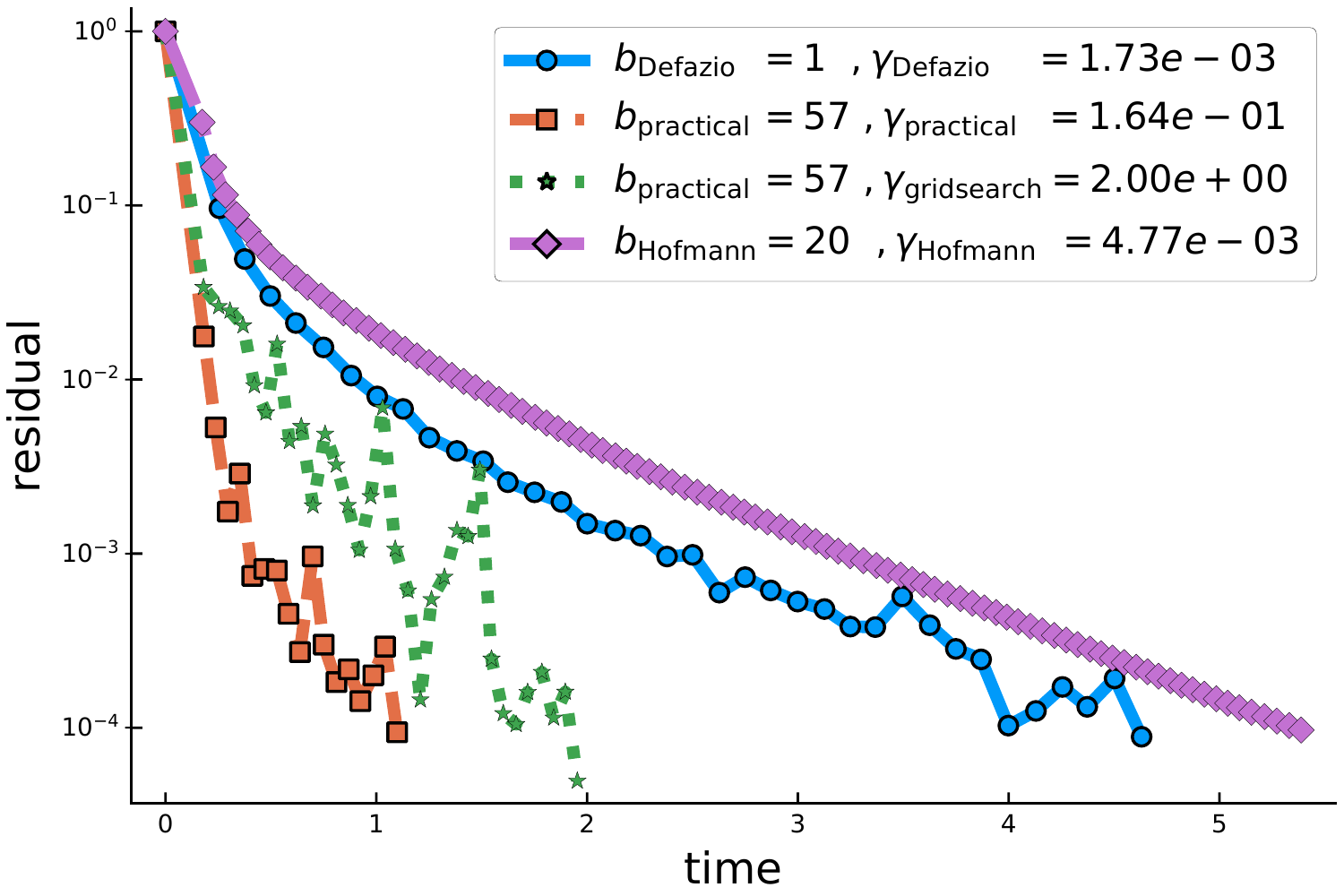}
      \caption{$\lambda = 10^{-3}$}
    \end{subfigure}
  \caption{Performance of SAGA implementations for the feature-scaled dataset \textit{ijcnn1}.}
  \label{fig:exp3_scaled_ijcnn1}
  \end{center}
  \vskip -0.2in
\end{figure}

\begin{figure}[!ht]
  \vskip 0.2in
  \begin{center}
    \begin{subfigure}[b]{\textwidth}
      \centering
      \includegraphics[width=0.35\textwidth]{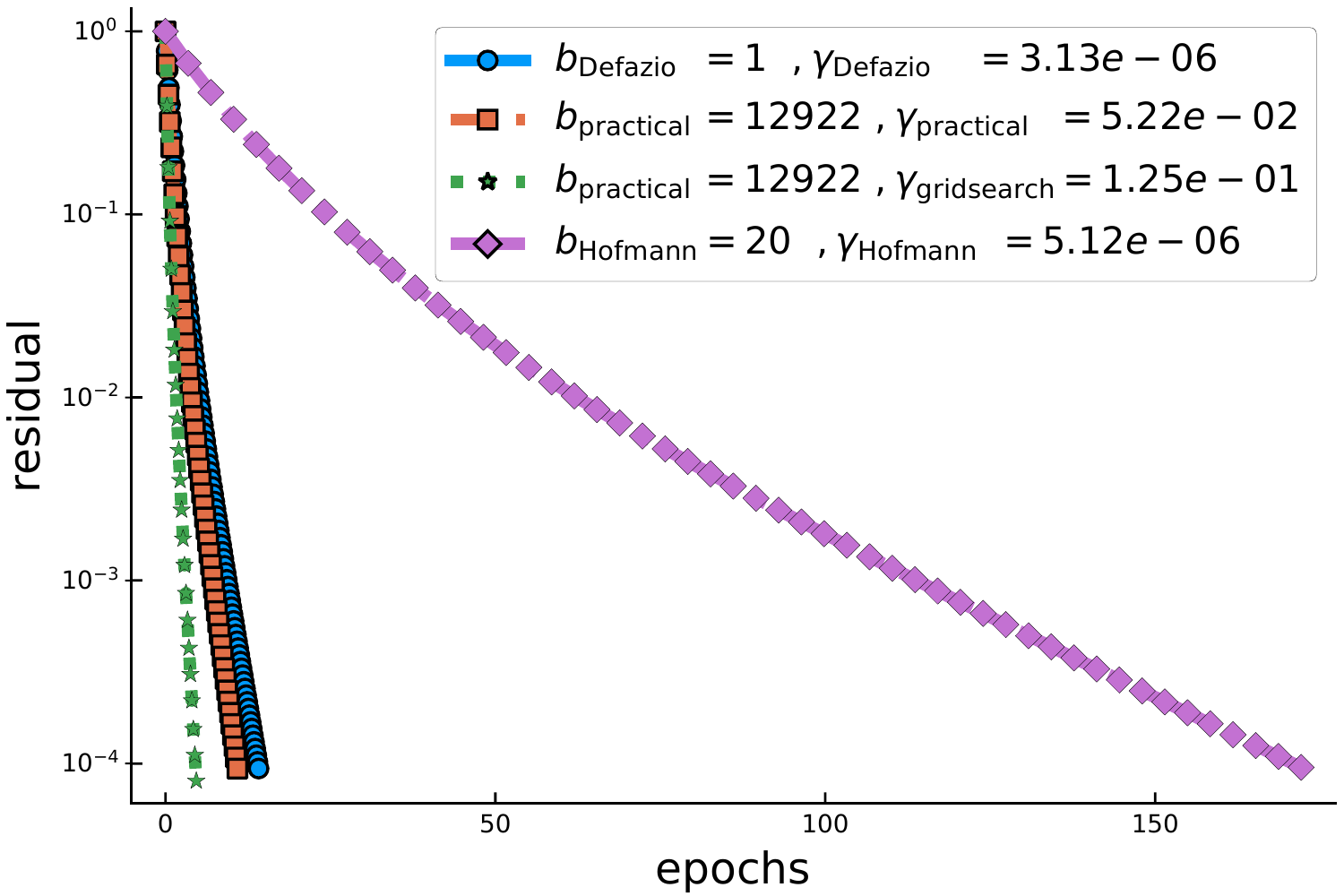}%
      \includegraphics[width=0.35\textwidth]{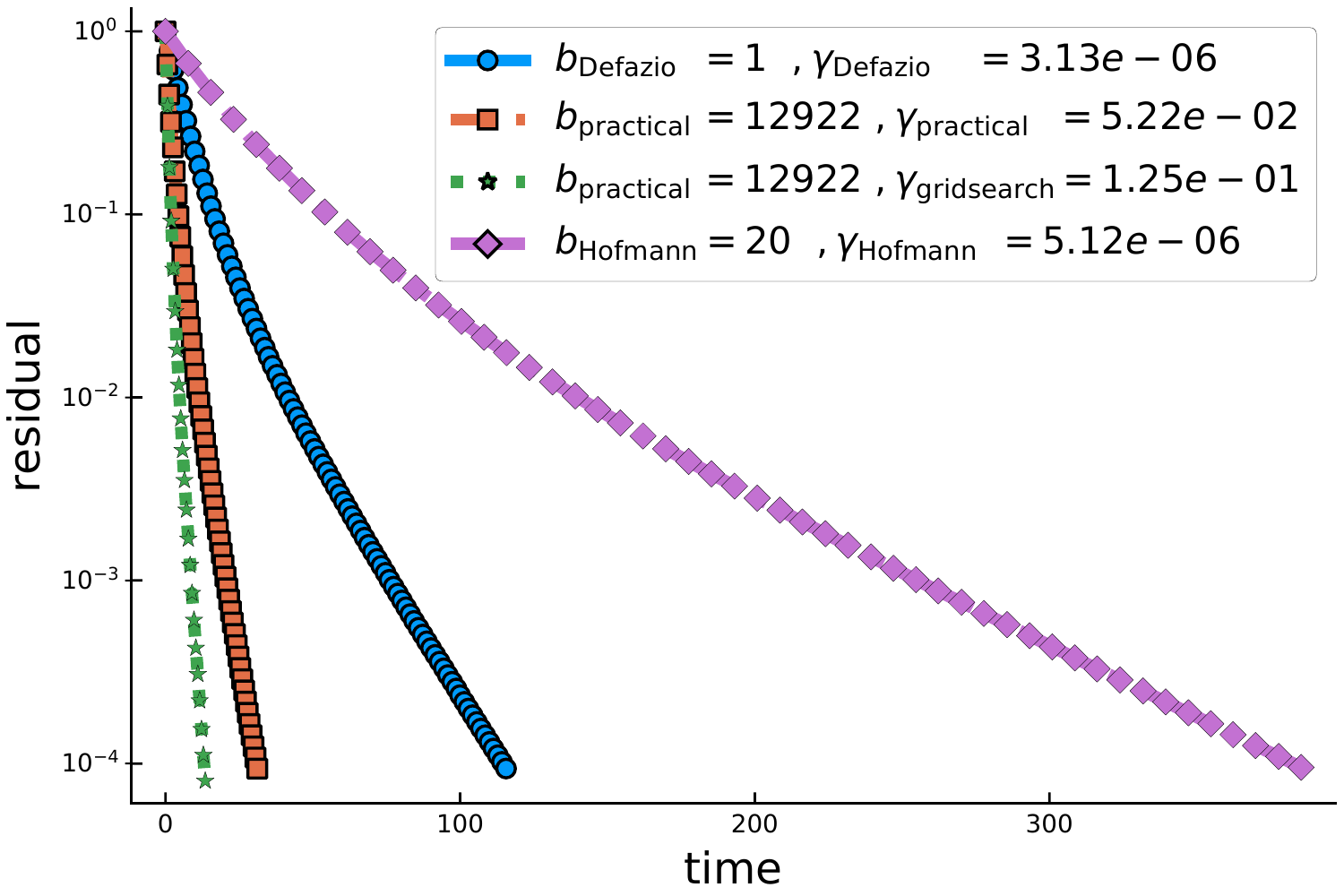}
      \caption{$\lambda = 10^{-1}$}
    \end{subfigure}\\
    \begin{subfigure}[b]{\textwidth}
      \centering
      \includegraphics[width=0.35\textwidth]{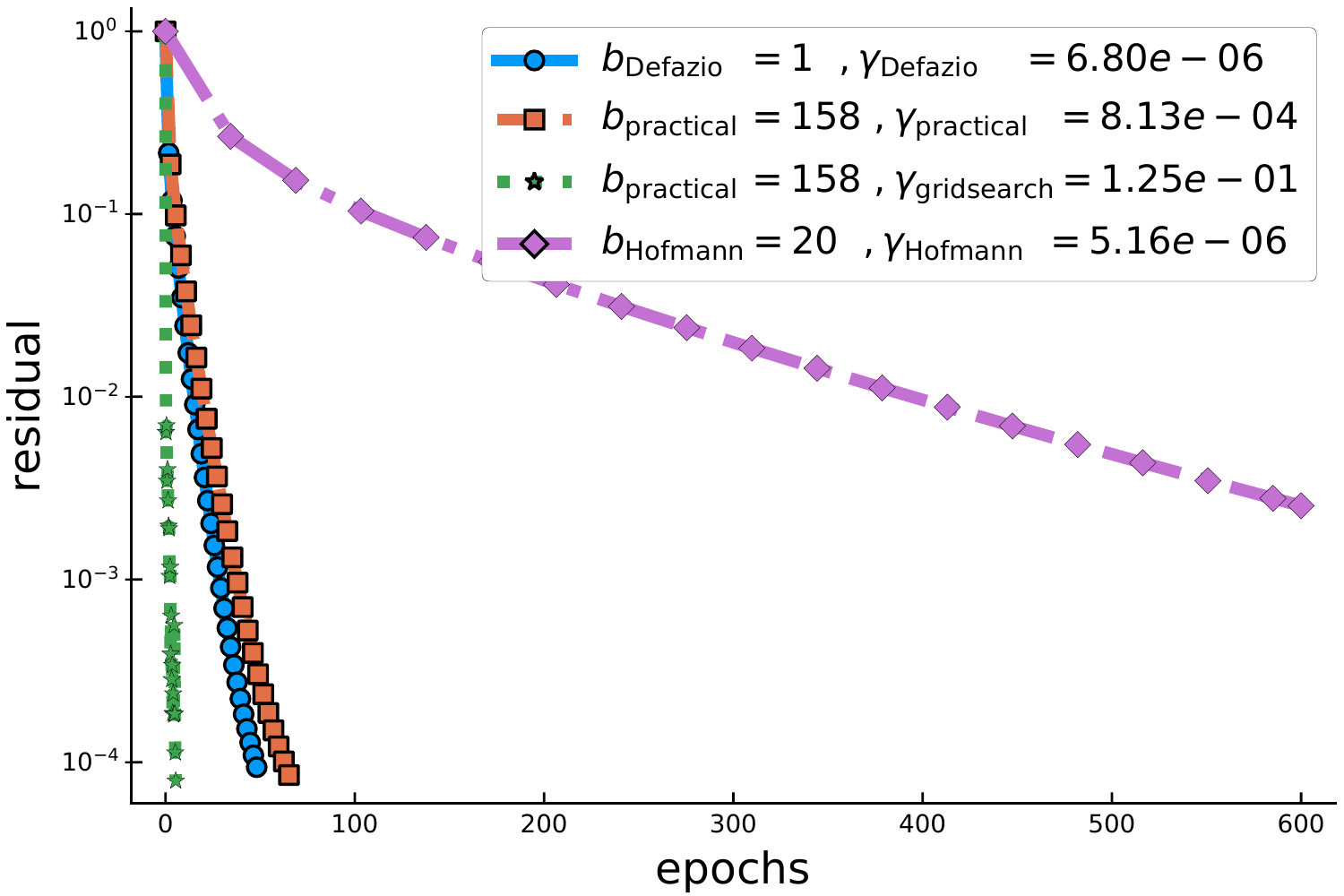}%
      \includegraphics[width=0.35\textwidth]{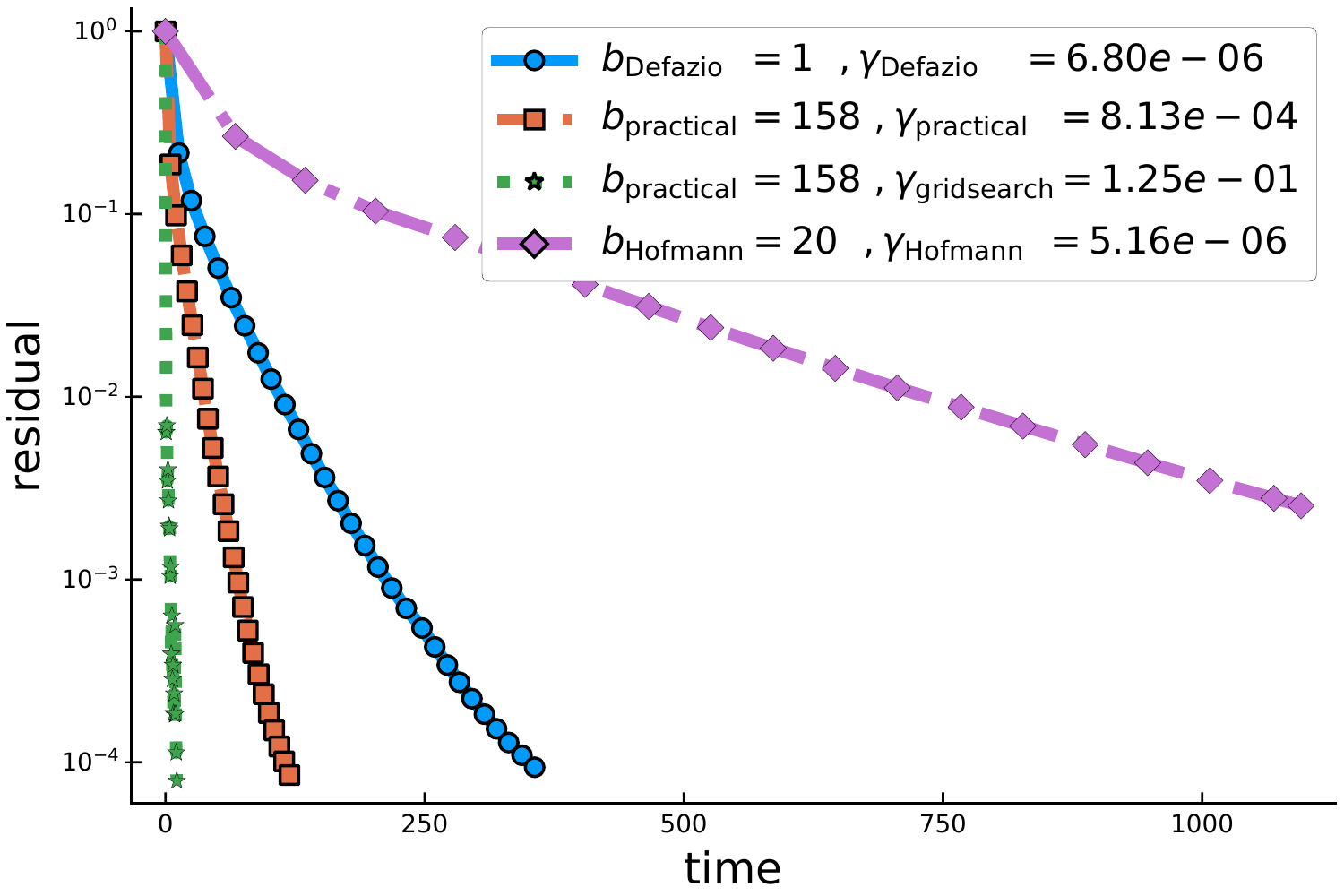}
      \caption{$\lambda = 10^{-3}$}
    \end{subfigure}
  \caption{Performance of SAGA implementations for the feature-scaled dataset \textit{covtype.binary}.}
  \label{fig:exp3_scaled_covtype}
  \end{center}
  \vskip -0.2in
\end{figure}

\begin{figure}[!ht]
  \vskip 0.2in
  \begin{center}
    \begin{subfigure}[b]{\textwidth}
      \centering
      \includegraphics[width=0.35\textwidth]{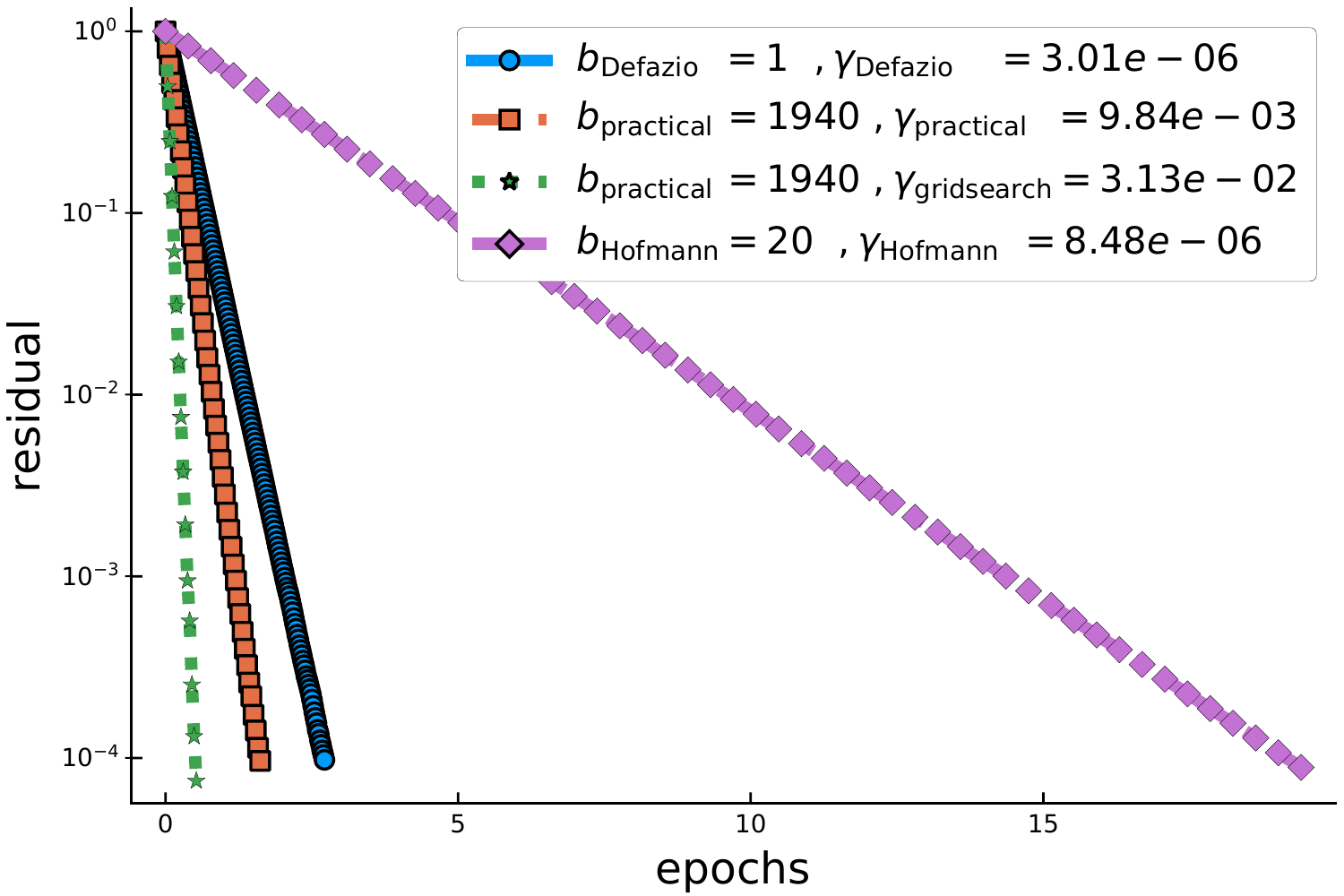}%
      \includegraphics[width=0.35\textwidth]{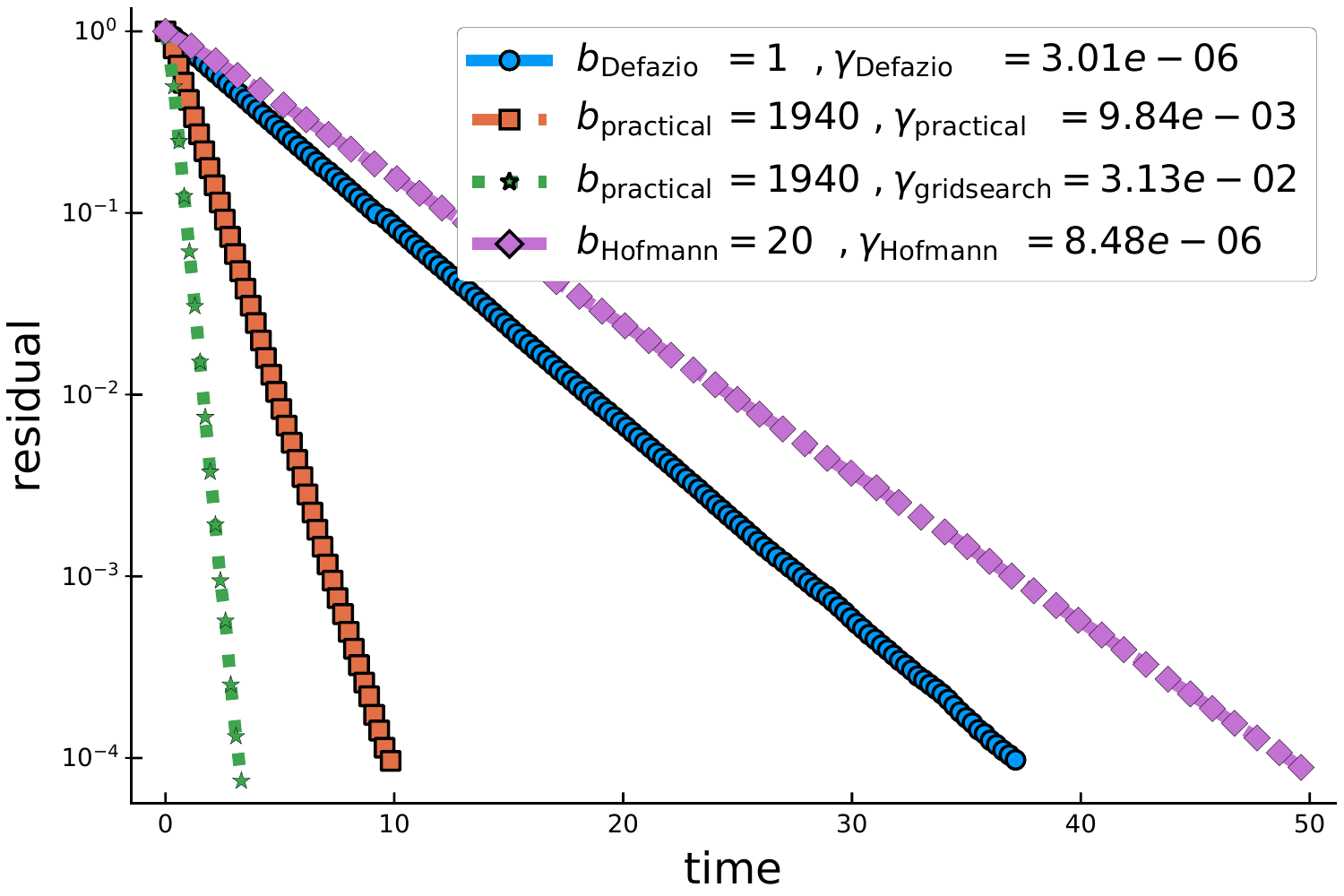}
      \caption{$\lambda = 10^{-1}$}
    \end{subfigure}\\
    \begin{subfigure}[b]{\textwidth}
      \centering
      \includegraphics[width=0.35\textwidth]{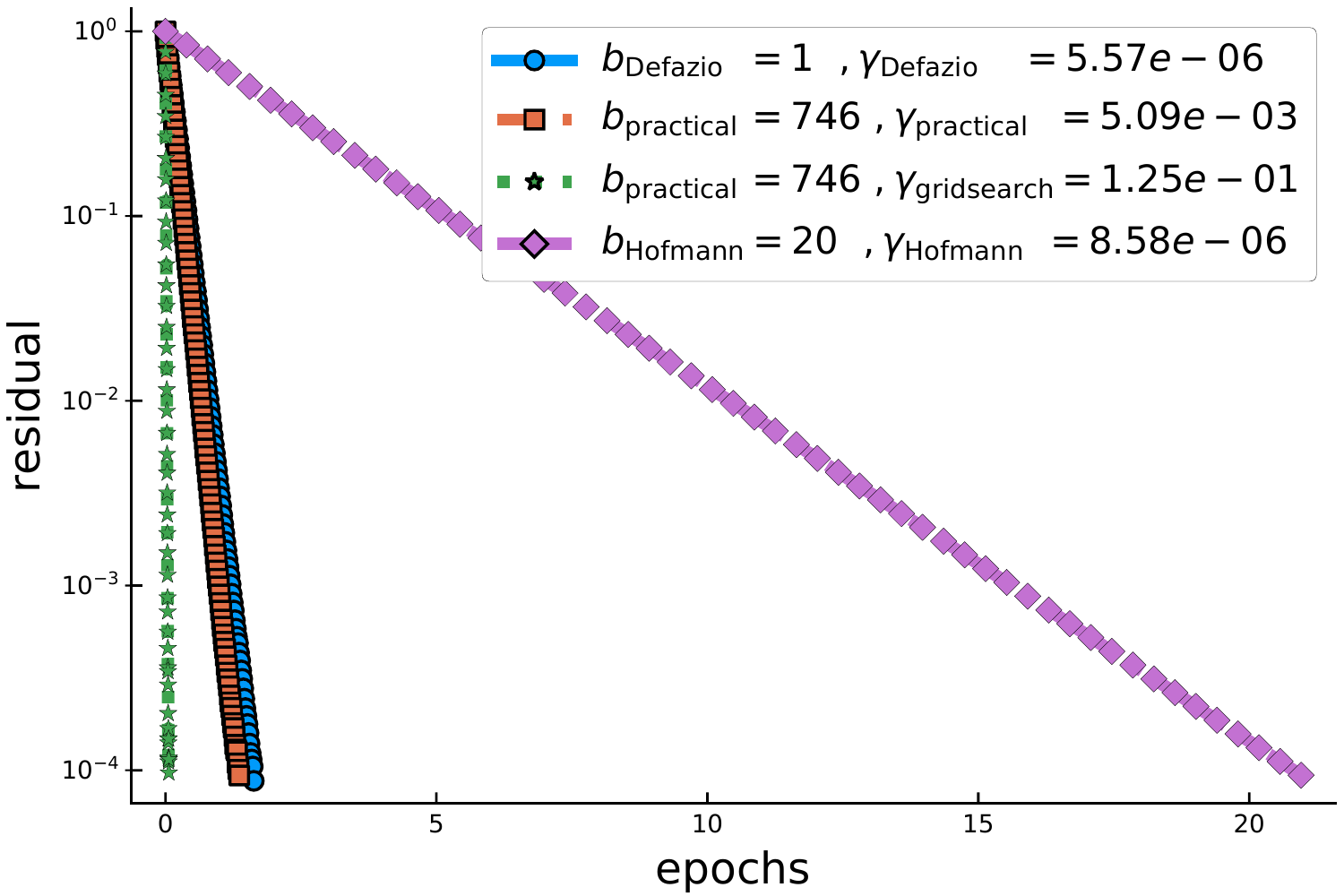}%
      \includegraphics[width=0.35\textwidth]{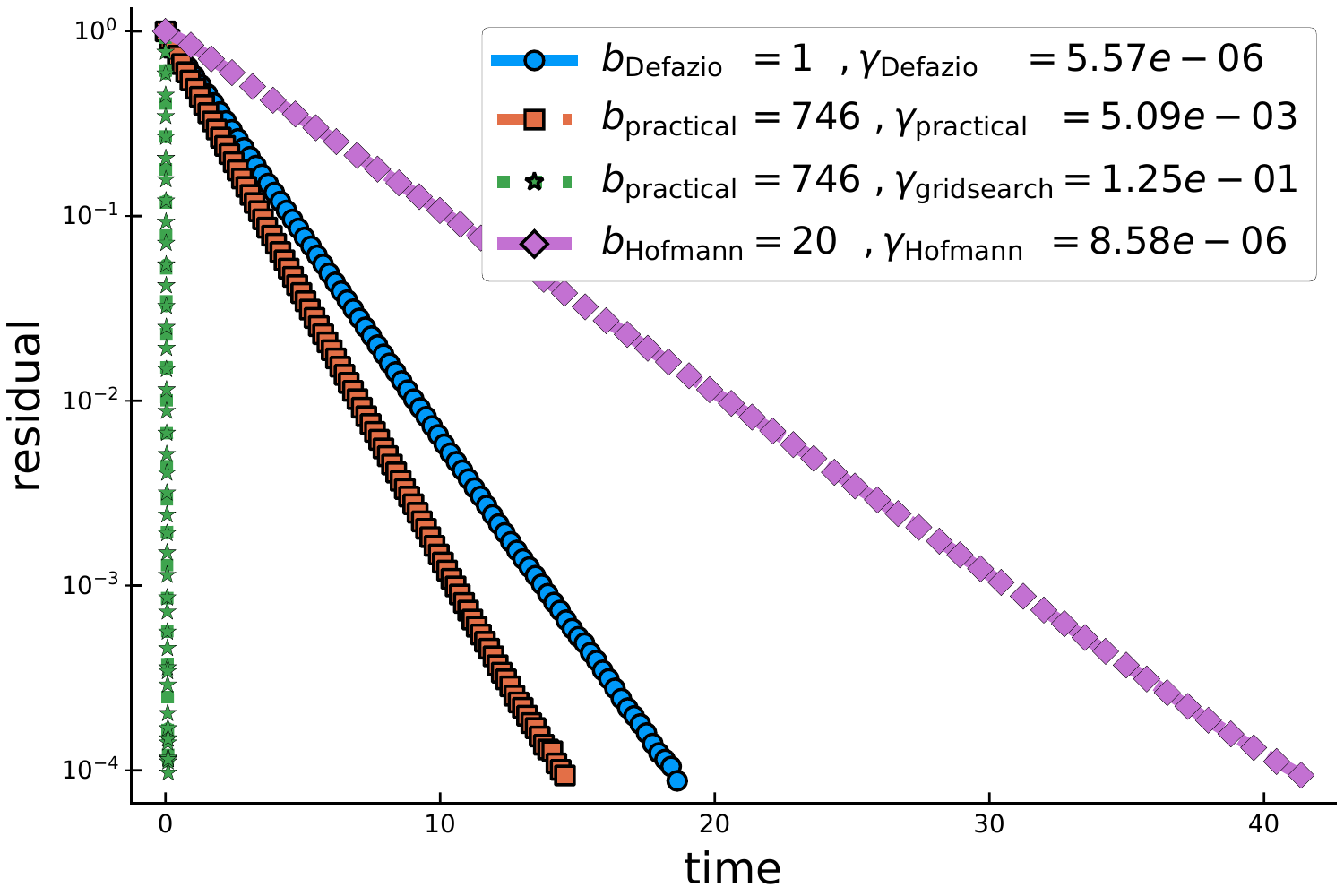}
      \caption{$\lambda = 10^{-3}$}
    \end{subfigure}
  \caption{Performance of SAGA implementations for the feature-scaled dataset \textit{YearPredictionMSD}.}
  \label{fig:exp3_scaled_YearPredicitonMSD}
  \end{center}
  \vskip -0.2in
\end{figure}

\begin{figure}[!ht]
  \vskip 0.2in
  \begin{center}
    \begin{subfigure}[b]{\textwidth}
      \centering
      \includegraphics[width=0.35\textwidth]{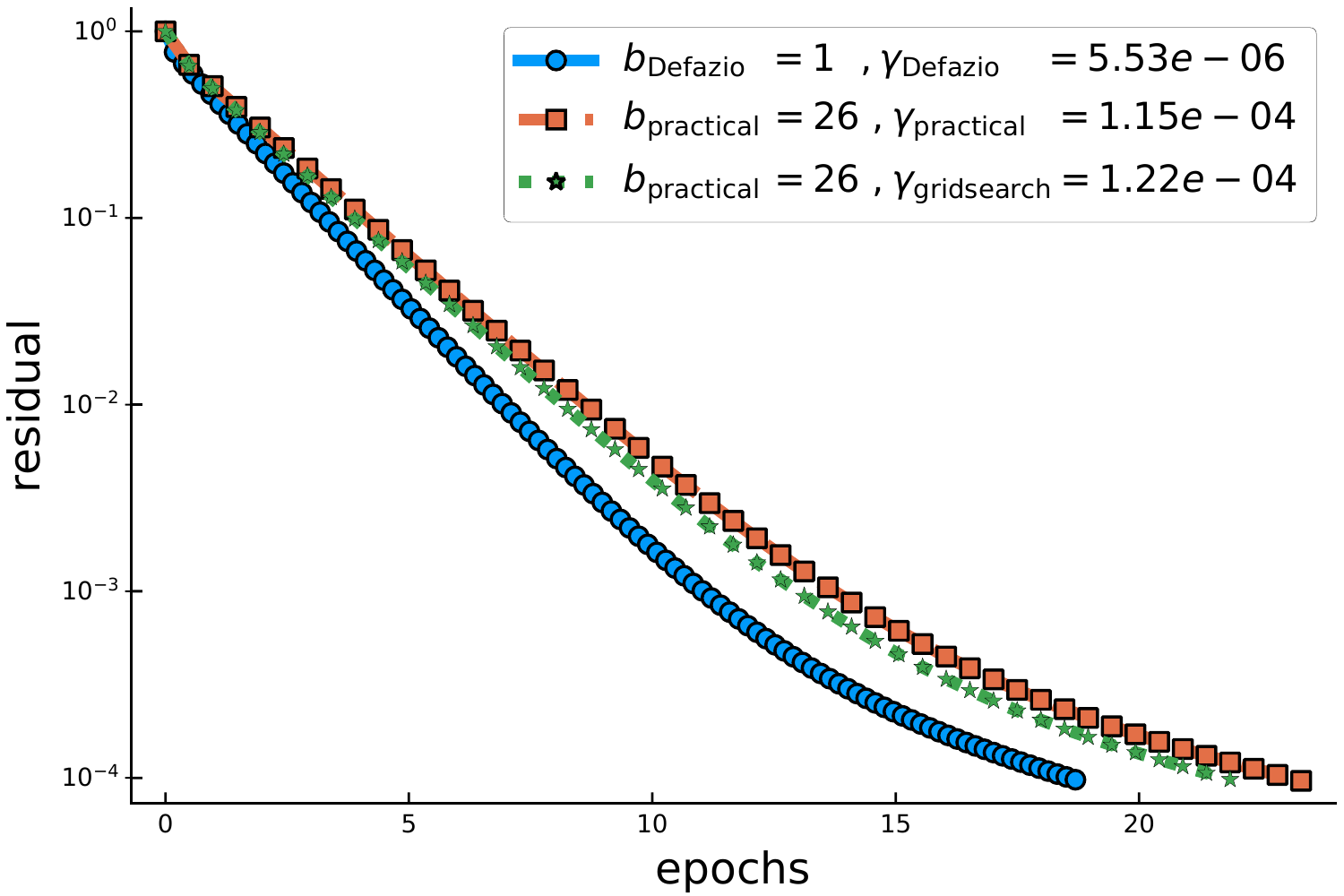}%
      \includegraphics[width=0.35\textwidth]{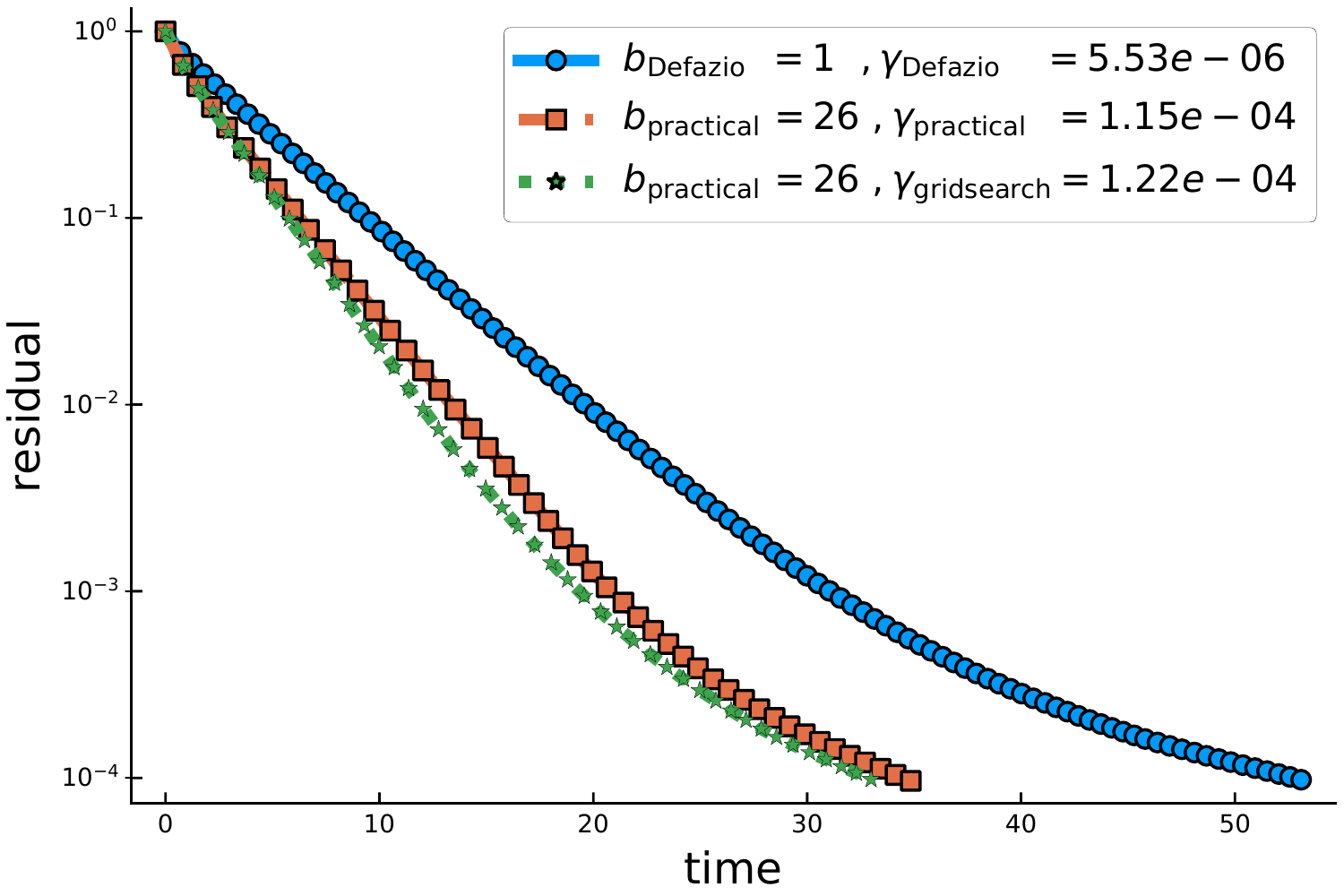}
      \caption{$\lambda = 10^{-1}$}
    \end{subfigure}\\
    \begin{subfigure}[b]{\textwidth}
      \centering
      \includegraphics[width=0.35\textwidth]{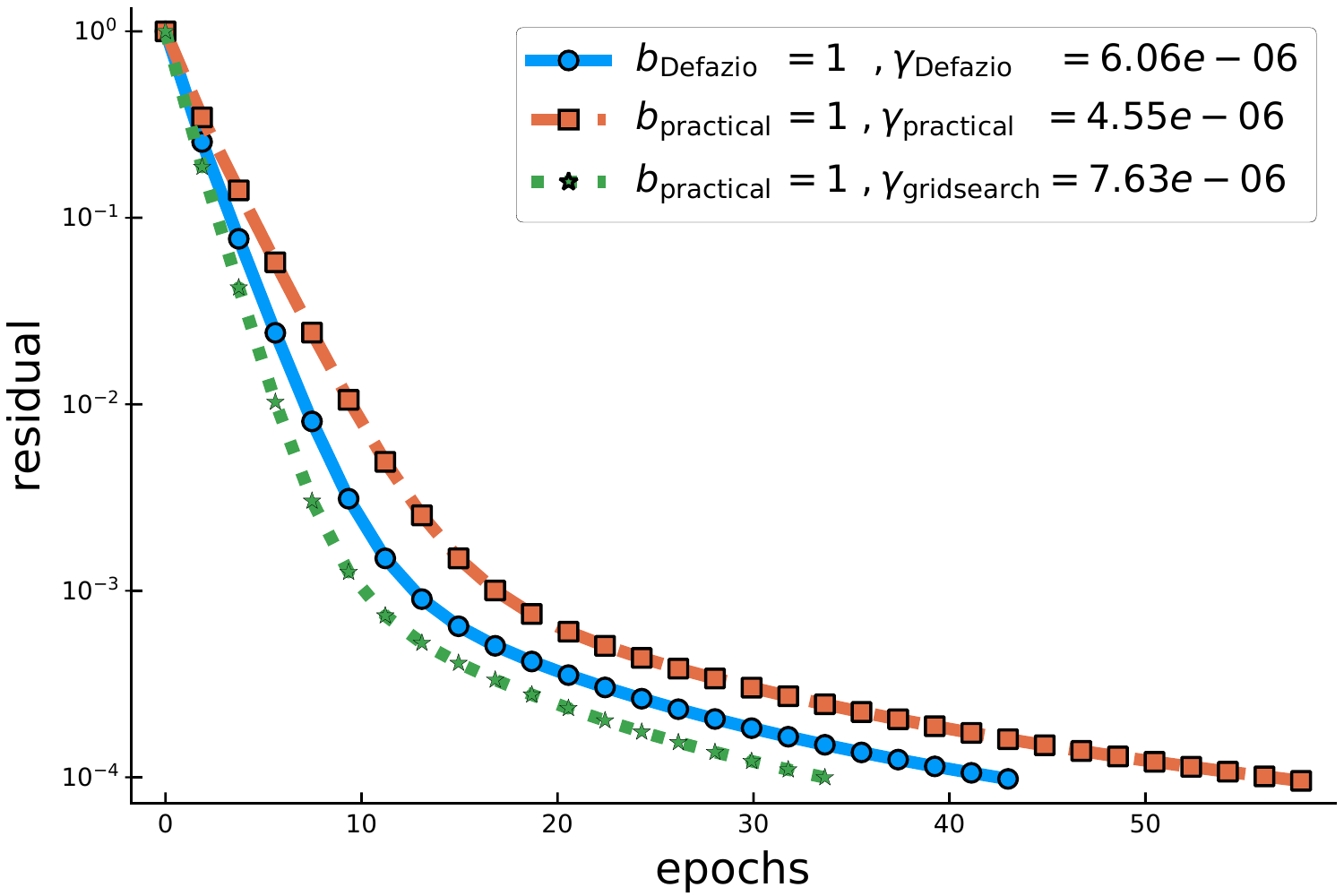}%
      \includegraphics[width=0.35\textwidth]{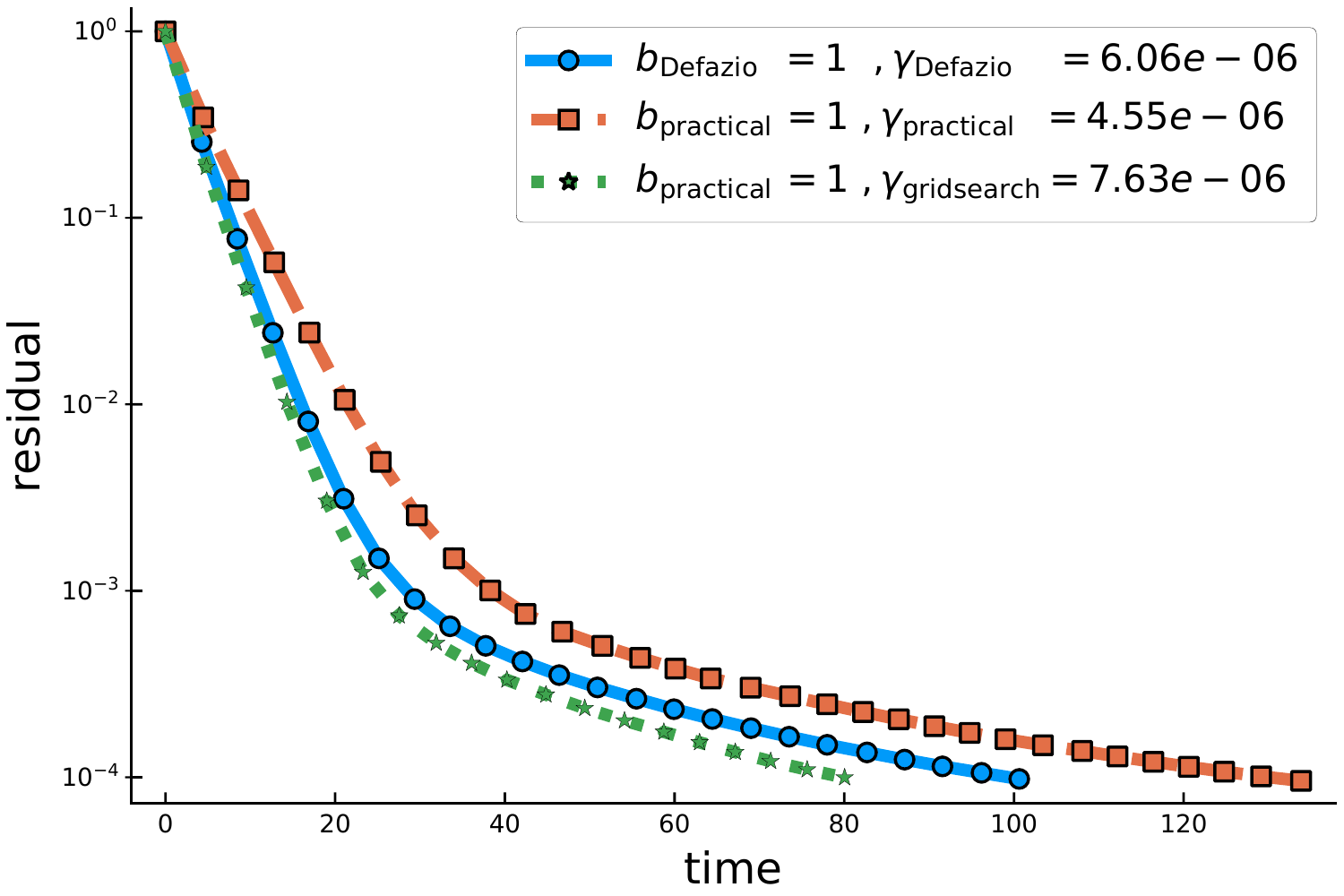}
      \caption{$\lambda = 10^{-3}$}
    \end{subfigure}
  \caption{Performance of SAGA implementations for the feature-scaled dataset \textit{slice}.}
  \label{fig:exp3_scaled_slice}
  \end{center}
  \vskip -0.2in
\end{figure}

\begin{figure}[!ht]
  \vskip 0.2in
  \begin{center}
    \begin{subfigure}[b]{\textwidth}
      \centering
      \includegraphics[width=0.35\textwidth]{exp3/slice/no_scaling/ridge_slice-none-regularizor-1e-01-exp3-epoc}%
      \includegraphics[width=0.35\textwidth]{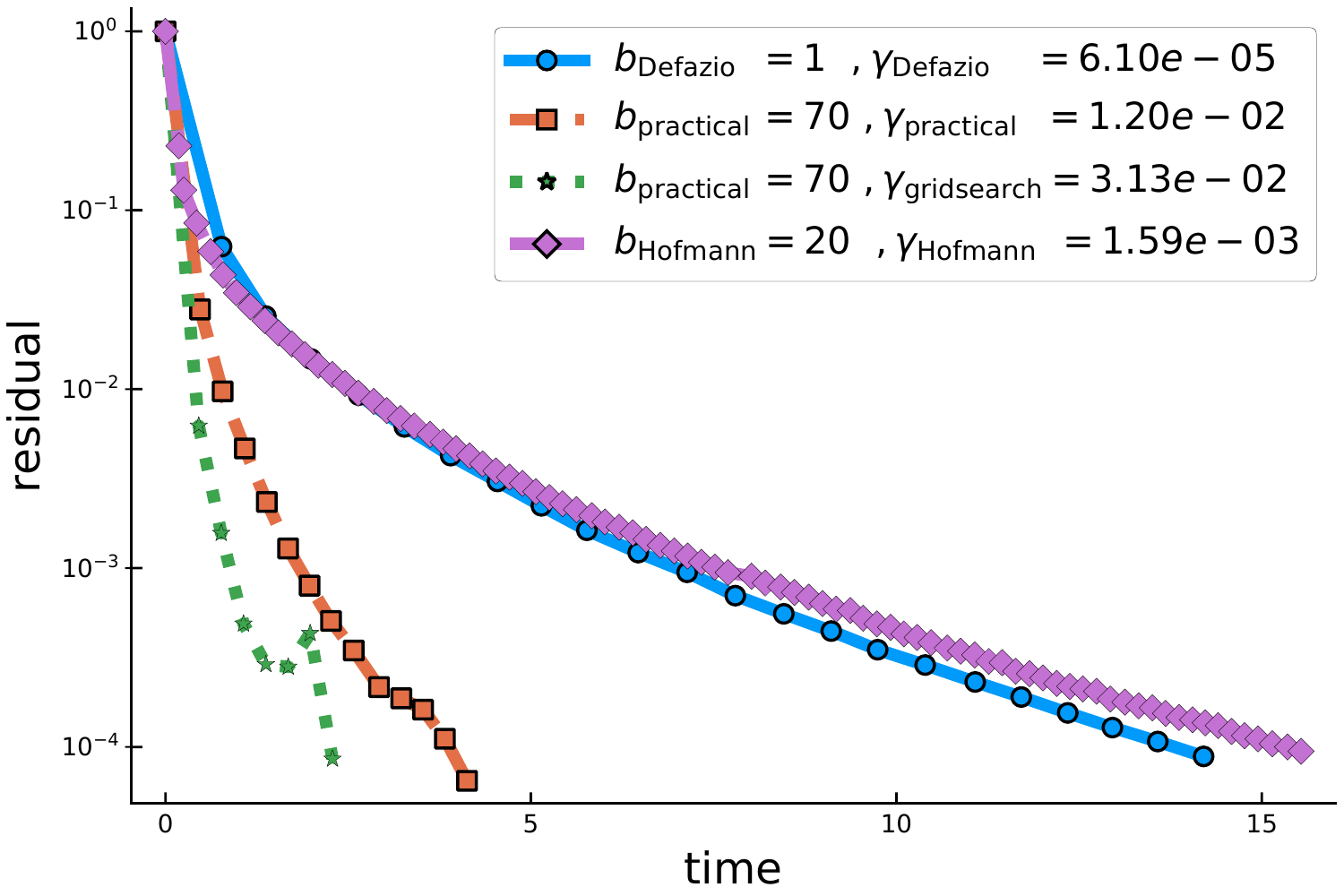}
      \caption{$\lambda = 10^{-1}$}
    \end{subfigure}\\
    \begin{subfigure}[b]{\textwidth}
      \centering
      \includegraphics[width=0.35\textwidth]{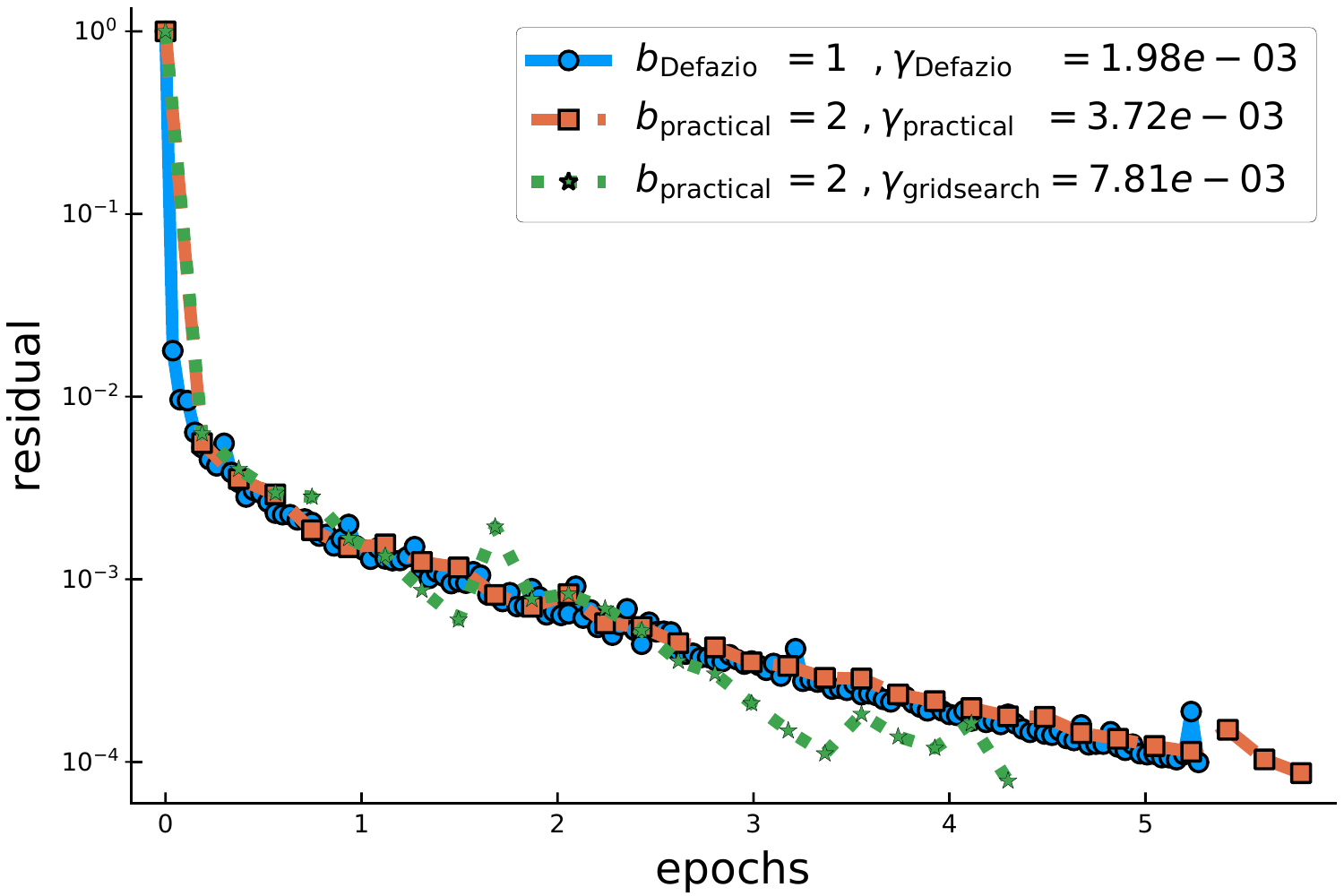}%
      \includegraphics[width=0.35\textwidth]{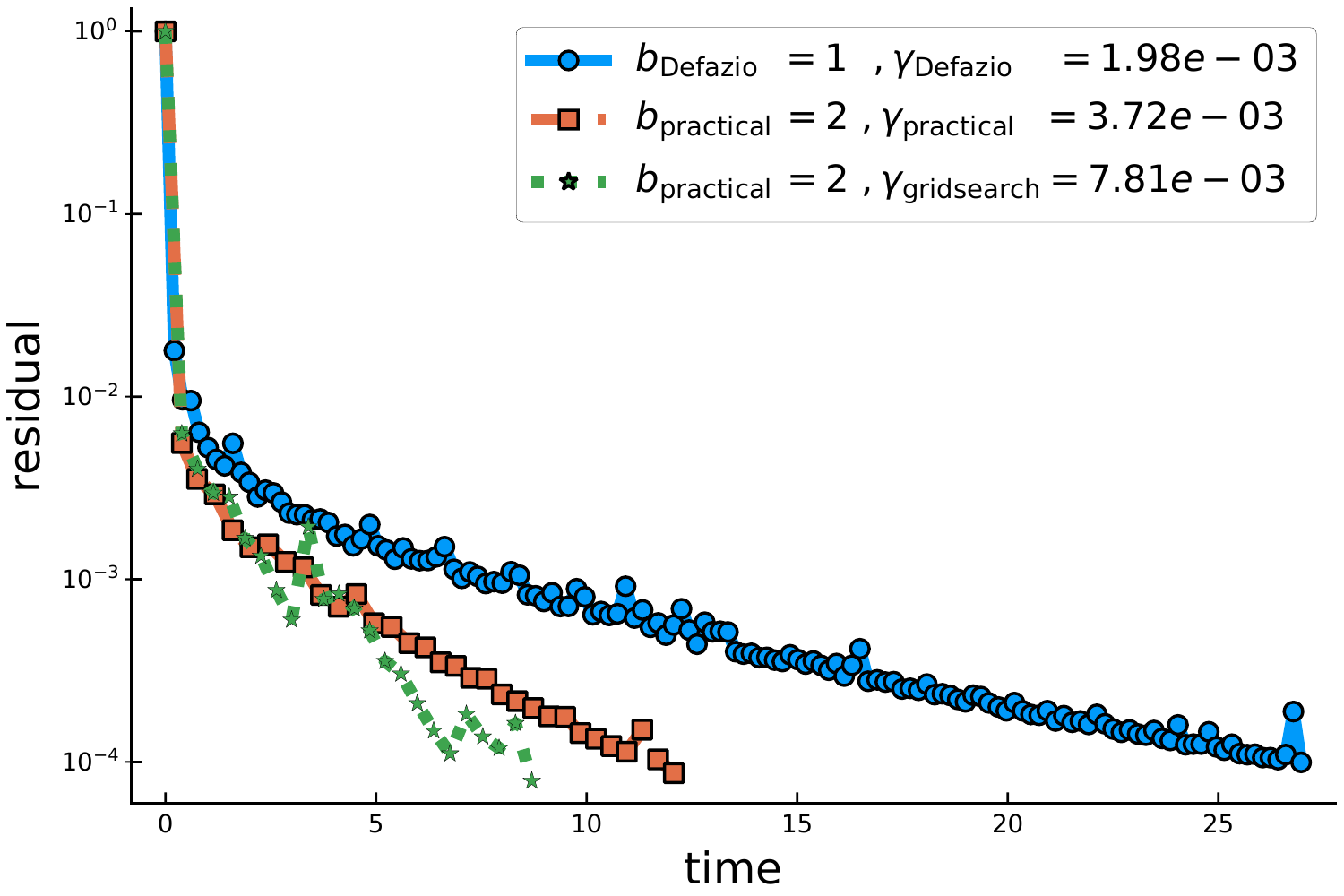}
      \caption{$\lambda = 10^{-3}$}
    \end{subfigure}
  \caption{Performance of SAGA implementations for the unscaled dataset \textit{slice}.}
  \label{fig:exp3_unscaled_slice}
  \end{center}
  \vskip -0.2in
\end{figure}

\begin{figure}[!ht]
  \vskip 0.2in
  \begin{center}
    \begin{subfigure}[b]{\textwidth}
      \centering
      \includegraphics[width=0.35\textwidth]{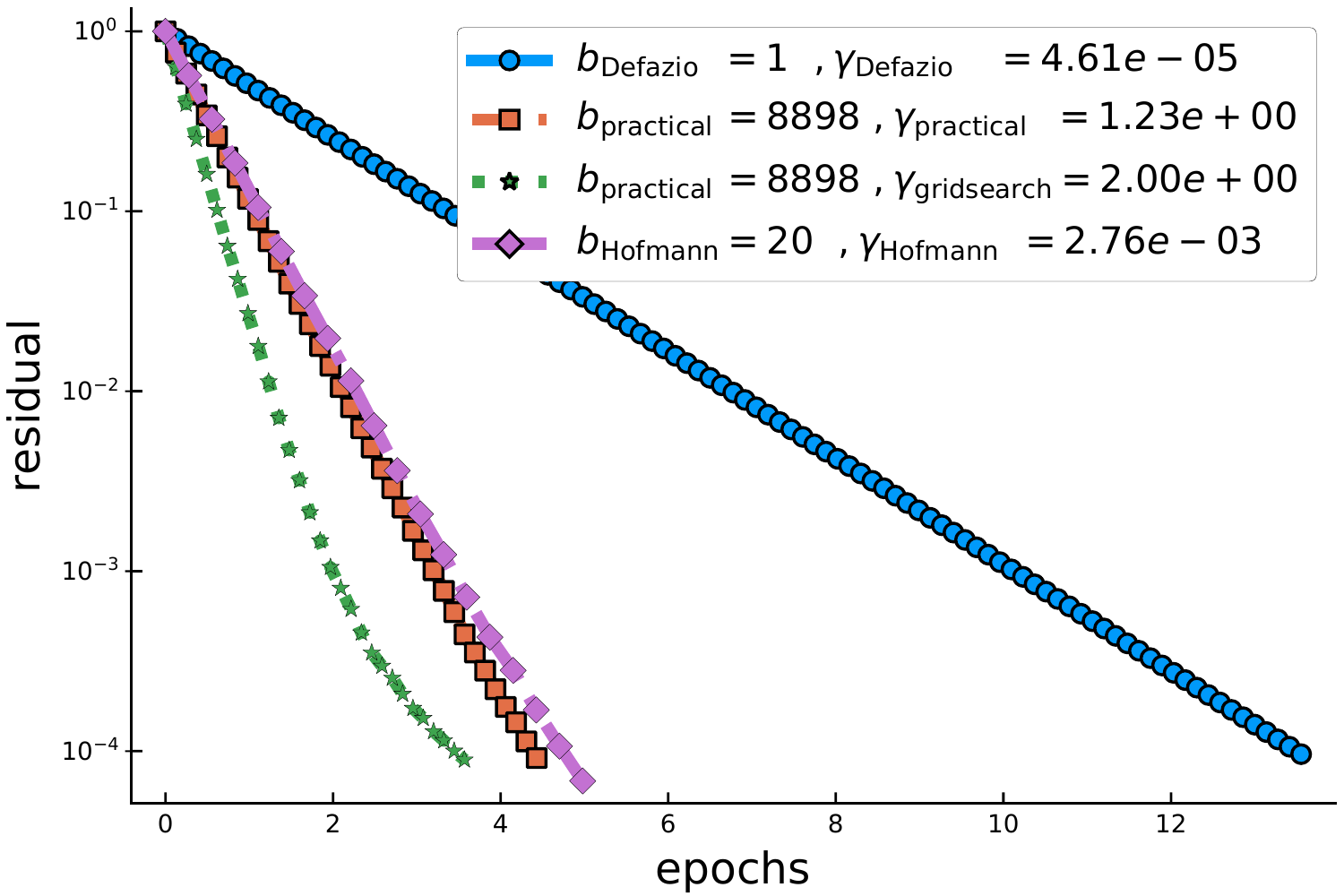}%
      \includegraphics[width=0.35\textwidth]{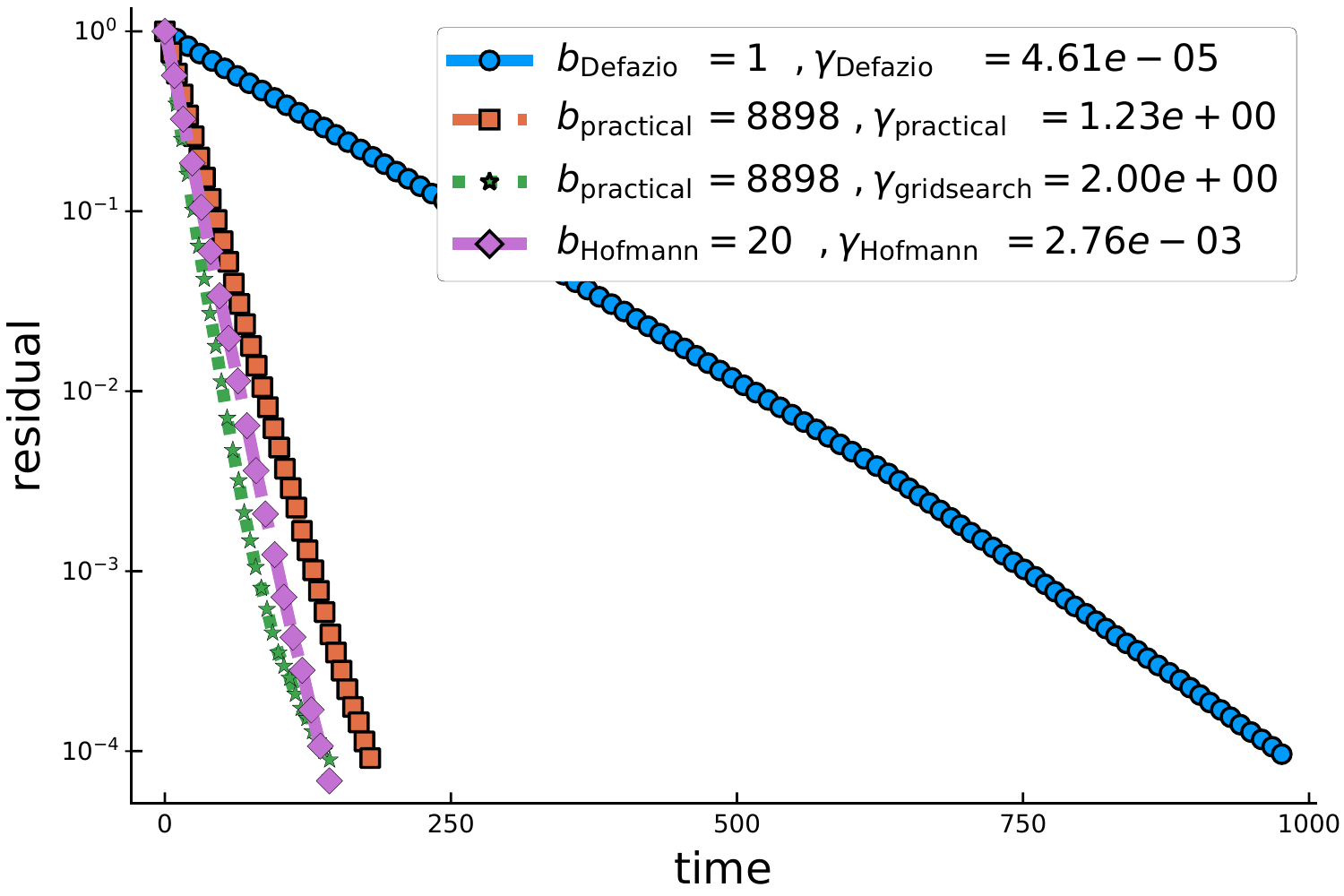}
      \caption{$\lambda = 10^{-1}$}
      \label{fig:exp3_unscaled_realsim_GD_better}
    \end{subfigure}\\
    \begin{subfigure}[b]{\textwidth}
      \centering
      \includegraphics[width=0.35\textwidth]{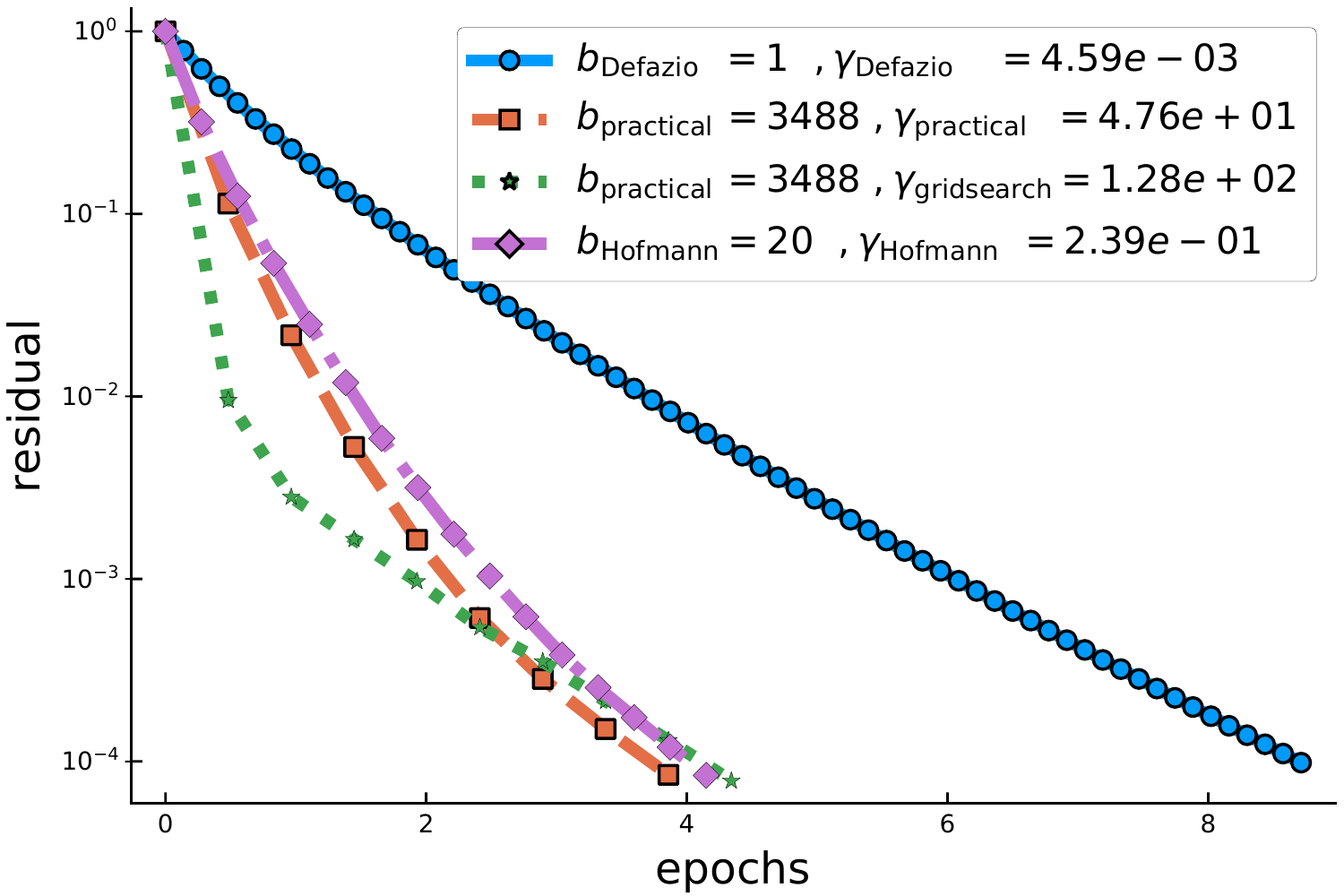}%
      \includegraphics[width=0.35\textwidth]{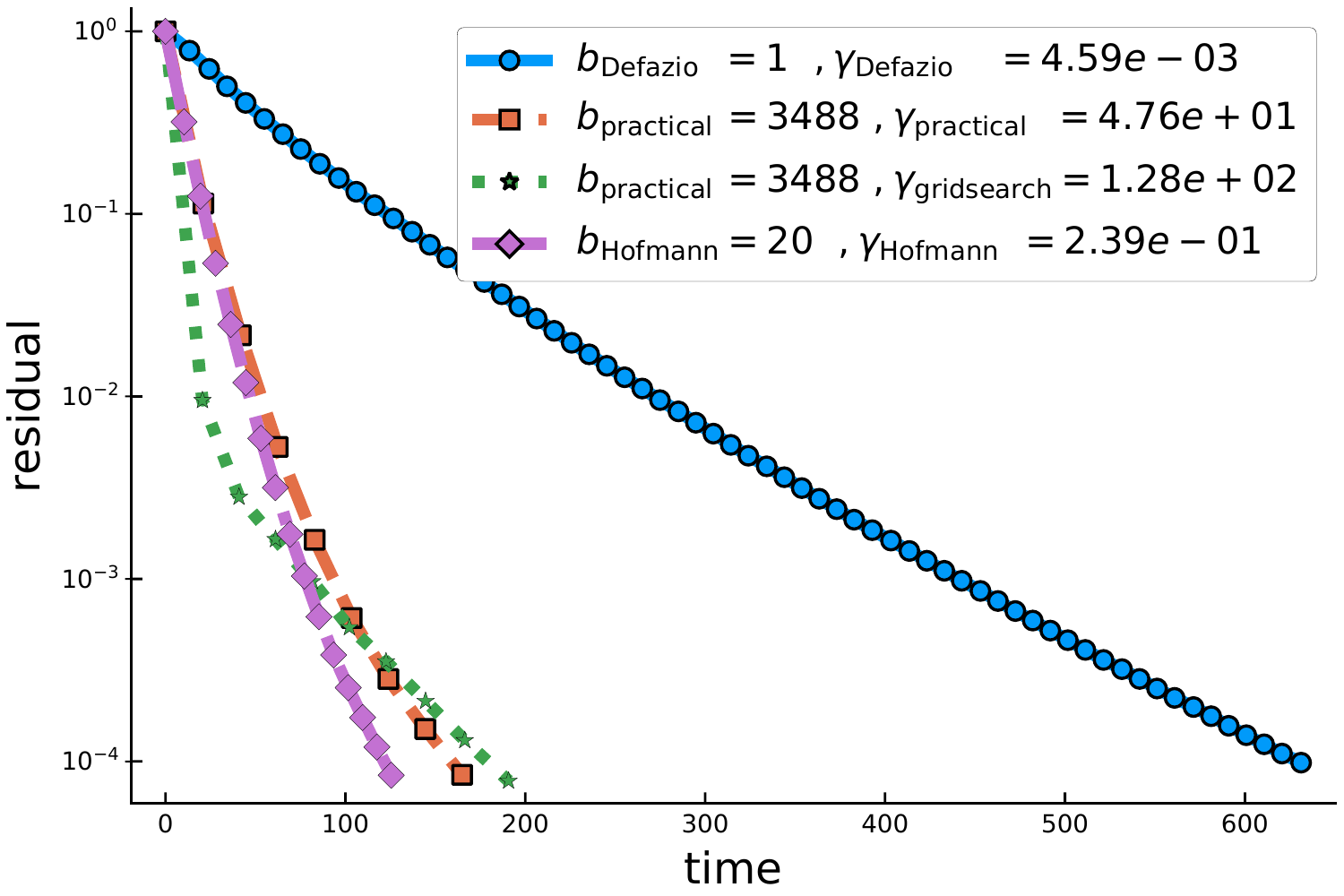}
      \caption{$\lambda = 10^{-3}$}
    \end{subfigure}
  \caption{Performance of SAGA implementations for the unscaled dataset \textit{real-sim}.}
  \label{fig:exp3_unscaled_realsim}
  \end{center}
  \vskip -0.2in
\end{figure}

\subsection{Experiment 4: optimality of the mini-batch size}
\label{appendix:exp_4}

This experiment aims to estimate how close is our practical estimate $b_{\text{practical}}$ to the empirical best mini-batch size one could get running a grid search. We recall that we use the following grid for the mini-batch sizes:
$\{2^i, i=0,\dots, 14\}$, with $2^{16}, 2^{18}$ and $n$ added in some cases. We show in the log-scaled \Cref{fig:exp4_scaled_ijcnn1,fig:exp4_scaled_covtype,fig:exp4_scaled_YearPredictionMSD,fig:exp4_scaled_slice,fig:exp4_unscaled_slice,fig:exp4_unscaled_real-sim} the empirical total complexity $K_{\mathrm{total}}$, \eg the number of computed stochastic gradients to reach a relative error of $10^{-4}$, as a function of the mini-batch $b$. For each mini-batch of the grid $b$, the step size used is the one corresponding to the \emph{practical estimate}, \ie when replacing $\cL(b)$ by $\cL_{\text{practical}}(b)$ in \Cref{eq:gammamaster}.

We always observe a change of regime in the empirical complexity. For small values of $b$, the complexity is of the same order of magnitude, then, for values greater than the empirical optimal mini-batch size, the complexity explodes. This experiment shows that our optimal mini-batch size $b_{\text{practical}}$ correctly designates the largest mini-batch achieving the best complexity as large as possible, without reaching the regime where the total complexity explodes.

\begin{figure}[!ht]
  \vskip 0.2in
  \begin{center}
    \begin{subfigure}[b]{0.4\textwidth}
      \includegraphics[width=\textwidth]{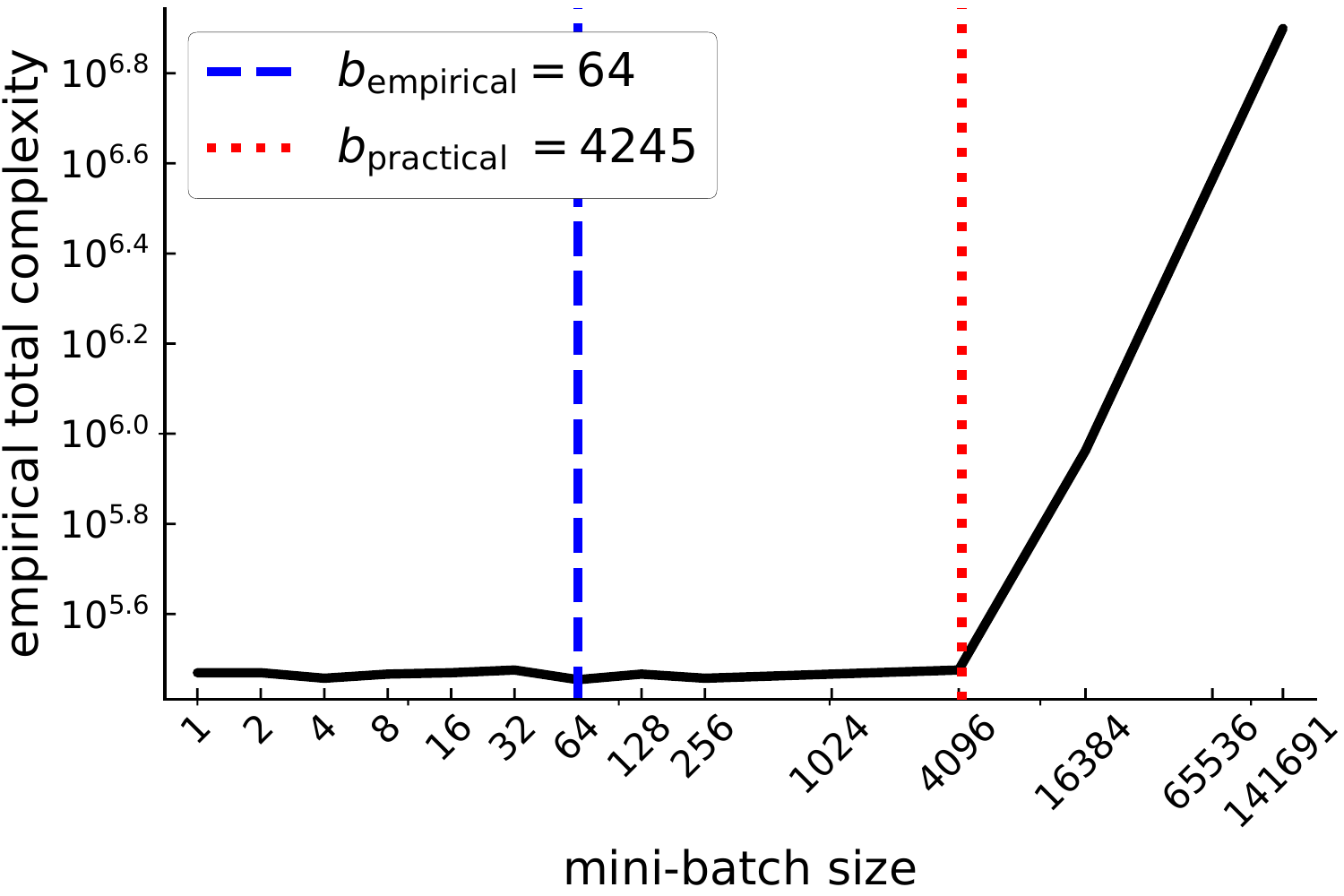}
      \caption{$\lambda = 10^{-1}$}
    \end{subfigure}%
    \begin{subfigure}[b]{0.4\textwidth}
      \includegraphics[width=\textwidth]{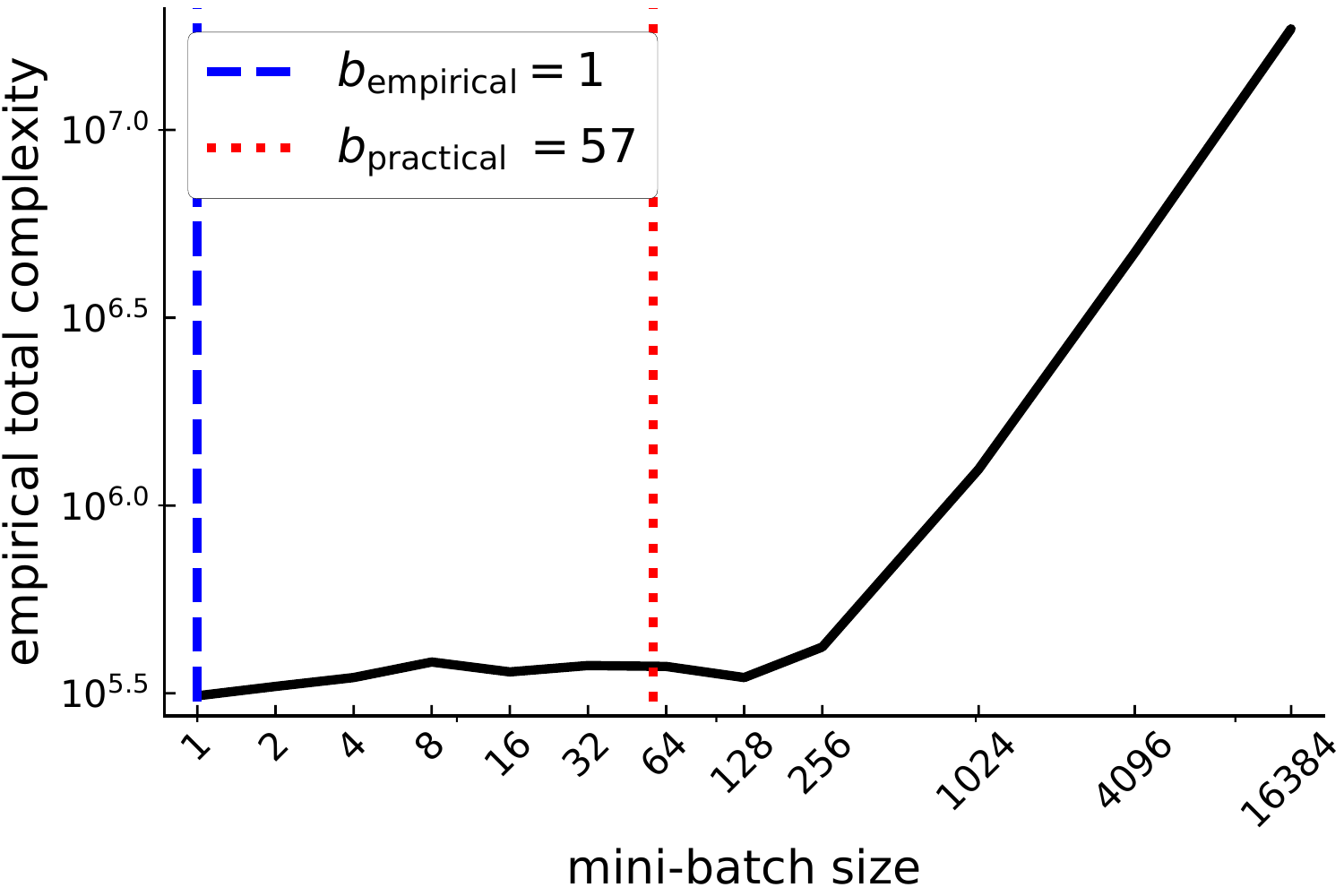}
      \caption{$\lambda = 10^{-3}$}
    \end{subfigure}
  \end{center}
  \caption{Empirical total complexity versus mini-batch size for the feature-scaled \emph{ijcnn1} dataset.}
  \label{fig:exp4_scaled_ijcnn1}
  \vskip -0.2in
\end{figure}

\begin{figure}[!ht]
  \vskip 0.2in
  \begin{center}
    \begin{subfigure}[b]{0.4\textwidth}
      \includegraphics[width=\textwidth]{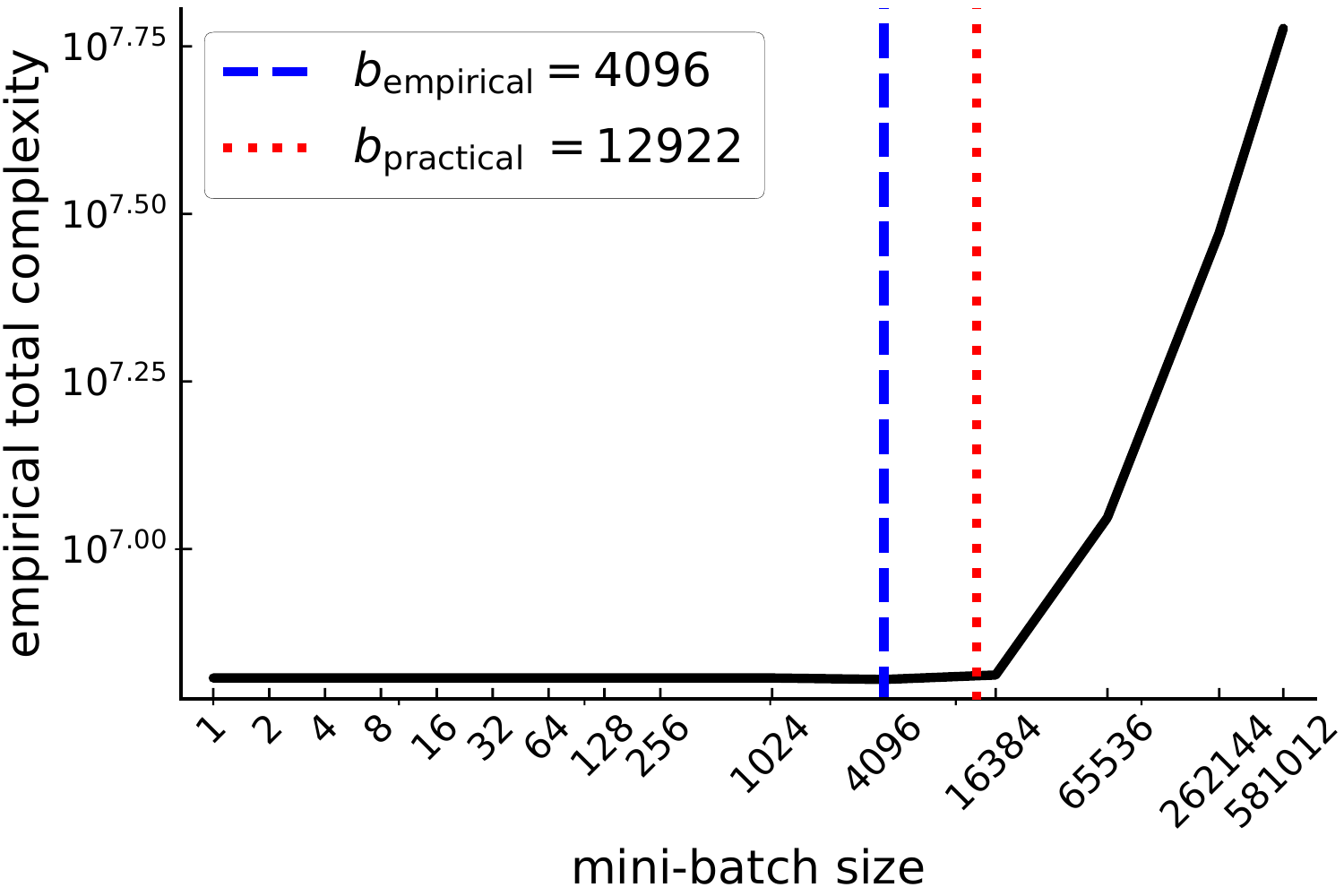}
      \caption{$\lambda = 10^{-1}$}
    \end{subfigure}%
    \begin{subfigure}[b]{0.4\textwidth}
      \includegraphics[width=\textwidth]{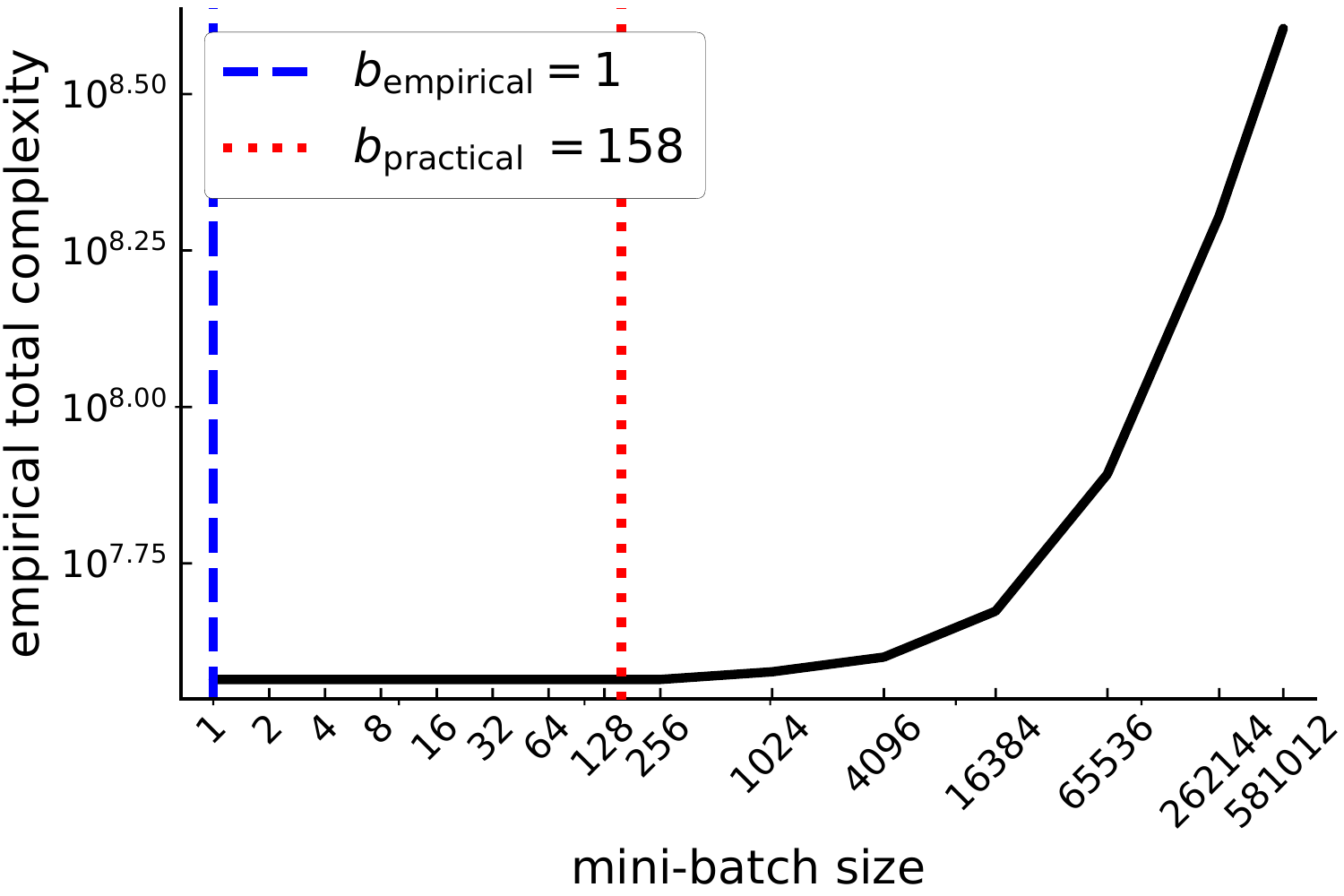}
      \caption{$\lambda = 10^{-3}$}
    \end{subfigure}
  \end{center}
  \caption{Empirical total complexity versus mini-batch size for the feature-scaled \emph{covtype.binary} dataset.}
  \label{fig:exp4_scaled_covtype}
  \vskip -0.2in
\end{figure}

\begin{figure}[!ht]
  \vskip 0.2in
  \begin{center}
    \begin{subfigure}[b]{0.4\textwidth}
      \includegraphics[width=\textwidth]{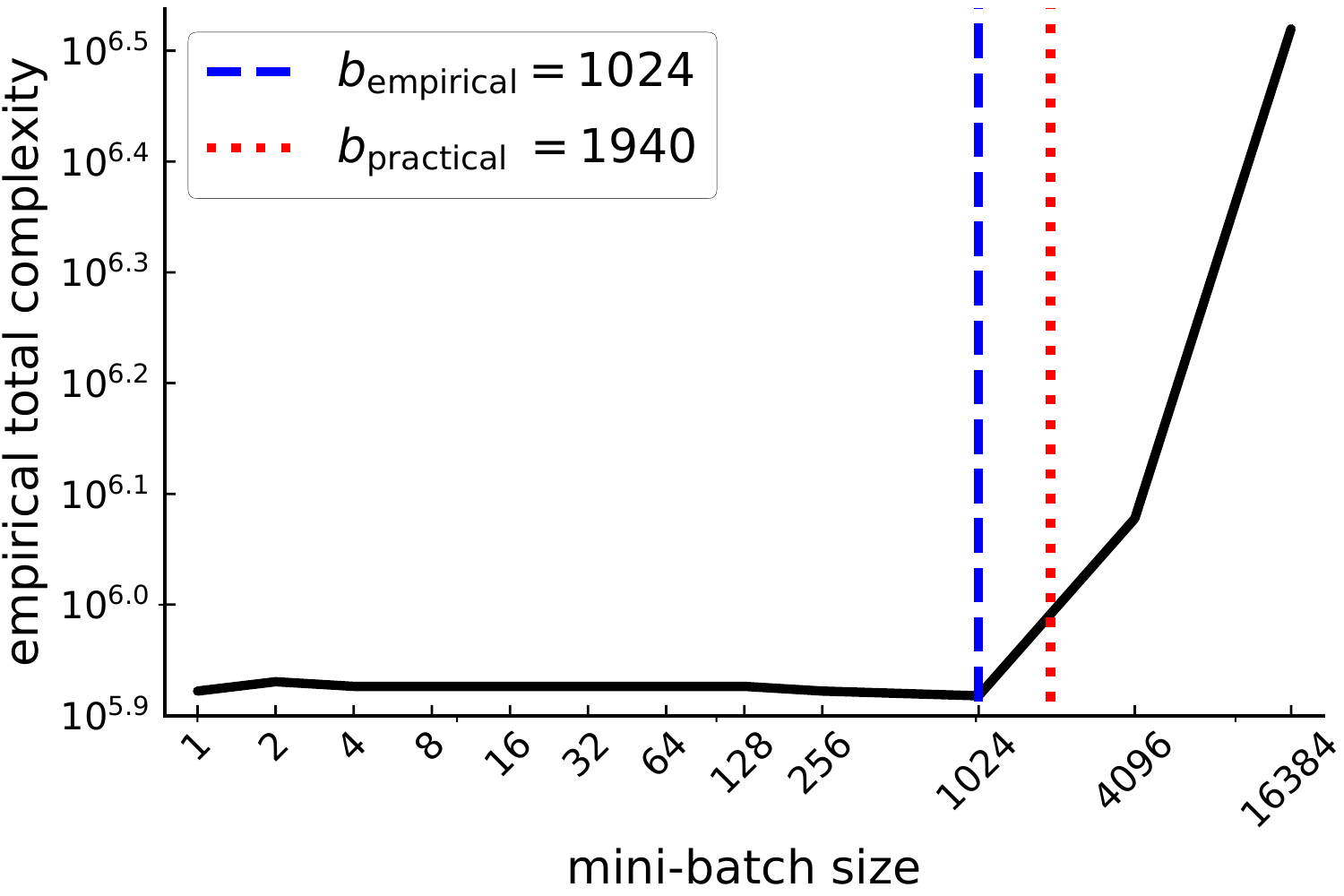}
      \caption{$\lambda = 10^{-1}$}
    \end{subfigure}%
    \begin{subfigure}[b]{0.4\textwidth}
      \includegraphics[width=\textwidth]{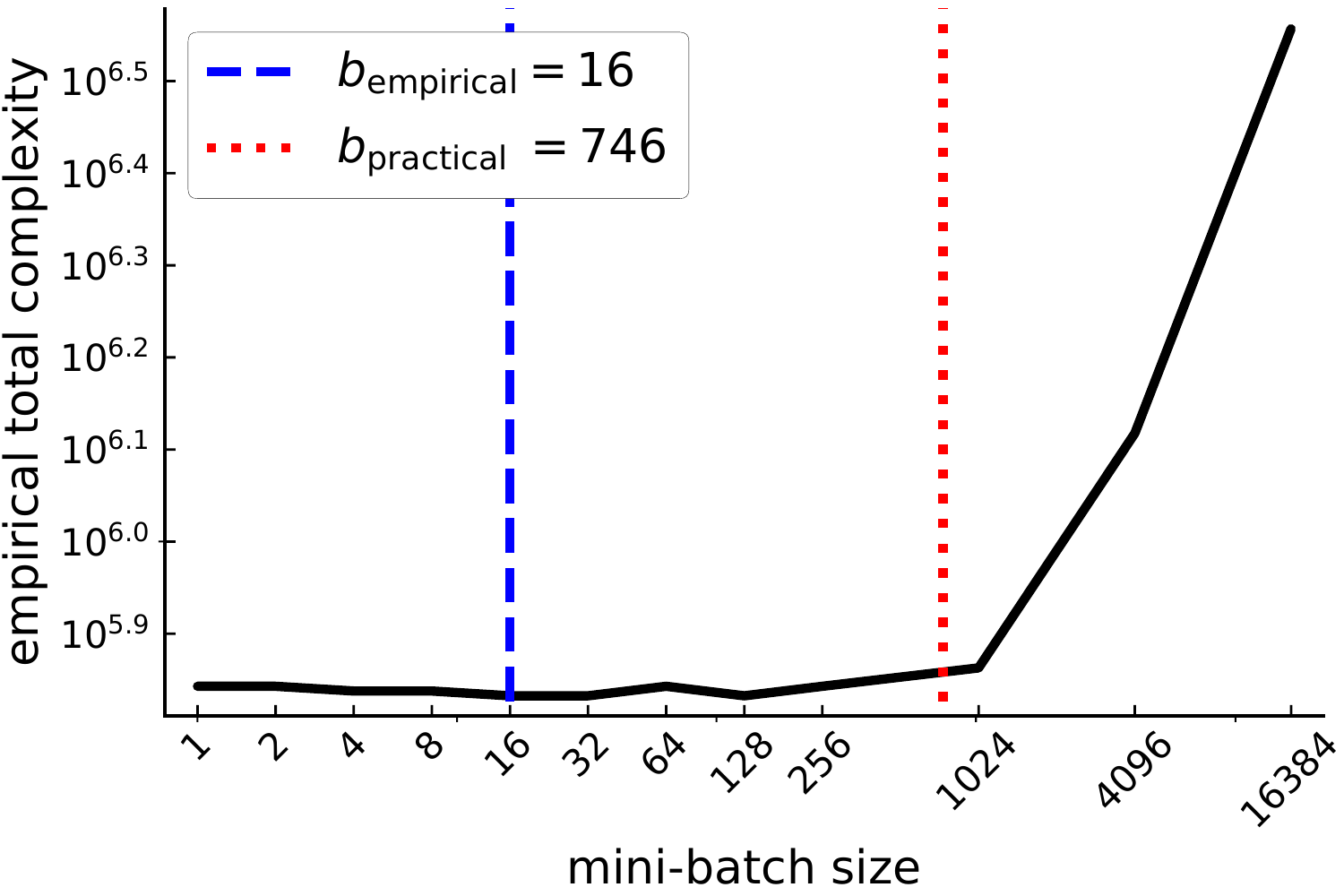}
      \caption{$\lambda = 10^{-3}$}
    \end{subfigure}
  \end{center}
  \caption{Empirical total complexity versus mini-batch size for the feature-scaled \emph{YearPredictionMSD} dataset.}
  \label{fig:exp4_scaled_YearPredictionMSD}
  \vskip -0.2in
\end{figure}

\begin{figure}[!ht]
  \vskip 0.2in
  \begin{center}
    \begin{subfigure}[b]{0.4\textwidth}
      \includegraphics[width=\textwidth]{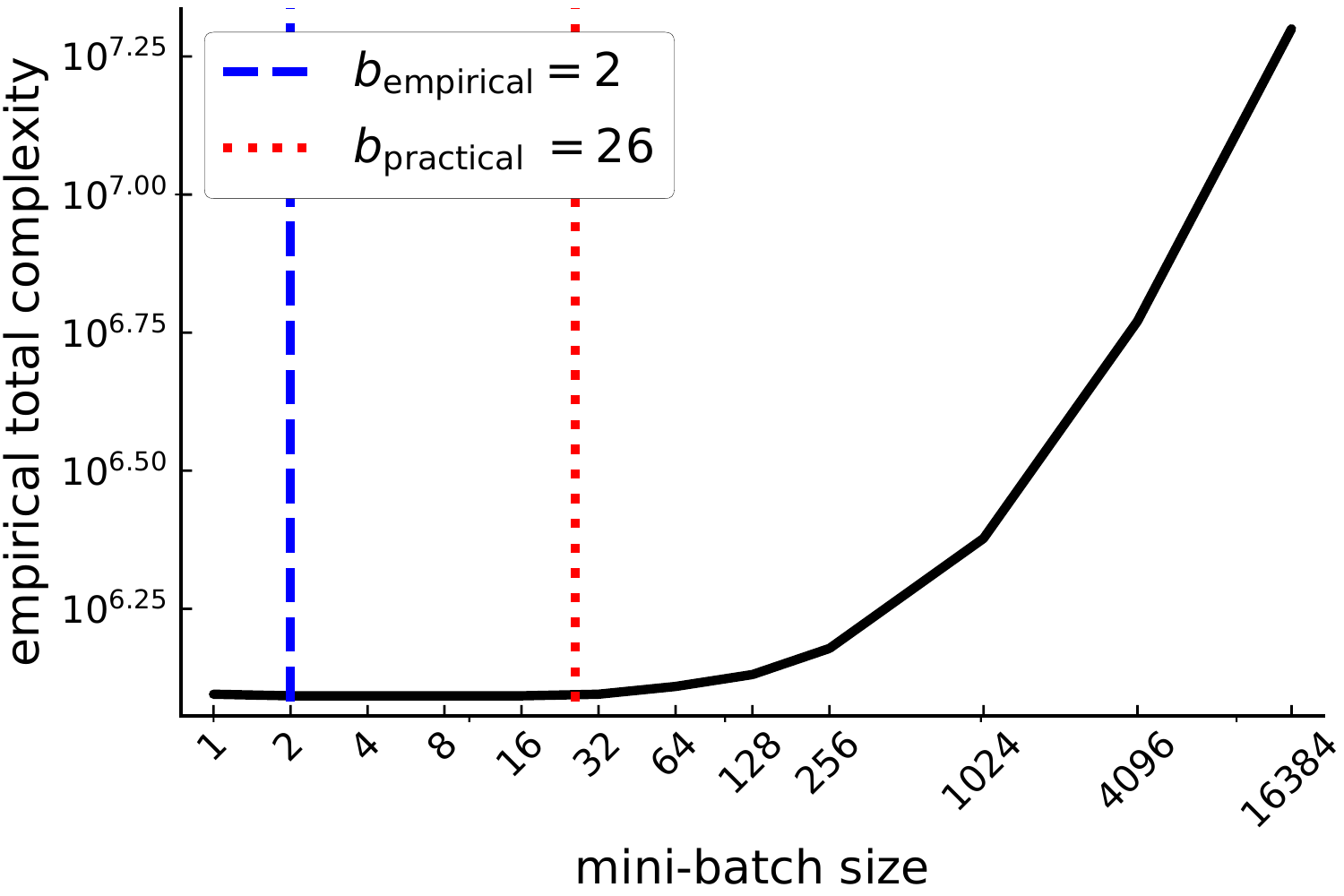}
      \caption{$\lambda = 10^{-1}$}
    \end{subfigure}%
    \begin{subfigure}[b]{0.4\textwidth}
      \includegraphics[width=\textwidth]{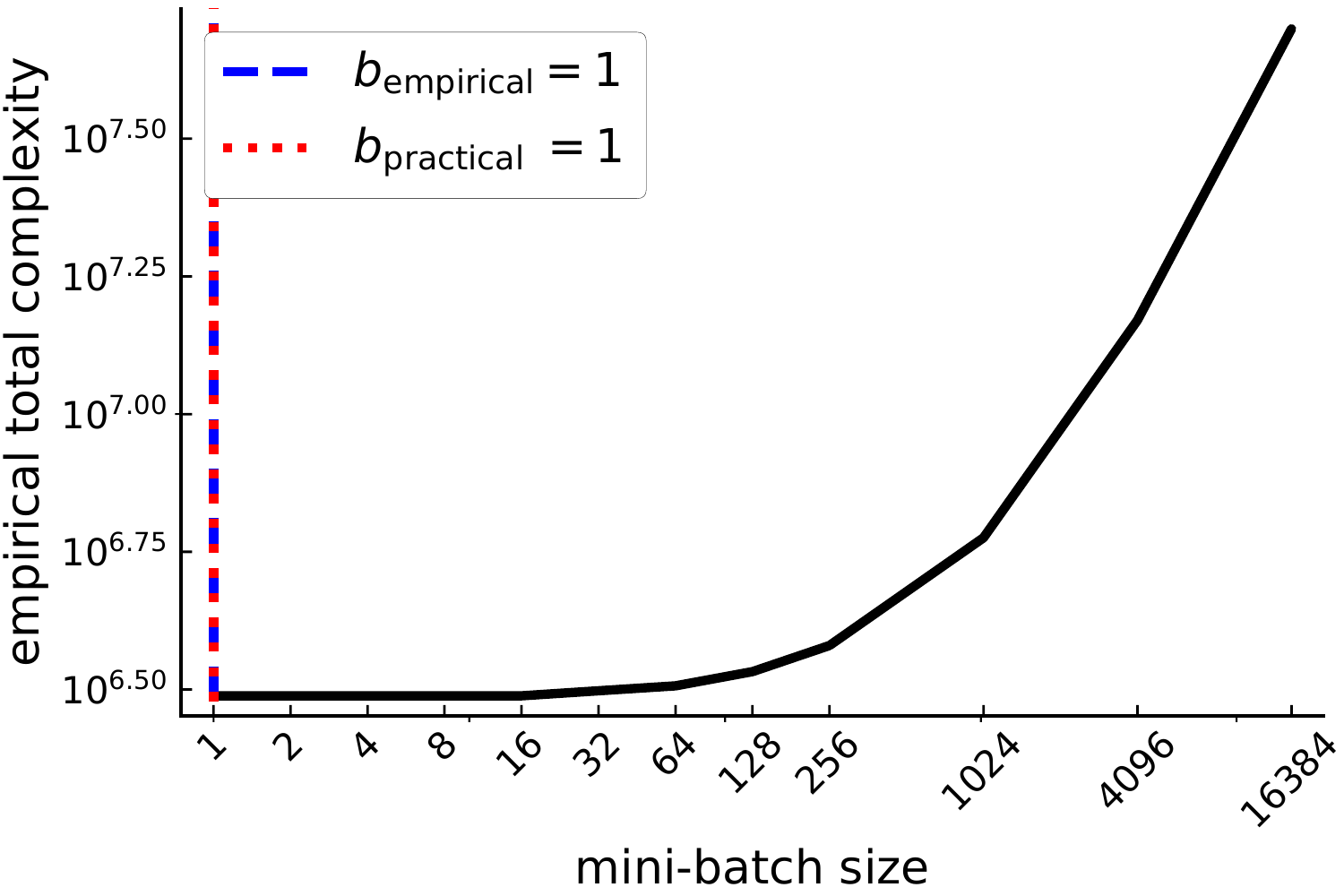}
      \caption{$\lambda = 10^{-3}$}
    \end{subfigure}
  \end{center}
  \caption{Empirical total complexity versus mini-batch size for the feature-scaled \emph{slice} dataset.}
  \label{fig:exp4_scaled_slice}
  \vskip -0.2in
\end{figure}

\begin{figure}[!ht]
  \vskip 0.2in
  \begin{center}
    \begin{subfigure}[b]{0.4\textwidth}
      \includegraphics[width=\textwidth]{exp4/ridge_slice-none-regularizor-1e-01-exp4-empcomplex-1-avg_skip_mult_0_1}
      \caption{$\lambda = 10^{-1}$}
    \end{subfigure}%
    \begin{subfigure}[b]{0.4\textwidth}
      \includegraphics[width=\textwidth]{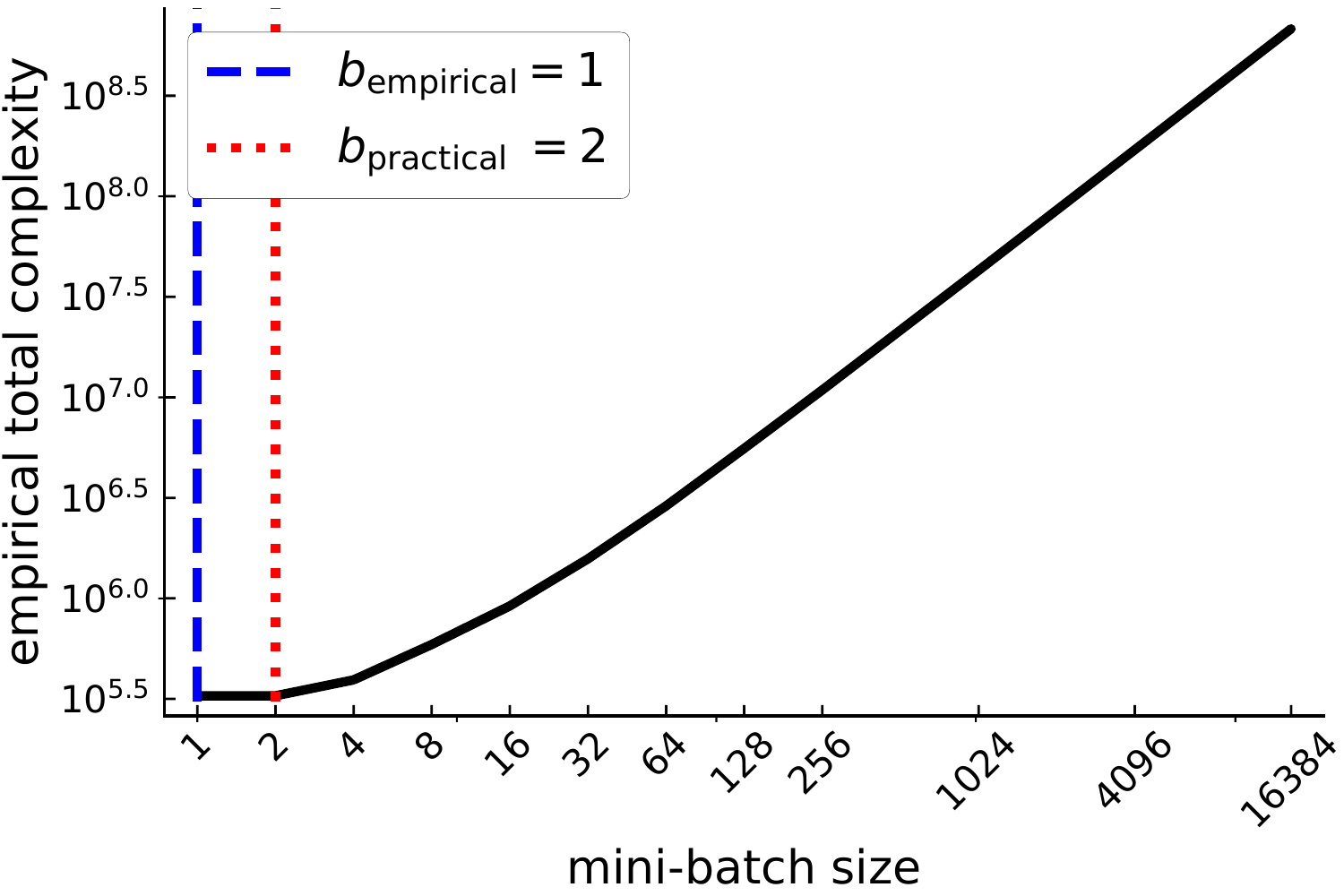}
      \caption{$\lambda = 10^{-3}$}
      \label{fig:exp4_saturation}
    \end{subfigure}
  \end{center}
  \caption{Empirical total complexity versus mini-batch size for the unscaled \emph{slice} dataset.}
  \label{fig:exp4_unscaled_slice}
  \vskip -0.2in
\end{figure}

\begin{figure}[!ht]
  \vskip 0.2in
  \begin{center}
    \begin{subfigure}[b]{0.4\textwidth}
      \includegraphics[width=\textwidth]{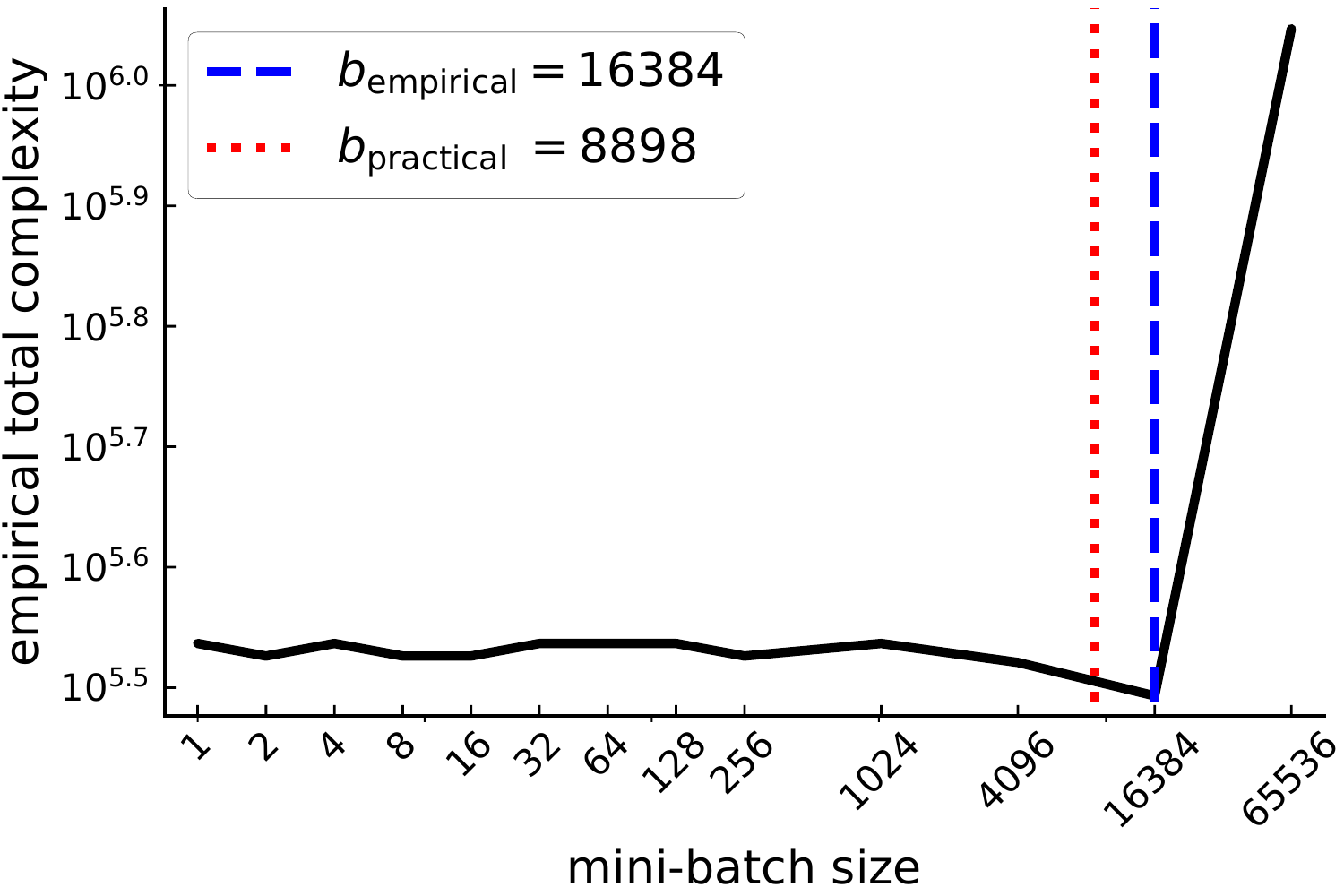}
      \caption{$\lambda = 10^{-1}$}
    \end{subfigure}%
    \begin{subfigure}[b]{0.4\textwidth}
      \includegraphics[width=\textwidth]{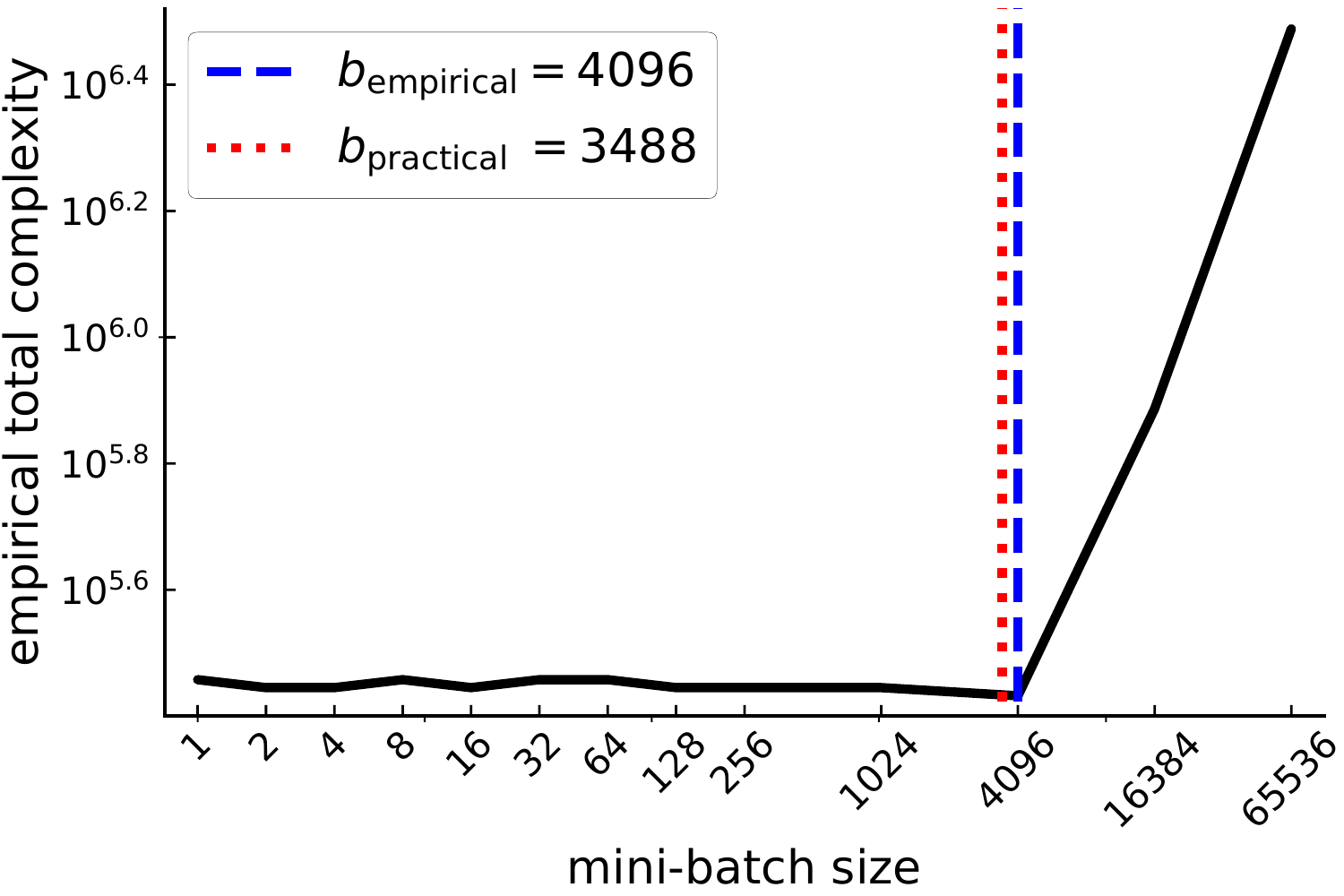}
      \caption{$\lambda = 10^{-3}$}
    \end{subfigure}
  \end{center}
  \caption{Empirical total complexity versus mini-batch size for the unscaled \emph{real-sim} dataset.}
  \label{fig:exp4_unscaled_real-sim}
  \vskip -0.2in
\end{figure}

\end{document}